\documentclass[12pt]{article}

\usepackage{jmlr2e}
\usepackage{float}
\usepackage{mathtools,bbm,amsmath,amsfonts,mathrsfs}
\usepackage{booktabs}
\usepackage{enumitem}  
\usepackage{xcolor}
\usepackage{relsize}

\makeatletter
\newcommand*{\eqrefplain}[1]{\textup{\tagform@{\ref*{#1}}}}
\makeatother
\newcommand{\ind}[1]{\mathbbm{1}\left\{#1\right\}}

\newcommand{\rdd}{\mathbb{R}^{d}}
\newcommand{\re}{\mathbb{R}}

\newcommand{\Prr}[1]{\Pr\left(#1\right)}
\newcommand{\limn}{\lim_{n\rightarrow\infty}}
\newcommand{\sam}[2]{\mathbb{#1}_{#2}}

\newcommand{\norm}[1]{\left\lVert#1\right\rVert}
\newcommand{\normm}[1]{\lVert#1\rVert}

\newcommand{\eqd}{\stackrel{\text{d}}{=}}

\newcommand{\sumn}{\sum_{i=1}^n}

\newcommand{\D}{{\rm D}}

\DeclareMathOperator{\HD}{HD}
\DeclareMathOperator{\IRW}{IRW}
\DeclareMathOperator{\IDD}{IDD}

\DeclareMathOperator{\SMD}{SMD}
\DeclareMathOperator{\SD}{SD}
\DeclareMathOperator{\MSD}{MSD}

\DeclareMathOperator{\sff}{SR}

\newcommand{\Ee}[2]{\mathbb{E}{#1}\left(#2\right)}
\newcommand{\E}[2]{\mathbb{E}_{#1}#2}
\DeclareMathOperator*{\argmax}{ {\rm argmax}}
\DeclareMathOperator*{\lap}{ {\rm Laplace}}

\DeclareMathOperator{\med}{MED}

\DeclareMathOperator{\mad}{MAD}
\newcommand{\expt}{\sigma}
\newcommand{\GS}{\mathrm{GS}}

\newcommand{\kr}[1]{{\color{black} #1 \color{black}}}

\newcommand{\eps}{\epsilon}
\newcommand{\cM}{\mathcal{M}}

\newcommand{\bR}{\mathbb{R}}
\newcommand{\cN}{\mathcal{N}}
\newcommand{\cF}{\sF}

\newcommand{\scA}{\mathscr{A}}
\newcommand{\sC}{\mathscr{C}}
\newcommand{\sN}{\mathscr{N}}
\newcommand{\sF}{\mathscr{F}}
\newcommand{\sG}{\mathscr{G}}
\newcommand{\sB}{\mathscr{B}}
\newcommand{\VC}{\operatorname{VC}}
\newcommand{\bX}{{\mathbb{X}}}

\newcommand{\bS}{\mathbb{S}}
\newcommand{\cube}{A}
\newcommand{\hmu}{{\hat{\mu}}}

\numberwithin{equation}{section}
\theoremstyle{remark}
\newtheorem{condition}{Condition}
\usepackage{lastpage}
\begin{document}

\title{Differentially private multivariate medians}

\author{\name Kelly Ramsay \email kramsay2@yorku.ca \\
       \addr Department of Mathematics and Statistics\\
       York University\\
       North York, ON M3J1P3, Canada
       \AND
       \name Aukosh Jagannath \email a.jagannath@uwaterloo.ca \\
       \addr Department of Statistics and Actuarial Science\\
       University of Waterloo\\
       Waterloo, ON N2L3G1, Canada
        \AND
       \name Shoja'eddin Chenouri \email schenouri@uwaterloo.ca \\
       \addr Department of Statistics and Actuarial Science\\
       University of Waterloo\\
       Waterloo, ON N2L3G1, Canada}

% \editor{}

\maketitle

\begin{abstract}%   <- trailing '%' for backward compatibility of .sty file
Statistical tools which satisfy rigorous privacy guarantees are necessary for modern data analysis. 
It is well-known that robustness against contamination is linked to differential privacy. 
Despite this fact, using multivariate medians for differentially private and robust multivariate location estimation has not been systematically studied.  
We develop novel finite-sample performance guarantees for differentially private multivariate depth-based medians, which are essentially sharp. 
Our results cover commonly used depth functions, such as the halfspace (or Tukey) depth, spatial depth, and the integrated dual depth. 
We show that under Cauchy marginals, the cost of heavy-tailed location estimation outweighs the cost of privacy. 
We demonstrate our results numerically using a Gaussian contamination model in dimensions up to $d = 100$, and compare them to a state-of-the-art private mean estimation algorithm. 
As a by-product of our investigation, we prove concentration inequalities for the output of the exponential mechanism about the maximizer of the population objective function. 
This bound applies to objective functions that satisfy a mild regularity condition.
\end{abstract}


\section{Introduction}
Protecting user privacy is necessary for a safe and fair society. 
In order to protect user privacy, many large institutions, including the United States Census Bureau \citep{Abowd20222020}, Apple \citep{Apple}, and Google \citep{Google}, utilize a strong notion of privacy known as differential privacy. 
Differential privacy is favored because it protects against many adversarial attacks while requiring minimal assumptions on both the data and the adversary \citep{Dwork2017}. 
It is now well-known that differential privacy is linked to robustness against contamination, see \citep{Dwork2009, Medina2020, 2021Liub} and the references therein. 
Surprisingly, however, multivariate medians have not been systematically studied in the differential privacy literature. 
We study this problem here. 
Note that there are several variants of differential privacy. 
Here, we focus on the setting of pure differential privacy, which is the strictest variant of differential privacy.

To date, the literature on differentially private location estimation has largely focused on means. 
\citet{Barber2014} proved a lower bound on the minimax risk of differentially private mean estimation in terms of the number of moments possessed by the population measure. 
Other early works focused on the univariate setting  \citep{Karwa2017, Bun2019}. 
\citet{Kamath2018} and \citet{Bun2019b} studied differentially private mean estimation for several parametric models, including the multivariate Gaussian model. 
Later, several authors studied differentially private mean estimation for sub-Gaussian measures \citep{Cai2019,Biswas2020,Brown2021,2021Liu}. 
After which, \citet{Kamath2020,2021Liub} and \citet{Hopkins2021} studied differentially private mean estimation for measures with a finite second moment. 
Within this framework, \citet{Kamath2020} considered heavy tails and \citet{2021Liub} and \citet{Hopkins2021} considered robustness in terms of contamination models. 
\citet{Kamath2020} quantified the sample complexity of their (purely) differentially private mean estimator in terms of the number of moments possessed by the population measure. 
\citet{2021Liub} quantified the relationship between the accuracy of an (approximately) differentially private mean estimator and the level of contamination in the observed sample. 
On a similar note, \citet{Hopkins2021} showed that there exists a level of contamination under which the sample complexity of their (purely) differentially private mean estimator remains unchanged.

On the other hand, if one wishes to estimate location robustly, the canonical estimators are medians. 
The differential privacy literature on medians has focused on the univariate setting:
\citet{Dwork2009} introduced an approximately differentially private median estimator with asymptotic consistency guarantees, 
\citet{Avella-Medina2019a, Brunel2020} then introduced several median estimators which achieve sub-Gaussian error rates, and 
\citet{zamos2020} introduced median estimators with minimax optimal sample complexity. 
Private estimation of multivariate medians, however, has not been carefully studied. 

In this paper, we study differentially private, multivariate median estimation through the framework of depth functions. 
Depth-based median estimation is the standard approach to multivariate median estimation. 
Here, one has a function, called the depth function, which provides a measure of centrality. 
Its maximizer is then the median, which is generally robust. Popular depth-based medians include the halfspace (or Tukey) median \citep{Tukey1974}, the spatial median \citep{Vardi2000}, and the simplicial median \citep{liu1990}. For more on the motivation of the depth-based medians approach, see Section~\ref{sec::rob-med-est} below and \cite{Small1990,Vardi2000, Serfling2006}. Our results apply to a broad class of depth functions, including those listed above, and we make minimal distributional assumptions. In particular, we do not require concavity of the depth function or that the population measure has moments of any order. 

Before turning to our main results, it is important to note that several authors have recently used depth functions in the context of differential privacy, likely due to their robustness properties, with a heavy focus on the halfspace depth. 
Depth functions have been used in settings such as finding a point in a convex hull \citep{Beimel2019,Gao2020}, comparing $k$-norm mechanisms \citep{Awan2021}, and improving robustness for mean estimation for Gaussian models \citep{Brown2021,2021Liub}. 
Another related work is that of \citet{beneliezer2022archimedes}, who focus on various problems concerning high-dimensional quantile estimation. 
In particular, one problem \citet{beneliezer2022archimedes} consider is that of privately obtaining one or more points which have halfspace depth above a given threshold. 
Though this problem is closely related, it is distinct. 
For instance, their algorithm can be used to estimate a point with high halfspace depth. 
However, it could not be used to directly estimate the halfspace median, since that would require setting the threshold to be close to the depth of the empirical halfspace median, which depends on the sample. 
To our knowledge, none of the aforementioned works have directly addressed depth-based median estimation. After completion of this work, \citet{ramsay2024} consider approximately differentially private projection-depth-based medians. Though many popular depth-based medians are covered by our framework, projection-depth-based medians are not covered by it. As a result, their work can be thought of as complementary to this one. 

\subsection{Our contributions}
% [Make this an itemized list.]
\kr{Our contributions are as follows:
\begin{itemize}
    \item Our main result is a general finite-sample deviations bound for private multivariate medians based on the exponential mechanism (Theorem~\ref{thm::main-dev-bound}). This deviations bound applies to private versions of many popular medians arising from depth functions, such as the halfspace median \citep{Tukey1974}, the simplicial median \citep{Liu1988} and the spatial median \citep{Vardi2000,Serfling2002}. 
    \item We use Theorem~\ref{thm::main-dev-bound} to give upper bounds on the deviations (and sample complexities, see Corollary~\ref{cor::main-sc-cor}) of several, purely differentially private multivariate medians arising from depth functions. These bounds are essentially sharp given recent results of \citet{Kamath2018,Cai2019} and do not require any moment assumptions on the population measure. 
    \item We provide a fast implementation of a smoothed version of the integrated dual depth-based median \citep{Cuevas2009}; we can compute the (non-private) median of ten thousand 100-dimensional samples in less than one second on a personal computer.  We show that this smoothed version of the integrated dual depth satisfies desirable properties for a depth function and can be approximated in polynomial time with finite-sample performance guarantees, see Proposition \ref{prop::app_dep}. 
    \item As a by-product of our analysis, we uncover a general but elementary concentration bound (Theorem~\ref{thm::main_result}) for the exponential mechanism. We give a novel regularity condition ``$(K,\sF)$-regularity'' (Condition \ref{cond::phi-k-reg}) on the objective function. This regularity condition implies both an upper bound on the sample complexity of a draw from the exponential mechanism and an upper bound on the global sensitivity of the objective function. 
    \item As a second by-product of our analysis, we uncover uniform non-asymptotic error bounds for several depth functions and a differentially private estimator of data depth values.
\end{itemize}}

\section{Main results}\label{sec::rob-med-est}
%%%%%%%%%%%%%%%%%%%%%%%%%%%%%%%%%%%%%%%%%%%%%%%%%%%%%%%%%%%%%%%%%%%%%%%%%%%%%%%%%%%%

\kr{In this section, we provide a general finite-sample deviation bound for differentially private multivariate median estimation. 
As the tools we use here combine ideas from two distinct literatures, namely the differential privacy literature and the robust statistics literature, we briefly recall some essential notions from each of these fields before stating our main results. For a more in-depth introduction to differential privacy see \citet{Dwork2014}, and for a more in-depth introduction to robust multivariate median estimation and the depth function framework see \citet{Mosler2002,Small1990,Liu_Serfling_Souvaine_2006}.  
We define all mathematical objects throughout the paper as they are used, but it is also helpful to have them collected in one place. Therefore, we also provide a summary of the notation we will use throughout the paper in Appendix~\ref{app::notation}. }

\subsection{Differential privacy}
Let us start first by recalling some essential notions from differential privacy.
% Let us begin by briefly recalling the following essential notions and definitions from differential privacy. For a textbook introduction to the concept of differential privacy see \citet{Dwork2014}.
A \emph{dataset} of size $n\times d$ is a collection of $n$ points in $\rdd$, $\mathbb{X}_n=(X_\ell)_{\ell=1}^n$, with repetitions allowed.
Let $\mathbf{D}_{n\times d}$ be the set of all datasets of size $n\times d$. 
For a dataset $\mathbb{Y}_n$, let
$\mathcal{D}(\mathbb{Y}_n,m)=\left \{ \mathbb{Z}_n\in \mathbf{D}_{n\times d} \colon  |\mathbb{Z}_n \triangle \mathbb{Y}_n|=2m \right\},$
that is, $\mathcal{D}(\mathbb{Y}_n,m)$ is the collection of datasets of size $n\times d$ which differ from $\mathbb{Y}_n$ in exactly $m$ points. 
Two datasets $\mathbb{Y}_n$ and $\mathbb{Z}_n$ are said to be \emph{adjacent} if $\mathbb{Z}_n\in \mathcal{D}(\mathbb{Y}_n,1)$. 
Finally, for a dataset $\mathbb{Y}_n$ with empirical measure $\hmu_{\mathbb{Y}_n}$ define
$$\widetilde{\cM}(\mathbb{Y}_n)=\{\tilde{\mu}\in\cM_1(\rdd)\colon \tilde{\mu}=\hmu_{\mathbb{Z}_n} \text{ for some } \mathbb{Z}_n\in \mathcal{D}(\mathbb{Y}_n,1) \},$$
where, for a space $S$, $\cM_1(S)$ denotes the space of probability measures on $S$.

Let us now recall the notion of differential privacy. 
Suppose that we observe a dataset $\bX_n$ comprised of $n$ i.i.d.\ samples from some unknown measure $\mu\in\cM_1(\bR^d)$ and produce a statistic, $\tilde{\theta}_n$, whose law conditionally on the samples depends on their empirical measure, $\hat{\mu}_n$.  
Let us denote this law by $P_{\hat{\mu}_n}$. The goal of differential privacy is to produce a statistic such that $P_{\hmu_n}$ is non-degenerate and, more precisely, obeys a certain ``privacy guarantee'' which is defined as follows.
Following the privacy literature, we call the algorithm that takes $\hmu_n$ and returns  the (random) statistic, $\tilde{\theta}_n$, a \emph{mechanism}. 
We also follow the standard abuse of terminology and call $\tilde{\theta}_n$ the mechanism when it is clear from context.

\begin{definition}\label{def::dp}
A mechanism $\tilde{\theta}_n$ is $\epsilon$-differentially private if for any dataset $\mathbb{Y}_n$, any $\tilde{\mu}\in \widetilde{\cM}(\mathbb{Y}_n)$, and any measurable set $B$, it holds that
\begin{equation}    \label{eqn::dp}
    P_{\hat{\mu}_{\mathbb{Y}_n}}( B)\leq e^\epsilon P_{\tilde{\mu}}( B).
\end{equation} 
\end{definition}
\noindent Here, $\eps>0$ is the \emph{privacy parameter}, for which smaller values enforce stricter levels of privacy. 
Many general purpose differentially private mechanisms rely on the concept of \emph{global sensitivity}. 
The global sensitivity $\GS_n$ of a function $\phi:\rdd\times\cM_1(\rdd)\to\re$ is given by
\[
\GS_n(\phi)=\sup_{\substack{\sam{X}{n}\in \mathbf{D}_{n\times d},\\ \tilde{\mu}_n\in \widetilde{\cM}(\sam{X}{n})}}\sup_{x\in\rdd}|\phi(x,\hat{\mu}_n)-\phi(x,{\tilde{\mu}_n})|.
\]

The simplest differentially private mechanisms are the additive noise mechanisms, e.g., the Laplace and Gaussian mechanisms. 
These mechanisms require the non-private estimator to have finite global sensitivity. 
Unfortunately, virtually all standard location estimators, including both univariate and multivariate medians, have infinite global sensitivity, see e.g., \citep{avellamedina2019differentially}. 
% \ajedit{This point needs to be mentioned in the response to referees, why? do they complain about this?}

% In the particular case of estimating a multivariate median, it is natural to set $\phi$ to be a \emph{depth function}. 
However, \emph{depth functions}, the objective functions for multivariate medians which we will discuss shortly, all have finite global sensitivity. 
This fact makes them a good candidate for use with the \emph{exponential mechanism}, which is a general mechanism used to produce differentially private estimators from non-private estimators that are maximizers of a given objective function. 
For a given $\beta>0$, base measure (i.e., prior) $\pi\in\cM_1(\rdd)$ and objective function $\phi:\rdd\times\cM_1(\rdd)\to\re$, suppose that $\int \exp(\beta \phi(\theta,\nu))d\pi(\theta)<\infty$ for all $\nu\in\cM_1(\rdd)$, then the exponential mechanism with base measure $\pi$ is given by
\begin{equation}\label{eqn::em-dens}
Q_{\hmu_n,\beta} \propto \exp(\beta\phi(\theta,\hmu_n))d\pi.
\end{equation}
A sample from the exponential mechanism $\tilde{\theta}_n\sim Q_{\hmu_n,\beta}$ then provides an estimate of the population value $\theta_0=\argmax_{\rdd}\phi(\theta,\mu)$. 
In regard to the choice of $\beta$, it was shown by \citet{McSherry2007} that if $\tilde{\theta}_n$ is produced by the exponential mechanism, with $\beta\leq  \eps/2\GS_n(\phi)$ then $\tilde{\theta}_n$ is $\eps$-differentially private. 

\subsection{The depth function framework}
When first studying multivariate median estimation, one might try to naively extend the notion of a univariate median to the multivariate setting, i.e., the coordinate-wise median. This, however, is well-known to be an unsatisfactory as a measure of center for many reasons \citep{Small1990,Serfling2006,Liu_Serfling_Souvaine_2006}. For example, the coordinate-wise median can reside outside the convex hull of the data \citep{Serfling2006}. The general framework of \emph{depth functions}, or statistical depth functions, was introduced to resolve these issues and is now the standard approach to defining and studying multivariate medians. See \citep{Small1990, Zuo2000, Mosler2002,Serfling2006,Liu_Serfling_Souvaine_2006} for a necessarily small introduction to this broad research field. 

Let us now recall the basic notions of (statistical) depth functions. 
Roughly speaking, depth functions are functions of the form $\D\colon \rdd\times \cM_1(\rdd) \rightarrow \re^+$ which, given a probability measure and a point in the domain, assign a number. 
This number represents how central that point is with respect to the measure; it is called the depth of that point. 
Given a depth function, the corresponding median is then defined to be a maximizer of this depth, that is, the ``deepest'' point:
\[\med(\mu; \D)=\argmax_{x\in\rdd} \D(x,\mu).\]
Depth-based medians are generally robust, in the sense that they are not affected by outliers. 
For instance, depth-based medians have a high breakdown point and favorable properties related to the influence function \citep{Chen2002, Zuo2004}. 
\kr{Let us now be more precise. For two random vectors $X,Y$, we write $X\eqd Y$ when $X$ is equal in distribution to $Y$. Recall that a measure $\mu\in\cM_1(\rdd)$ is centrally symmetric about a point $x\in\rdd$ if $X-x\eqd x-X$ for $X\sim \mu$. 
In the following, let $\sC_x\subset\cM_1(\rdd)$ denote the set of centrally symmetric measures about $x$ and for $A\in \re^{d\times d}$ and $b\in\rdd$, let $A\mu+b$ be the law of $AX+b$ if $X\sim \mu$. 

\begin{definition}\label{def::depth}
A function $\D\colon \rdd\times \cM_1(\rdd)\rightarrow\re^+ $ is a \emph{depth function} with admissible set $\scA\subseteq \cM_1(\rdd)$ if, for all $\mu\in \scA$, the following four properties hold:
\begin{enumerate}
    \item Similarity invariance: For all orthogonal matrices $A\in \re^{d\times d}$ and $b\in\rdd$ it holds that $\D (x,\mu)=\D (Ax+b,A\mu+b)$. 
    \item Maximality at center: If $\mu\in \sC_{\theta_0}$ then $\D(\theta_0,\mu)=\sup_x\D (x,\mu)$.
    \item Decreasing along rays: Suppose $\D$ is maximized at $\theta_0$. For all $p\in (0,1)$, it holds that $\D (x,\mu)\leq \D(p\theta_0+(1-p)x,\mu)\leq \D (\theta_0,\mu).$
    \item Vanishing at infinity: $\lim_{c\rightarrow\infty}\D (cu,\mu)=0$ for any unit vector $u$. 
\end{enumerate}
\end{definition}
\kr{\noindent The first property says that the depth function has limited dependence on the chosen coordinate system. The remaining properties ensure that $\D$ measures depth. The admissible set $\scA$ is the set of measures or distributions for which the function measures depth.}
% If a function $\D$ satisfies Definition \ref{def::depth} for a given $\scA$, then $\D$ measures depth for any$\mu\in \scA$. 
% Note that the precise mathematical properties that should be possessed by a function in order to say it measures depth is not agreed upon in the literature. 
\begin{remark}\label{rem::depth_def}
There are a few distinct definitions of the notion of a depth function in the statistics literature. See, e.g., \citep{Zuo2000, Liu_Serfling_Souvaine_2006,RAMSAY201951} for some alternatives. 
We choose Definition~\ref{def::depth} as it includes as many existing depth functions as possible while retaining the spirit of a depth function. 
We use here a weaker definition of a depth function than the influential work of \citet{Zuo2000}, which replaces similarity invariance with affine invariance. 
For a discussion of desirable properties for a depth function to possess, see \citep{Zuo2000, Liu_Serfling_Souvaine_2006, SerflingDepthFO}.
\end{remark}}

Examples of commonly used depth functions include the halfspace depth (or Tukey depth) \citep{Tukey1974}, the simplicial depth \citep{Liu1988}, the spatial depth \citep{Vardi2000,Serfling2002}, the integrated dual depth (IDD) \citep{Cuevas2009}, and the integrated-rank-weighted (IRW) depth \citep{RAMSAY201951}. 
\kr{For the reader's convenience, we provide a brief review of these depth functions and their basic properties in Section~\ref{sec::depth}. 
We emphasize that this is a small selection of the large number of depth functions used in the literature, and reviewing them all is outside of the scope of this work. 

There are a number of desirable properties of a depth function, relating to transformation invariance, robustness, computability, ability to measure depth for a broad range of distributions and more, from which no depth function arises as the ``gold standard'' \citep{Zuo2000,Liu_Serfling_Souvaine_2006}. 
Each depth function yields a distinct notion of a median, each with its own desirable properties. As our goal in this paper is to unify these two frameworks, we will not investigate the question of which depth is optimal, statistically or computationally, for a given problem. Instead, our main contribution is to give general finite-sample deviations bounds which applies to many of the existing depth functions, under mild conditions on the population measure. }

% There are many competing motivations for using a given median and 
% I don't think this needs to be here
% [This sentence needs to be moved somehwere else: If, instead of the median, one is interested in computing depth values themselves, this can be done privately with an additive noise mechanism, see Section~\ref{sec::computing-depths}.]

\subsection{A finite-sample deviations bound}
Our main result relies on a regularity condition, for which we need to introduce the following notation. 
Let $\sB$ denote the space of Borel functions from $\rdd$ to $[0,1]$. 
For a family of functions, $\sF\subseteq\sB$, define a pseudometric on $\cM_1(\rdd)$,  $d_\sF(\mu,\nu) = \sup_{g\in\sF}|\int gd(\mu-\nu)|$, where $\mu,\nu\in\cM_1(\rdd)$. 
Recall that many standard metrics on probability measures can be written in this fashion. 
For example, taking $\sF$ to be the set of indicator functions of Borel sets yields the Total Variation distance and taking $\sF$ to be the set of indicator functions of semi-infinite rectangles yields the Kolmogorov--Smirnov distance. 
% For example, taking $\sF$ to be the set of indicator functions of Borel sets yields the Total Variation distance,  taking $\sF$ to be the set of indicator functions of semi-infinite rectangles yields the Kolmogorov--Smirnov distance, and taking $\sF$ to be the set of 1-Lipschitz functions yields the 1-Wasserstein distance. 
We then define the following important regularity condition.

\begin{definition}\label{def::kf-reg}
We say that $\D$ is \emph{$(K,\sF)$-regular} if there exists a class of functions $\sF\subset \mathscr{B}$ such that $\D(x,\cdot)$ is $K$-Lipschitz with respect to the $\sF$-pseudometric uniformly in $x$, i.e., for all $\mu,\nu\in\cM_1(\rdd)$
\[
\sup_{x}|\D(x,\mu)-\D(x,\nu) | \leq K d_\sF(\mu,\nu).
\]
\end{definition}
\noindent \kr{Intuitively, $(K,\sF)$-regularity can be thought of as a robustness condition as it says that $(K,\sF)$-regular functions are affected by extreme observations in a very limited manner. For example, it immediately implies the bound $\GS_n(\D)\leq K/n$ on the global sensitivity through the boundedness of $g\in \sF$, see Lemma \ref*{lem::GS}.  
This condition is very convenient for two reasons: Firstly, it is elementary to check this condition for many popular depth functions (see, Table~\ref{table::cond}). Secondly, it immediately yields an $\epsilon-$differentially private estimator by choosing $\beta=n\eps/2K$, see Lemma~\ref*{lem::ub-gs-dp}. While this assumption is convenient for differential privacy, we show in a companion paper \citep{ramsay2025an} that our main deviations bound holds under a weaker condition and can apply to other statistical learning settings.} 
% Thus, for the proposed private median estimator, we will assume that $\D$ is $(K,\sF)$-regular (Condition \ref{cond::phi-k-reg}) and take $\beta = n\eps/2K$, which ensures $\tilde{\theta}_n$ is $\eps$-differentially private, see Lemma \ref*{lem::ub-gs-dp}.
% \ajedit{We work with the assumption of $(K,\cF)$-regularity as it is very convenient for two reasons: first, as just observed, it immediately implies $\epsilon$-differential privacy of the cooresponding mechanism}

% \kr{The implication of $\eps$-differentially privacy, and the fact that for many popular depth functions, there is some $K>0$ and some family of functions $\sF\subseteq \sB$ with $\VC(\sF)<\infty$ such that $\D$ is $(K,\sF)$-regular (see Table~\ref{table::cond}), makes this a convienent assumption.} 
% Note that in our case, we will typically take $\sF$ to be a class of functions with low complexity (in the sense of Vapnik--Chervonenkis dimension). 

We now state the necessary regularity conditions on the pair $(\mu,\D)$ that are required for our main result. 
In the following, let $\VC(\sF)$ denote the Vapnik--Chervonenkis dimension of $\sF$. 
\begin{condition}\label{cond::phi_um}
The function $\D(\cdot,\mu)$ has a maximizer.
\end{condition}
\begin{condition}\label{cond::phi_lip} The map
$x\mapsto \D(x,\mu)$ is $L$-Lipschitz a.e.\ for some $L>0$.
\end{condition}
\begin{condition}\label{cond::phi-k-reg}
There is some $K>0$ and some family of functions $\sF\subseteq \sB$ with $\VC(\sF)<\infty$ such that
$\D$ is $(K,\sF)$-regular.
\end{condition}

To state our main result, we must also introduce the $\emph{discrepancy function}$. 
Let  $E_{\D,\mu}=\{\theta:\D(\theta,\mu)=\max\limits_{x\in\rdd} \D(x,\mu)\}$ denote the set of maximizers of $\D(\cdot,\mu)$ for fixed $\D$ and $\mu$. 
For a set $A$ and $r>0$, let $B_r(A)=\{x\colon \min\limits_{y\in A}\norm{x-y}\leq r\}$. 
The discrepancy function of the pair $(\D,\mu)$ is 
\begin{equation}
\alpha(t) = \D(\theta_0,\mu) - \sup_{x\in B^c_t(E_{\D,\mu})} \D(x,\mu),
\end{equation}
where $\theta_0\in E_{\D,\mu}.$
\kr{Given $t>0$, the discrepancy function is the minimum distance between the maximal population depth value, and the maximal population depth value that is at least $t$-far from the maximizing set $E_{\D,\mu}$. For example, often there is one maximizer, the median, and the discrepancy function measures the minimum difference between the depth of the median, and the depth of any point at least $t$-far from the median. 
The faster the discrepancy function grows in $t$, the faster the depth function decreases as we move away from the median.}
Note that $\alpha$ is a non-decreasing function because the optimization problem involved is over decreasing sets. 
In Table~\ref{table::alpha} below, we explicitly calculate $\alpha$ for several depth functions.

We are now ready to state our main result, which provides bounds on the finite-sample deviations of the proposed differentially private multivariate medians via depths. 
For concreteness, we present our results for two of the most popular choices of priors, namely Gaussian measures and the uniform measure on $\cube_{R}(y)$, the $d$-dimensional cube of side-length $R$ centered at $y$. 
In the following, let $d_{R,y}(x)$ denote the minimum distance from a point $x=(x_1,\,\dots,\,x_d)$ to a face of the cube $\cube_{R}(y)$: $d_{R,y}(x)=\min\limits_{1\leq i\leq d}|x_i-y_i\pm R/2|$.
Similarly, denote the minimum distance from a set $B\subset\rdd$ to a face of the cube $\cube_{R}(y)$ by $d_{R,y}(B)=\inf\limits_{x\in B}d_{R,y}(x)$ and let $d(x,B)=\inf\limits_{y\in B}\norm{x-y}$ denote the usual point-to-set distance, where $\normm{\cdot}$ denotes the Euclidean norm. 
%%%%%%%%%%%%%%%%%%%%%%%%%%%%%%%%%%%%%%%%%
Define $\alpha^{-1}(t)=\sup\{r\geq 0\colon \alpha(r)=t\}$, where one notes that $\alpha^{-1}$ exists because $\alpha$ is a monotone function. 
Lastly, we write $a\lesssim b$ if $a\leq C b$ for some universal constant %\footnote{By universal constant, we mean a constant which is independent of all other parameters defined in this paper, except potentially other universal constants.} 
$C>0$. 
\begin{theorem}\label{thm::main-dev-bound}
If Conditions \ref{cond::phi_um}--\ref{cond::phi-k-reg} hold, then the following holds:
\begin{enumerate}[label=(\roman*)]
\item Suppose that $L\geq 1$. If $\pi=\mathcal{N}(\theta_\pi, \sigma_\pi^2I)$, 
then there exists a universal constant $c>0$ such that for all $n\geq 8K/\epsilon$, all $d>2$ and all $0<\gamma<1$, with probability at least $1-\gamma$, we have that
\begin{multline}\label{devi_normal}
    d(E_{\D,\mu},\tilde\theta_n)\lesssim   \alpha^{-1}\Bigg(cK\Bigg[\sqrt{\frac{\log(1/\gamma)\vee \VC(\sF)\log n}{n}}\\
 \bigvee\ \frac{\log\left(1/\gamma\right)\vee\left(\frac{d(E_{\D,\mu},\theta_\pi)^2}{\sigma_\pi^2}+ d\log\left(\frac{\sigma_\pi Ln\epsilon}{K}\vee d\right)\right)}{n\epsilon}\Bigg]\Bigg) . 
\end{multline}
\item If instead $\pi\propto \ind{x\in \cube_{R}(\theta_\pi)}$ where $E_{\D,\mu}\subset\cube_{R}(\theta_\pi)$, then there exists a universal constant $c>0$ such that for all $n,d\geq 1$ and all $0<\gamma<1$, with probability at least $1-\gamma$, we have that
\begin{multline}\label{devi_cube}
    d(E_{\D,\mu},\tilde\theta_n)\lesssim   \alpha^{-1}\Bigg(cK\Bigg[\sqrt{\frac{\log(1/\gamma)\vee \VC(\sF)\log n}{n}}\\
 \bigvee\ \frac{\log\left(1/\gamma\right)\vee d\log\left(\frac{Rn}{nd_{R,\theta_\pi}(E_{\D,\mu})\wedge K/L}\right)}{n\epsilon}\Bigg]\Bigg) .
\end{multline}
\end{enumerate}
\end{theorem}
The proof of Theorem~\ref{thm::main-dev-bound} can be found in Section \ref{app::proofs}, and relies on an elementary concentration bound (Theorem~\ref{thm::main_result}), which we introduce in Section \ref{sec::concen--bound}. 
We also provide a version of Theorem~\ref{thm::main-dev-bound} in terms of sample complexity, see Corollary~\ref{cor::main-sc-cor}. 
\kr{Theorem~\ref{thm::main-dev-bound} is a deviations bounds for private multivariate medians. Put simply, it says that $\alpha(d(E_{\D,\mu},\tilde\theta_n))$ is bounded above, with high probability, by two major terms: the sampling error term, which appears first in the upper bound, and the cost of privacy term, which appears second in the upper bound. 

To further help interpret Theorem~\ref{thm::main-dev-bound}, note that we are mainly concerned about small deviations. For many choices of $\mu$ and $\D$, one can show that $\alpha^{-1}$ is at most linear for small $t>0$ under various regularity conditions that capture many popular depths, see Section~\ref{sec::linear} where we present such conditions. In this case one can remove the $\alpha^{-1}$ from the right-hand sides of \eqref{devi_normal}-\eqref{devi_cube}.
%For small $t$, for many choices of $\mu$ and $\D$, $\alpha^{-1}$ is often at most linear, see Section \ref{sec::linear}. 
%From a technical perspective, this is convenient, since then we need not be concerned about inverting $\alpha$ directly. 
In addition, this fact coupled with the fact that many of the depth functions satisfy Condition \ref{cond::phi-k-reg} with $K=O(1)$ and $\VC(\sF)=O(d)$, and, in many contexts, the population median is unique, recovers a bound similar to those obtained in the case of private mean estimation \citep{Kamath2018, Kamath2020, Cai2019}. 
That is, the bound given in Theorem~\ref{thm::main-dev-bound} then reduces to the simpler bound
$$    \norm{\med(\mu; \D)-\tilde\theta_n}\lesssim   \sqrt{\frac{\log(1/\gamma)\vee d\log n}{n}}\\
 \bigvee\ \frac{\log\left(1/\gamma\right)\vee d\log\left(\frac{Rn}{nd_{R,\theta_\pi}(E_{\D,\mu})\wedge 1/L}\right)}{n\epsilon}.$$ 
Under such $\mu$ and $\D$, assuming a reasonable prior and ignoring logarithmic terms, the cost of privacy is proportional to $\sqrt{d/n}\eps$. For concrete examples of Theorem~\ref{thm::main-dev-bound} applied to specific depth functions and population measures $\nu$, see Examples \ref{exam::mvtg} and \ref{exam::cauchy}.}
% \begin{remark}[Small $t$ and the cost of privacy]
% \end{remark}

While we state our result for general depth functions satisfying the above assumptions, our results apply to many of the most popular depths, such as halfspace depth and, more broadly, Type D depths in the sense of \citet{Zuo2000}, spatial depth, simplicial depth and both the integrated rank-weighted and integrated dual depths. 
Depth functions to which these results apply and the conditions under which they apply are discussed at length in Section~\ref{sec::depth}. 
The conditions for specific popular depths are summarized in Table~\ref{table::cond}.  
% Observe that many of the depths have $K=O(1)$ and $\VC(\sF)=O(d)$. 
% In addition, for many popular models $\alpha^{-1}(t)\propto t$ for small $t$, see Examples \ref{exam::mvtg} and \ref{exam::cauchy}. 
% Thus, in these models and ignoring logarithmic terms, the cost of privacy is proportional to $\sqrt{d/n}\eps$ under both a Gaussian prior and a uniform prior. 

We note that our bound is essentially sharp, see Lemma \ref{lem::sharp}. 
Given that if $\mu$ is Gaussian, then the population mean equals the population median, the bound on the sample complexity implied by Theorem~\ref{thm::main-dev-bound} (see Corollary~\ref{cor::main-sc-cor}) matches the lower bound on the sample complexity for private Gaussian mean estimation given by \cite{Kamath2018}, Theorem 6.5 and \cite{Hopkins2021}, Theorem 7.2 up to logarithmic factors. 

Alternatively, \citet{Cai2019} give a minimax lower bound for approximately differentially private sub-Gaussian mean estimation. 
Unsurprisingly, given that it is designed for the approximately differentially private setting, this lower bound becomes $-\infty$ if we apply it to the setting of pure differential privacy. 
However, it is still interesting to compare Theorem~\ref{thm::main-dev-bound} to the bound of \citet{Cai2019} in some way. 
A trivial manipulation of the conclusion of Lemma~\ref{lem::sharp} shows that the bound given in Theorem~\ref{thm::main-dev-bound} matches the lower bound of \citet{Cai2019}, up to logarithmic terms, when the ``approximate differential privacy parameter'', i.e., $\delta$ is taken to be at least $n^{-k}$ for some $k>0$. 
Therefore, as stated above, the bound in Theorem~\ref{thm::main-dev-bound} is essentially sharp. 
Furthermore, at least for the setting of Gaussian mean estimation, relaxing to the setting of approximate differential privacy will not yield faster rates than those given by Theorem~\ref{thm::main-dev-bound}.\footnote{It should not be difficult to extend this conclusion to the more general case of symmetric sub-Gaussian median estimation.}

Recall that estimating a parameter from an unbounded parameter space, i.e., unbounded estimation, is non-trivial in the setting of differential privacy, e.g., see \cite{Avella-Medina2019a}. 
Theorem~\ref{thm::main-dev-bound} shows that choosing $\sigma_\pi$ relatively large allows one to perform unbounded estimation at the cost of an extra $\log d$ in the second term of \eqref{devi_normal}, or, equivalently, an  extra $\log d$ to the sample complexity. 
To elaborate, using a depth function in the exponential mechanism with a Gaussian prior constitutes unbounded estimation of the population median. 
Theorem~\ref{thm::main-dev-bound} then shows that  $d(E_{\D,\mu},\,\theta_\pi)/\sigma_\pi=O(\sqrt{d\log d})$ is sufficient for the deviations under a misspecified Gaussian prior to match the deviations under a correctly specified uniform prior, up to a logarithmic factor in $d$. 
Thus, we only require $d(E_{\D,\mu},\,\theta_\pi)$ to be polynomial in $d$, i.e., choosing $\sigma_\pi$ to be some large polynomial in $d$ will ensure that $\sigma_\pi\geq d(E_{\D,\mu},\,\theta_\pi)/\sqrt{d\log d}$.

\begin{remark}[Projection depth]
One popular notion of depth not covered by our analysis is the projection depth \citep{zuo2003}, see also \citep{Stahel1981,Donoho1982}. This is because a private median drawn from the exponential mechanism based on the empirical version of this depth function, is inconsistent $\mu$-a.s, see Lemma \ref*{lem::incon}. For a private version of the projection-depth-based median, see \citep{ramsay2024}. 
\end{remark}
\begin{remark}[Comparison to non-private bounds]
The first term in the maximum in Theorem~\ref{thm::main-dev-bound} is a deviations bound for non-private depth-based medians. 
If we assume that $\theta_0\in \cube_{R}(0)$ and choose the depth function to be halfspace depth (Definition~\ref{def::hs}), then this upper bound recovers the upper bound on the non-private estimation error of the halfspace median for the Gaussian measure, see Theorem~2.1 of \citet{chen2018robust}, see also Theorem~3 of \citet{Zhu2020}. 
\end{remark}

\begin{remark}[Relaxing the Lipschitz requirement]
The Lipschitz requirement for $\D$ can be relaxed to obtain more general deviation bounds, at the cost of complicating the presentation. Furthermore, all the depth functions we consider are Lipschitz functions a.e.\ under mild assumptions, see Theorem~\ref{thm::depth_cond}. 
% In addition, virtually all definitions of depth are bounded above by 1 for any $\mu$ \citep{Zuo2000}. 
It is also not necessary to require that $E_{\D,\mu}\subset\cube_{R}(\theta_\pi)$. 
It is sufficient to assume that $E_{\D,\mu}\cap \cube_{R}(\theta_\pi)\neq \emptyset$. In this case, one can replace $d_{R,\theta_\pi}(E_{\D,\mu})$ with $d_{R,\theta_\pi}(E_{\D,\mu}\cap \cube_{R}(\theta_\pi))$. 
\end{remark}
\begin{remark}[The univariate setting]
Theorem~\ref{thm::main_result} can be applied to obtain an optimal univariate private median estimator. 
\citet{zamos2020} gave a lower bound on the minimax sample complexity for estimating the one-dimensional median of a population measure whose density is bounded below by a positive constant. 
Applying Theorem~\ref{thm::main-dev-bound} when $d=1$ and $\D$ being the halfspace depth to such measures yields an upper bound with matches that lower bound. This upper bound recovers the one given by \citet{Karwa2017} for one-dimensional Gaussian mean estimation. 
\end{remark}

Note that Theorem~\ref{thm::main-dev-bound} has a straightforward restatement in terms of sample complexity.  
%%%%%%%%%%%%%%%%%%%%%%%%%%%%%%%%%%%%%%%%%%%%%%%%%%%%%
\begin{corollary}\label{cor::main-sc-cor}
Suppose that Conditions \ref{cond::phi_um}--\ref{cond::phi-k-reg} hold. 
Then the following holds
\begin{enumerate}[label=(\roman*)]
\item Suppose that $\D(x,\mu)\leq 1$ and $L\geq 1$. If $\pi=\mathcal{N}(\theta_\pi, \sigma_\pi^2I)$, 
then there exists universal constants $C,c>0$ such that for all $t\geq 0$, all $d>2$ and all $0<\gamma<1$, we have that $d(E_{\D,\mu},\tilde\theta_n)<t$ with probability at least $1-\gamma$ provided 
\begin{multline}\label{sc_normal}
    n>CK^2\Bigg(\frac{\log(1/\gamma)\vee[ \VC(\sF)\log(\frac{K}{c\alpha(t)})\vee 1]}{\alpha(t)^2}\\ \bigvee \frac{\log\left(1/\gamma\right)\vee\left(\frac{d(E_{\D,\mu},\theta_\pi)^2}{\sigma_\pi^2} \right)\vee d\log\left(\frac{\sigma_\pi\, L}{\alpha(t)}\vee  d  \right)}{\epsilon\alpha(t)}\Bigg).
\end{multline}
% {\small{
% \begin{equation}\label{sc_normal}
%     n>CK^2\left(\frac{\log(1/\gamma)\vee[ \VC(\sF)\log(\frac{K}{c\alpha(t)})\vee 1]}{\alpha(t)^2} \vee \frac{\log\left(1/\gamma\right)\vee\left(\frac{d(E_{\D,\mu},\theta_\pi)^2}{\sigma_\pi^2} \right)\vee d\log\left(\frac{\sigma_\pi\, L}{\alpha(t)}\vee  d  \right)}{\epsilon\alpha(t)}\right).
% \end{equation}}}
\item If instead $\pi\propto \ind{x\in \cube_{R}(\theta_\pi)}$ where $E_{\D,\mu}\subset\cube_{R}(\theta_\pi)$, then there exists universal constants $C,c>0$ such that for all $t\geq 0$, all $d\geq 1$ and all $0<\gamma<1$, we have that $d(E_{\D,\mu},\tilde\theta_n)\leq t$ with probability at least $1-\gamma$ provided
\begin{multline}\label{sc_cube}
    n>C K^2\Bigg(\frac{\log(1/\gamma)\vee [ \VC(\sF)\log(\frac{K}{c\alpha(t)})\vee 1]}{\alpha(t)^2}\\ \bigvee \frac{\log\left(1/\gamma\right)\vee d\log\left(\frac{ R}{d_{R,\theta_\pi}(E_{\D,\mu})\wedge \alpha(t)/L}\right)}{\alpha(t)\epsilon} \Bigg).
\end{multline}
% \begin{equation}\label{sc_cube}
%     n>C K^2\left(\frac{\log(1/\gamma)\vee [ \VC(\sF)\log(\frac{K}{c\alpha(t)})\vee 1]}{\alpha(t)^2}\vee \frac{\log\left(1/\gamma\right)\vee d\log\left(\frac{ R}{d_{R,\theta_\pi}(E_{\D,\mu})\wedge \alpha(t)/L}\right)}{\alpha(t)\epsilon} \right).
% \end{equation}
\end{enumerate}
\end{corollary}
The proof is deferred to Appendix \ref*{app:tech-proofs}.

%Our result cannot recover Theorem 12 of \citet{2021Liub} since their problem construction differs; they assume sample is corrupted and this assumption is built into their sample complexity bound. %%%%%%%%%%%%%%%%%%%%%%%%%%%%%%%%%%%%%%%%%%%%
\begin{table}[t]
\caption{Table of the discrepancy function $\alpha(t)$ for different depth functions and underlying population measures. 
Recall that $\mu$ is $d$-version symmetric about zero if for any unit vector $u$, $X^\top u\eqd  a(u)Z$ where $X\sim \mu$ and $Z$ is a random variable such that $Z\eqd-Z$ and $a(u)=a(-u)$. 
Here, $v_1$ is the eigenvector associated with the largest eigenvalue $\lambda_1$ of $\Sigma$ and $F_0(x)=F(x,u,\mu)$, see \eqref{eqn::cdf_u}, and $\nu$ is the uniform measure on the $(d-1)$-dimensional unit sphere. }
% \vspace{1em}
\begin{tabular}{lcc}
\vspace{0.1em}
 $\D$/$\mu$:& $N_d(\theta_0,\Sigma)$      &  $d$-version symmetric       \\ 
\hline
& & \\
$\HD$  & $\Phi\left(\frac{t}{\sqrt{\lambda_1}}\right)-\frac{1}{2}$ &  $\frac{1}{2}-\sup\limits_{\norm{v}=1}\,\inf\limits_{\norm{u}=1} F_0\left(\frac{-t\,v^\top u}{a(u)}\right)$ \\
& & \\
$\IRW$  &  $\mathlarger{\int} \left|\frac{1}{2}-\Phi\left(\frac{t\, v_1^\top u}{\sqrt{u^\top \Sigma u}}\right)\right|\,d\nu $ & $\inf\limits_{\norm{v}=1}\mathlarger{\int}\left|\frac{1}{2}-F_0\left(\frac{t\,v^\top u}{a(u)}\right)\right|\,d\nu$ \\ 
& & \\
$\IDD$  &  $\frac{1}{4}-\mathlarger{\int}\Phi\left(\frac{-t \,v_1^\top u}{\sqrt{u^\top \Sigma u}}\right)\Phi\left(\frac{t \,v_1^\top u}{u^\top \Sigma u}\right)\,d\nu$  & $\frac{1}{4}-\sup\limits_{\norm{v}=1}\mathlarger{\int} F_0\left(\frac{t\,v^\top u}{a(u)}\right)F_0\left(\frac{-t\,v^\top u}{a(u)}\right)\,d\nu$
\end{tabular}
\label{table::alpha}
\end{table}
%%%%%%%%%%%%%%%%%%%%%%%%%%%%%%%%%%%%%%%%%%%%%%

\subsection{Examples}
Let us now turn to some concrete examples that illustrate our result. 
\kr{We first introduce the halfspace depth, which will be used in our examples, where we apply our main results to the halfspace median. 
For $X\sim\mu$, define 
\begin{equation}\label{eqn::proj_dist}
   F(x,u,\mu)=\mu\left(X^\top u\leq x^\top u\right) .
\end{equation}
Halfspace depth, otherwise known as Tukey depth \citep{Tukey1974}, is defined as follows:
\begin{definition}[Halfspace depth] \label{def::hs}
% Let $\bS^{d-1}= \{u\in \rdd\colon \ \norm{u}=1\}$ be the set of unit vectors in $\rdd$. 
The halfspace depth $\HD$ of a point $x\in \rdd$ with respect to $\mu\in\cM_1(\rdd)$ is
\begin{equation}\label{eqn::cdf_u}
    \HD (x,\mu)=\inf_{\norm{u}=1} F(x,u,\mu).
\end{equation}
\end{definition}
The halfspace depth of a point $x\in\rdd$ is the minimum probability mass contained in a closed halfspace containing $x$. 
The halfspace median is the maximizer of the halfspace depth (when it is unique).}
\kr{We first consider the canonical problem of multivariate Gaussian location estimation under a misspecified prior, where here, we say $\pi$ is misspecified when its center $\theta_\pi$ is far from the population halfspace median.} 
We quantify the relationship between finite-sample deviations, prior misspecification, privacy budget, dimension and statistical accuracy. 
In this case, omitting logarithmic terms, for a misspecification level of  $\norm{\theta_0-\theta_\pi}/\sigma_\pi\leq \sqrt{d}$, the finite-sample deviations of the private halfspace median reduce to $O(\sqrt{\lambda_1d/n}
 \vee\ \sqrt{\lambda_1}d/n\epsilon)$. 
Here, $\lambda_1$ is the largest eigenvalue of the covariance matrix of $\mu$. 

\kr{Next, we consider estimating the location parameter of a population measure $\mu$ made up of Cauchy marginals, which is known as the poly-Cauchy distribution. 
This setup is more difficult in the sense that the mean of the probability measure $\mu$ does not exist. 
For instance, theory concerning differentially private mean estimators cannot be applied in this setting.} 
% Next, we consider estimating the location parameter of a population measure $\mu$ made up of Cauchy marginals. 
% This setup is more difficult in the sense that the mean of the probability measure $\mu$ does not exist. 
% For instance, theory concerning differentially private mean estimators cannot be applied in this setting. 
% We demonstrate that, omitting logarithmic terms, the number of samples required for non-private estimation $O(d^2)$ eclipses the cost of privacy, which is $O(d^{3/2}/\epsilon)$; if we have enough samples to perform ``heavy tailed'' estimation, then we have enough samples to ensure privacy. 

\begin{example}[Multivariate Gaussians with misspecified prior]\label{exam::mvtg}
Consider the problem of estimating the location parameter $\theta_0$ of a multivariate Gaussian measure $\mu= \mathcal{N}(\theta_0,\Sigma)$ and we wish to estimate $\theta_0$. 
For this example, let $\lambda_1>\ldots>\lambda_d>0$ be the ordered eigenvalues of the covariance matrix $\Sigma$. 
Suppose that in order to estimate $\theta_0$ we use the exponential mechanism paired with the halfspace depth $\D=\HD$ (see Definition~\ref{def::hs}) in conjunction with a misspecified Gaussian prior: $\pi=\mathcal{N}(\theta_\pi, \sigma_\pi^2 I)$. 
% such that $\norm{\theta_\pi-\theta_0}$ is at most polynomial in $d$.

% Corollary~\ref{cor::main-sc-cor}
In order to apply Theorem~\ref{thm::main-dev-bound}, we need to check Condition \ref{cond::phi_um}, \ref{cond::phi-k-reg} and Lipschitz continuity of $\HD$. 
First, note that Condition \ref{cond::phi_um} holds for any bounded depth function as defined in Definition \ref{def::depth}. 
Next, one can show that (see Theorem~\ref{thm::depth_cond}) the halfspace depth is $(1,\sF)$-regular where $\sF$ is the set of closed halfspaces in $\rdd$, thus, $\VC(\sF)=d+2$. 
Lastly, the map $x\mapsto \HD(x,\mu)$ is $L$-Lipschitz with $L=(2\pi\lambda_d)^{-1/2}$, implying Condition \ref{cond::phi_lip} is satisfied. 
We may now apply Theorem~\ref{thm::main-dev-bound}, for which we must compute $\alpha(t)$. 
To this end, it holds that $\alpha(t)\gtrsim  t C_R/\sqrt{\lambda_1},$
for some large $R>t$, see Lemmas \ref{lem::alpha_linear} and \ref{lem::alpha_mg}. 
% To this end, it holds that\footnote{See Lemma \ref{lem::alpha_mg}.} $$\alpha(t)=1/2-\sup_{\norm{x-\theta_0}=t} \inf_{\norm{u}=1} \Phi\left( \frac{(x-\theta_0)^\top u}{\sqrt{u^\top \Sigma u}}  \right)= \Phi\left( t/\sqrt{\lambda_1} \right) -1/2\gtrsim  t C_R/\sqrt{\lambda_1},$$
% for some large $R>t$, see Lemma \ref{lem::alpha_linear}. 
Therefore, under halfspace depth, with probability at least $1-\gamma$, it holds that 
\begin{equation*}
\normm{\theta_0-\tilde\theta_n}\lesssim   \frac{\sqrt{\lambda_1}}{C_R}\left[\sqrt{\frac{\log(1/\gamma)\vee d\log n}{n}}\\
 \bigvee\ \frac{\log\left(1/\gamma\right)\vee\frac{\norm{\theta_0-\theta_\pi}^2}{\sigma_\pi^2}\vee d\log\left(n\epsilon\sigma_\pi/\lambda_d \vee d\right)}{n\epsilon}\right].
\end{equation*}
% {\small{
% $$n \geq C \left(\frac{\log(1/\gamma)\vee [d\log(\frac{1}{\Phi\left( t/\sqrt{\lambda_1} \right) -1/2})\vee 1]}{(\Phi\left( t/\sqrt{\lambda_1} \right) -1/2)^2} \vee \frac{\log\left(1/\gamma\right)\vee\left( \norm{\theta_0-\theta_\pi}^2/\sigma_\pi^2\right)\vee d\log\left(\frac{\sigma_\pi/\lambda_d}{\Phi\left( t/\sqrt{\lambda_1} \right) -1/2}\vee  d  \right)}{\epsilon(\Phi\left( t/\sqrt{\lambda_1} \right) -1/2)}\right),$$}}
% for some universal $C>0$.
Suppose we would like to achieve a success rate of $1-e^{-d}$, i.e., $\gamma=e^{-d}$. 
This reduces the deviations bound to 
\begin{equation*}
\normm{\theta_0-\tilde\theta_n}\lesssim   \frac{\sqrt{\lambda_1}}{C_R}\left[\sqrt{\frac{d\log n}{n}}
 \bigvee\ \frac{\frac{\norm{\theta_0-\theta_\pi}^2}{\sigma_\pi^2}\vee d\log\left(n\epsilon\sigma_\pi/\lambda_d\vee d\right)}{n\epsilon}\right].
\end{equation*}
% \begin{align*}
%     n \geq C \left(\frac{ d\log\left(\frac{1}{\Phi\left( t/\sqrt{\lambda_1} \right) -1/2}\right)\vee 1}{(\Phi\left( t/\sqrt{\lambda_1} \right) -1/2)^2} \vee \frac{\left(\norm{\theta_0-\theta_\pi}^2/\sigma_\pi^2\right)\vee d\log\left(\frac{\sigma_\pi/\lambda_d}{\Phi\left( t/\sqrt{\lambda_1} \right) -1/2}\vee  d  \right)}{\epsilon(\Phi\left( t/\sqrt{\lambda_1} \right) -1/2)}\right)\cdot
% \end{align*}
% For very large $d$, the right-hand term in the maximum dominates the deviations bound. 
Observe now that in order for the prior effect to be negligible, we require $\norm{\theta_0-\theta_\pi}/\sigma_\pi \leq  \sqrt{d\log d}$ and that $\sigma_\pi$ is at most polynomial in $d$. 
This suggests that, up to a point, a larger prior variance $\sigma_\pi^2$ produces lower deviations.  
In this case, when the error level, $\epsilon$ and $\Sigma$ are fixed, we only need $n$ to grow slightly faster than the dimension  $n\geq C d\log d$ in order to maintain a consistent level of error. 
If the prior effect is negligible in the sense just described, then, omitting logarithmic terms for brevity, the deviations bound reduces to 
% \begin{align*}
%     n \geq C\left( \frac{ d\log\left(\frac{1}{\Phi\left( t/\sqrt{\lambda_1} \right) -1/2}\right)\vee 1}{(\Phi\left( t/\sqrt{\lambda_1} \right) -1/2)^2} \vee \frac{ d\log \left(\frac{d/\lambda_d}{\Phi\left( t/\sqrt{\lambda_1} \right) -1/2}\right)}{\epsilon(\Phi\left( t/\sqrt{\lambda_1} \right) -1/2)}\right)\cdot 
% \end{align*}
\begin{equation}\label{eqn::gp_bound}
    \normm{\theta_0-\tilde\theta_n}\lesssim \frac{\sqrt{\lambda_1}}{C_R}\cdot\left(\sqrt{\frac{d}{n}}
 \vee\ \frac{d}{n\epsilon}\right),
\end{equation}
% where the last inequality holds for large $n$ on account of Lemma \ref{lem::alpha_linear}. 
% \begin{equation*}
% \alpha(\normm{\theta_0-\tilde\theta_n})\lesssim  \sqrt{\frac{d\log n}{n}}
%  \bigvee\ \frac{ d\log\left(nd\epsilon/\lambda_d\right)}{n\epsilon}
% \end{equation*}
We then have that the cost of privacy is eclipsed by the sampling error whenever $n\gtrsim d/\epsilon^2$. 
Therefore, differential privacy is free, provided that $d$ is not too large. 
Furthermore, \eqref{eqn::gp_bound}, makes it is easy to see that in this case, the private medians achieve the minimax lower bound for approximately differentially private sub-Gaussian mean estimation \citep{Cai2019}, if $\delta\propto n^{-k}$. 
\kr{There exists other private estimators that are minimax optimal for private Gaussian mean estimation, e.g., \citep{Kamath2018, Biswas2020}. These estimators may be computable in polynomial time, whereas the private halfspace median would take exponential time. However, existing estimators are not medians and therefore are not necessarily robust. For example, they may not perform well if the population measure is not multivariate Gaussian but instead is multivariate Cauchy. 
Robustness against violations of assumptions is especially important in private data analysis, where we may not have enough privacy budget to check assumptions. 
Note that a bound in terms of Mahalanobis distance (with respect to $\Sigma$), denoted $\norm{\cdot}_\Sigma$, is implied by \eqref{eqn::gp_bound}, through the inequality $||\theta_0-\tilde\theta_n||_\Sigma\leq ||\theta_0-\tilde\theta_n||/\sqrt{\lambda_d}$, viz.
\begin{equation*}
    \norm{\theta_0-\tilde\theta_n}_\Sigma\lesssim \frac{\sqrt{\lambda_1\lambda_d}}{C_R}\cdot\left(\sqrt{\frac{d}{n}}
 \vee\ \frac{d}{n\epsilon}\right).
\end{equation*}
It is possible the techniques used to prove Theorem~\ref{thm::main_result} could be used to prove a general deviations bound in terms of the Mahalnobis distance.}
% $\norm{\theta_0-\tilde\theta_n}_\Sigma\leq \norm{\Sigma^{-1/2}}\norm{\theta_0-\tilde\theta_n}\leq \frac{1}{\sqrt{\lambda_1}}\norm{\theta_0-\tilde\theta_n}$
% The trade-off between the deviations $\normm{\theta_0-\tilde\theta_n}$ and the privacy parameter $\epsilon$ is now clear. 
% For example, in the case of a very small error: $\left(\Phi\left(t/\sqrt{\lambda_1}\right)-1/2\right)\leq \epsilon/\log d$, the privatization does not contribute to the sample complexity. On the other hand, in the case of large $d$ and fixed $t$, i.e., $\left(\Phi\left(t/\sqrt{\lambda_1}\right)-1/2\right)\geq \epsilon/\log d$, the sample complexity is directly affected by the privacy parameter: $n\geq C\cdot\epsilon^{-1}\cdot d\log d$. 
\end{example}

\begin{table}[t]
\caption{Table of values for which different depth functions satisfy Theorem~\ref{thm::main-dev-bound}. Letting $\mu_u$ be the law of $X^{\top}u$ if $X\sim \mu$, $f_u$ is the density of $\mu_u$ with respect to the Lebesgue measure and $L'$ denotes $\sup_y\E{}{\norm{y-X}^{-1}}$. Note $\scA$ is the set of admissible measures for which $\D$ is a depth function, in the sense of Definition \ref{def::depth}. Here,  $\sC\subset\cM_1(\rdd)$ denotes the set of centrally symmetric measures over $\rdd$ and $\sN\subset\cM_1(\rdd)$ denotes the set of angularly symmetric measures over $\rdd$ that are absolutely continuous with respect to the Lebesgue measure. }
\vspace{1em}
\begin{tabular}{@{}l|llllll@{}}
\toprule
 & $\HD$ & $\IDD$ & $\IRW$ & $\SMD$& $\SD$& $\MSD$ \\ \hline
$\VC(\sF)$ & $  O(d)$ & $O(d)$ & $O(d)$ & $O(d^2\log d)$& $O(d)$ & $O(d)$\\ 
 $K$    & 1 & 3 & 4 & $d+1$& $d$ & 1\\
 $L$    & $\sup_u \norm{f_u}$ & $3\sup_u \norm{f_u}$& $2\sup_u \norm{f_u}$& $\sup_u \norm{f_u}$& $2L'$&$2\sqrt{d}L'$\\
 $\scA$ & $\cM_1(\rdd)$ & $\sC$ & $\sC$ & $\sN$& $\emptyset$ (see Remark~\ref{rem::spat})& $\emptyset$ \\
\end{tabular}
\label{table::cond}
\end{table}
%%%%%
%%%%%%%%%%%%%%%%%%%%%%%
\begin{example}[Cauchy marginals]\label{exam::cauchy}
% Let $|u|=(|u_1|,\ldots,|u_d|)^\top$. 
Suppose that $\mu$ is the product probability measure constructed from independent Cauchy marginals with location parameter 0 and scale parameters $\sigma_1,\ldots,\sigma_d$ and we would like to use the exponential mechanism to privately estimate the halfspace median of $\mu$. 
In this case, $\mu$ is $d$-version symmetric\footnote{Recall that $\mu$ is $d$-version symmetric about zero if for any unit vector $u$, $X^\top u\eqd  a(u)Z$ where $X\sim \mu$ and $Z\sim\nu$ such that $Z\eqd-Z$ and $a(u)=a(-u)$. }, with $a(u)=\sum_{j=1}^d\sigma_j|u_j|$ where $u_j$ is the $j^{th}$ coordinate of $u$. 
A similar analysis as to that of Example \ref{exam::mvtg}, gives that halfspace depth (under this $\mu$) satisfies the conditions of Theorem~\ref{thm::main-dev-bound}. 
Note that in this setting, the Lipschitz constant scales with the dimension: $L=O(\sqrt{d})$. 
The next step is to compute $\alpha(t)$. Let $\bar{\sigma}=d^{-1}\sum_{j=1}^d\sigma_j$. 
Observe that when the scale parameters are constant in $d$, $\sup_u a(u)=~\sqrt{d}~\bar{\sigma}$. 
It follows that under the halfspace depth, $\alpha(t)= \frac{1}{\pi}\arctan\left(\frac{t}{\sqrt{d}\bar{\sigma}}\right)$. 
If $t<R$ for some large $R$, then we can write
$\alpha(t)\gtrsim t(\bar\sigma\sqrt{d}(1+R^2/d\bar\sigma^2))^{-1}.$
% $\alpha(t)\approx \pi^{-1} t\sqrt{d}\bar{\sigma}/(d\bar{\sigma}^2+t).$
If we use a Gaussian prior with unit scale,\footnote{We don't lose generality here; we can arbitrarily shrink the scale of the $\mu$.} then, omitting logarithmic terms, the deviations bound becomes
\begin{align*}
     \normm{\tilde\theta_n}&\lesssim  \frac{ d\bar{\sigma}}{\sqrt{n}} \vee \frac{ d^{1/2}  \norm{\theta_\pi}^2\vee  d^{3/2}
 }{\epsilon n }.
\end{align*}
% \begin{align*}
%      \normm{\theta_0-\tilde\theta_n}&\lesssim c\left( \sqrt{\frac{ d^{2}\bar{\sigma}^2\log d}{n}} \vee \frac{ d^{1/2}\left[  \norm{\theta_\pi}^2+  d\log\left(n\epsilon d  \right)\right]
%  }{\epsilon n }\right).
% \end{align*}
% \begin{align*}
%      n&\geq C\left( \frac{ d^{2}\bar{\sigma}^2\log d}{t^2} \vee \frac{ d^{1/2}\left[  \norm{\theta_\pi}^2+  d\log\left(\frac{\bar{\sigma}}{t} d  \right)\right]
%  }{\epsilon t}\right).
% \end{align*}
%%%%% OLD
\kr{Here, observe that the deviations are a factor of $\sqrt{d}$ larger than in the Gaussian setting. 
The increase in the sampling error term is due to the fact that if $X$ follows the poly-Cauchy distribution, then there is a direction $u$ where $X^\top u$ has scale proportional to $\sqrt{d}$. 
This causes the discrepancy function to be dependent on the dimension. For the poly-Cauchy measure, considering all directions is actually a drawback of the halfspace median. 
On the other hand, the coordinate-wise median does not consider the direction with scale proportional to $\sqrt{d}$, and does not have this issue. Indeed, when we apply Theorem~\ref{thm::main-dev-bound} to the univariate private median, the resulting deviations bound is  $O(\sqrt{ d\max_{i\in[d]}\sigma_i^2/n} \vee  [ d^{3/2}
( \norm{\theta_\pi}_\infty^2\vee 1) /\epsilon n] )$. 
Observe that the sampling error term is $O(\sqrt{d/n})$. 
(Note that any other optimal, univariate private median estimator would achieve the same deviations bound, e.g., \citep{zamos2020}.)
On the other hand, the cost of privacy term in the deviations bound is $O(d^{3/2}/n\epsilon)$ for both the halfspace median and coordinate-wise median. This represents the difficulty of estimating location when $\mu$ is high-dimensional and heavy-tailed.}
% One can immediately see the difficulty of estimating location when $\mu$ is high-dimensional and heavy-tailed: both the non-private and private terms are now a factor of $\sqrt{d}$ larger. 
% In this case, the number of samples required for non-private estimation $O(d^2)$ eclipses the cost of privacy, which is $O(d^{3/2}/\epsilon)$. 
% Thus, if we have enough samples to perform heavy tailed estimation, then we have enough samples to ensure privacy. 

\begin{figure}[t]
    \centering
    \includegraphics[width=2.8in]{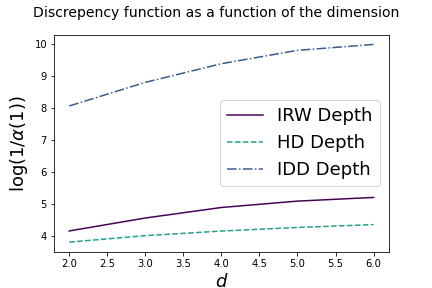}
    \caption{Logarithm of the inverse of the value of the discrepancy function at $t=1$, $\log(1/\alpha(1))$, as the dimension increases. Here, $\mu$ is made up of $d$ independent standard Cauchy marginals. 
    Notice that halfspace depth has the smallest values of $1/\alpha(1)$ which implies it has a smaller bound on the sample complexity. 
    This demonstrates the trade-off between accuracy and computability when choosing a depth function. 
    For the integrated depths, the value of the integrand is approximated via Monte Carlo simulation with 10,000 unit vectors.}
    \label{fig:alpha_cauchy}
\end{figure}
In the Gaussian setting, the discrepancy function is relatively similar between the halfspace depth to the integrated depths. 
However, in the Cauchy setting, the discrepancy functions differ between the halfspace depth, the integrated-rank weighted depth and the integrated dual depth. 
% A first observation is that $\alpha_{\IRW}(t)\leq \alpha_{\HD}(t)$. 
% How much bigger is $\alpha_{\IRW}(t)$ in this case? What about $\alpha_{\IDD}(t)$?
To see this, we must evaluate the integrals in Table \ref{table::alpha}. 
For the $\IRW$ this integral is given by
$$\alpha(t)= \frac{1}{\pi}\inf_{\norm{v}=1} \int\left| \arctan\left(\frac{t v^\top u}{\sum_{j=1}^d\sigma_j|u_j|}\right)\right|d\nu(u).$$
% \qquad \text{and}\qquad \alpha_{\IDD}(t)=1/4-\frac{1}{\pi}\inf_{\norm{v}=1} \int\arctan\left(\frac{t v^\top u}{\sum_{j=1}^d\sigma_j|u_j|}\right)\arctan\left(\frac{-t v^\top u}{\sum_{j=1}^d\sigma_j|u_j|}\right)d\nu(u).$$
We evaluate these integrals via Monte Carlo simulation with 10,000 unit vectors. 
Figure \ref{fig:alpha_cauchy} plots the logarithm of the inverted discrepancy function computed at $t=1$, i.e., $\log (1/\alpha(1))$, for the three depth functions as the dimension increases. 
Observe that the halfspace depth's discrepancy function decreases the slowest. 
Thus, the halfspace depth gives the smallest sample complexity for heavy tailed estimation when the dimension is large. 
The difference in accuracy between the halfspace depth and the integrated depths demonstrates the trade-off between computability and statistical accuracy. 
\end{example}

\subsection{Simulation}\label{sec::simu}

To complement our theoretical results, we demonstrate the cost of privacy through simulation. 
\kr{In order to simulate a private median, we use (non-private) discretized Langevin dynamics. 
We recognize that privacy guarantees are affected by the use of Markov chain Monte Carlo (MCMC) methods; the current implementation may not guarantee $\epsilon$-differential privacy. 
The purpose of the current implementation is to investigate the statistical accuracy of the methods. 
The computation of specific private medians via private MCMC is an interesting direction of future research.} 
For a discussion on the level of privacy guaranteed by MCMC, see \citep{mcmcps}. 
As a by-product of our work, we also provide a fast implementation of a modified version of the non-private integrated dual depth-based median. 
This implementation shows that this multivariate median can be computed quickly, even when the dimension is large. 
For example, we can compute the median of ten thousand 100-dimensional samples in less than one second on a personal computer.\footnote{The code was run on a desktop computer with an Intel i7-8700K 3.70GHz chipset and an Nvidia GTX 1660 super graphics card.} 
We call this modified version of the integrated dual depth the smoothed integrated dual depth, see Section~\ref{sec::sm-idd} for more details. 
Our code is available on GitHub. 
% at \citep{Ramsay2022}.

Aside from the private smoothed integrated dual depth median, we also computed the non-private smoothed integrated dual depth median and the private coin-press mean of \citet{Biswas2020}. 
When computing the coin-press mean, we used the existing GitHub implementation provided by the authors. 
The bounding ball for the coin-press algorithm had radius $10\sqrt{d}$ and was run for four iterations. 

\begin{figure}[t]
\begin{minipage}[c]{2.8in} 
    \centering
    \includegraphics[width=2.8in]{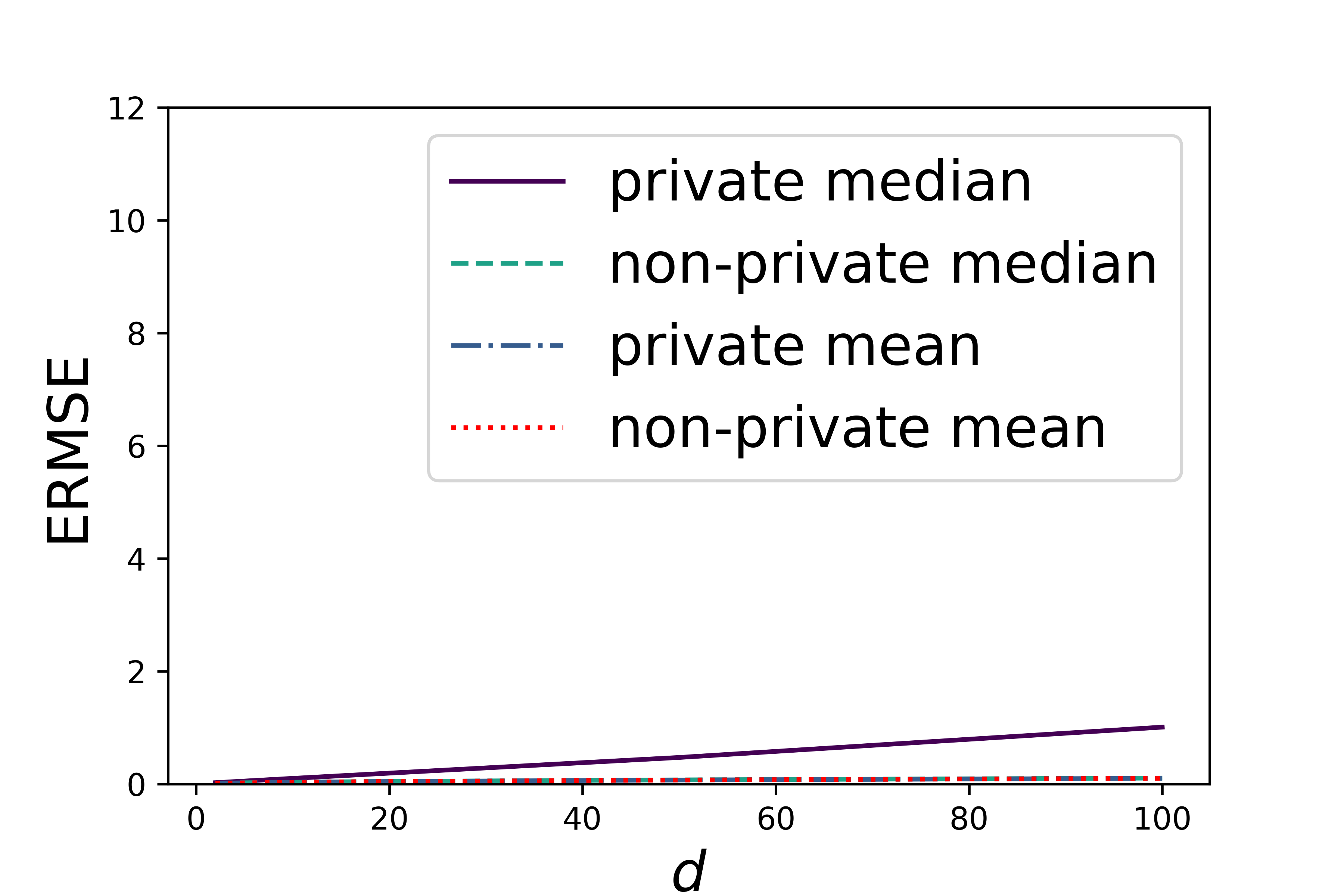}
\end{minipage}
\begin{minipage}[c]{2.8in} 
    \centering
    \includegraphics[width=2.8in]{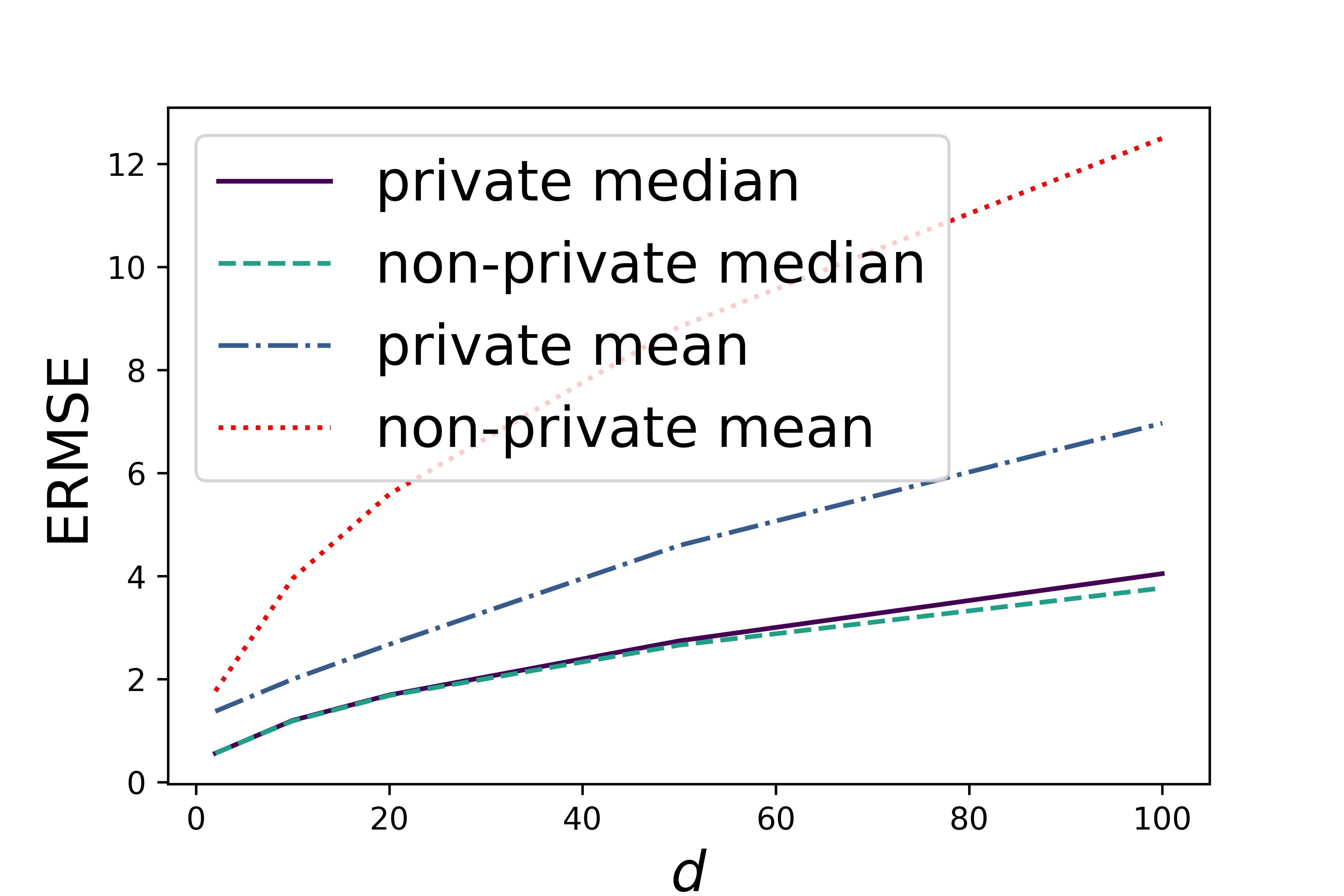}
\end{minipage}
\caption{Empirical root mean squared error (ERMSE) of the location estimators under (left) Gaussian data and (right) contaminated Gaussian data. In the uncontaminated setting, the ERMSE of the private median grows at a slightly faster rate than its non-private counterpart and the mean estimators. This is due to the cost of privacy. On the other hand, in the contaminated setting, the ERMSE of median estimators is much lower than that of the mean estimators. This demonstrates the usual trade-off between accuracy and robustness between the mean and the median. For a zoomed in version of the left plot, see Figure~\ref*{fig:dimension2} in the appendix.}
    \label{fig:dimension}
\end{figure}
We simulated fifty instances with $n=10,000$ over a range of dimensions from two to one hundred. 
The data were either generated from a standard Gaussian measure or contaminated standard Gaussian measure. 
In the contaminated model, 25\% of the observations had mean $(5,\ldots,5)$. 
The privacy parameter $\epsilon$ and the smoothing parameter $s$ were both fixed at 10 and $\pi=\cN(0,25dI)$. 

Figure \ref{fig:dimension} shows the empirical root mean squared error (ERMSE) for the different location estimators as the dimension increases. 
Notice that the private median ERMSE grows at a similar rate as that of the non-private ERMSE. 
The mean estimators perform very well in the uncontaminated setting, but poorly in the contaminated setting. 
The median estimators do not perform as well as the mean estimators in the standard Gaussian setting, but do not become corrupted in the contaminated setting. 
This demonstrates the usual trade-off between accuracy and robustness between the mean and the median.

%%%%%%%%%%%%%%%%%%%%%%%%%%%%%%%%%%%%%%%%%%

\section{An elementary concentration bound}\label{sec::concen--bound}
In this section, we introduce the underlying concentration bound used to produce Theorem~\ref{thm::main-dev-bound}, which we imagine will be useful more broadly. 
Our result shows that for a general objective function $\phi$, $\tilde{\theta}_n$ concentrates around $\theta_0 = \argmax_{\theta\in\rdd}\phi(\theta,\mu)$ with high probability. 
\kr{In particular, Theorem~\ref{thm::main-dev-bound} is specifically for privately estimating multivariate medians, given normal or uniform priors, meanwhile the concentration bound presented in this section applies more broadly to $\tilde{\theta}_n$ drawn from the more general \eqref{eqn::em-dens} as an estimate of $\theta_0 = \argmax_{\theta\in\rdd}\phi(\theta,\mu)$, under a general prior and a relaxed version of Condition~\ref{cond::phi_lip}.} 
The relaxed version of Condition \ref{cond::phi_lip} is stated below.
\begin{condition}\label{cond::phi_uc} 
The map $x\mapsto \phi(x,\mu)$ is $\pi$-a.s.\ uniformly continuous with modulus of continuity $\omega$.
\end{condition}
We now introduce some new important functions. 
Replacing the depth function $\D$ with $\phi$, we can define the generalized discrepancy function of the pair $(\phi,\mu)$ as 
$\alpha(t) = \phi(\theta_0,\mu) - \sup_{x\in B^c_t(E_{\phi,\mu})} \phi(x,\mu),
$
where $\theta_0\in E_{\phi,\mu}.$
% It is helpful to note here that all prior works on concentration of the exponential mechanism bound an empirical version of the discrepancy function as opposed to the distance between the estimator and the population value \citet{McSherry2007}.
The \emph{calibration function} of the triple $(\phi,\mu,\pi)$ is the function
\begin{equation}\label{eqn::calibration-func}
    \psi(\lambda) =\min_{r>0} [\lambda\cdot \omega(r) -\log \pi(B_r(E_{\phi,\mu}))]
\end{equation}
Note that $\psi(\lambda)$ is increasing and tends to $\infty$ as $\lambda\to\infty$. The \emph{rate function} of the prior is given by 
\begin{equation}
    I(t)=-\log\pi(B_t^c( E_{\phi,\mu})). 
\end{equation}
The concentration/entropy function of the prior is the rate at which the prior concentrates around the maximizers of the population risk. 
For large $t$, it is the rate of decay of the tails of $\pi$. 

The concentration bound is given in the following theorem, which quantifies the accuracy of $\tilde{\theta}_n$ drawn from \eqref{eqn::em-dens} as an estimate of $\theta_0$.
%%%%%%%%%%%%%%%%%%%%%%%%%%
% pi merge cor
\begin{theorem}\label{thm::main_result}
Suppose that for the pair $(\phi,\mu)$, Conditions \ref{cond::phi_um}, \ref{cond::phi-k-reg} and \ref{cond::phi_uc} hold for some $K>0$. Then, there are universal constants $c_1,c_2>0$ such that the following holds. For any $t,\beta>0$, we have that
\[
 \Pr\left(\max_{\theta_0\in E_{\phi,\mu}}\norm{\tilde{\theta}_n-\theta_0}\geq t\right)\leq  \left(\frac{c_1 n}{\VC(\sF)}\right)^{\VC(\sF)}e^{-c_2n \left[\frac{1}{\beta}I(t)\vee \alpha(t)\right]^2/K^2} + e^{-\beta\alpha(t)/2-I(t)/2+\psi(\beta)}.
    \] 
\end{theorem}
\begin{remark}[Comparison to \citet{McSherry2007} and \citet{Chaudhuri2012}]
It is helpful to compare our result to the widely used accuracy bound of \citet{McSherry2007}. The latter controls
the accuracy of the estimator $\tilde\theta_n$ as measured by the empirical objective value: it bounds the difference between the empirical objective value achieved by the exponential mechanism and the maximal empirical objective value.  More precisely, if we let $S_t=\{x\in \rdd\colon\ \phi(x,\hmu_n)-\sup_\theta \phi(\theta,\hmu_n))\geq 2t\}$, the result of \citet{McSherry2007} says that 
$$\Prr{\sup_\theta \phi(\theta,\hmu_n)-\phi(\tilde{\theta}_n,\hmu_n)>2t}\leq e^{-t\epsilon }/\pi(S_t).$$
As well, note that there is no analogous sample complexity bound because this approximation is not improving with the sample size. 
By contrast, \eqref{eqn::private-conc-bound} concerns the consistency of $\tilde\theta_n$ as an estimator, namely the distance between the estimator and any maximizer of the population objective, that is $\lVert\tilde{\theta}_n-\theta_0\rVert$. Furthermore, the quality of this concentration bound depends on $n$ and thus a sample complexity bound can be read off readily. 
\kr{In addition to \citet{McSherry2007}, \citet{Chaudhuri2012} also provide an accuracy bound for $M$-estimators drawn from the exponential mechanism in the univariate setting. They give the number of samples required for the error of the private estimator to be within $t>0$ of the error of the non-private estimator. 
% Specifically, they consider using the exponential mechanism to produce private, univariate $M$-estimators (of $\psi$-type) with a bounded, ``smooth'' $\psi$ function. They give the number of samples required for the error of the private estimator to be within $t$ of the error of the nonprivate estimator. 
Our bound can be applied to the same problem, and we also recover the requirement that $n\gtrsim 1/t\epsilon$.}
% The relevent term in the resulting sample complexity bound to compare to the bound of \citet{Chaudhuri2012} is the ``cost of privacy term'', that is, the term that contains the $\epsilon$. In particular, we recover the requirement that $n\gtrsim 1/t\epsilon$}
\end{remark}

\begin{proof}%[\emph{\textbf{Proof of Theorem~\ref{thm::main_result}}}]
% Let $\beta=\epsilon/2\GS_n$. 
We first prove that the regularity conditions imply that the empirical objective function concentrates around the population objective function. 
By Condition~\ref{cond::phi-k-reg}, there are some $K>0$ and $\sF$ with $\VC(\sF)<\infty$ such that $\phi$ is $(K,\sF)$-regular, i.e.,
\begin{equation}\label{eqn::epbound}
    \sup_x|\phi(x,\mu)-\phi(x,\hat{\mu}_n) |\leq K\sup_{g\in\sF}|\E{\hat{\mu}_n}{g(X)}-\E{\mu}{g(X)} |.
\end{equation}
Now, for a probability measure $Q$ over $\sB$, let $N(\tau,\sF,L_2(Q))$ be the $\tau$-covering number of $\sF$ with respect to the $L^2(Q)$-norm. 
By assumption, for any $g\in \sF$, $\norm{g}_\infty\leq 1$ and $\VC(\sF)<\infty$. This fact implies that there is some universal $c>0$ such that for any $0<\tau<1$, it holds that,
\begin{align*}
    \sup_Q N\left(\tau, \sF, L^2(Q)\right)%&\leq (c_2\VC(\sF))\left(\frac{2(3e)^2}{\tau}\right)^{2\VC(\sF)}
    \lesssim \VC(\sF)\left(\frac{c}{\tau}\right)^{2\VC(\sF)}.
\end{align*}
    See \citep[Theorem 7.12]{sen2018gentle}, or  \citep[Theorem 9.3]{kosorok2008introduction}. 
%See theorem 7.11 in the EP notes
A straightforward manipulation of Talagrand's inequality \cite[Theorem 1.1]{Talagrand1994} implies that there is some universal $c>0$ such that for all $n\geq 1$ and for all $t>0$, it holds that
\begin{equation*}
\Prr{\sqrt{n}\sup_{g\in\sF}\left|\frac{1}{n}\sum_{i=1}^n g(X_i)-\E{}{g(X)}\right|\geq t}\leq \left(\frac{cn}{\VC(\sF)}\right)^{\VC(\sF)}e^{-2t^2}.
\end{equation*}
For all $t>0$, it holds that
\begin{align}
    \Prr{\sup_{x\in\rdd}|\phi(x,\hat{\mu}_n)-\phi(x,\mu)| >t} %&\leq     \Prr{K\sup_{g\in\sF}\left|\frac{1}{n}\sum_{i=1}^n g(X_i)-\E{}{g(X)}\right|>t} 
    % &=\Prr{\sqrt{n}\sup_{g\in\sF}\left|\frac{1}{n}\sum_{i=1}^n g(X_i)-\E{}{g(X)}\right| >\sqrt{n}t/K} commented out%\nonumber\\ 
    \label{eqn::ob_con}
    &\leq \left(\frac{c n }{\VC(\sF)}\right)^{\VC(\sF)}e^{-2nt^2/K^2}.
\end{align}
where to go from the second inequality to the third  we multiplied both sides of the inequality by $\sqrt{n}$.

We can now use the concentration result \eqref{eqn::ob_con} to prove the rest of the theorem. For brevity, let $D_{n,t}=\{\max_{\theta_0\in E_{\phi,\mu}}\norm{\tilde{\theta}_n-\theta_0}>t\},$
and for any $y>0$, define the event $A_{n,y} = \{ \norm{\phi(\cdot,\hat{\mu}_n)-\phi(\cdot,\mu)}_\infty<y\}$. 
It follows from the law of total probability and \eqref{eqn::ob_con} that for any $y>0$ it holds that
\begin{align}
    \Pr\left(D_{n,t}\right)&=\Pr\left(D_{n,t}\cap A_{n,y}^c\right)+ \Pr\left(D_{n,t}\cap A_{n,y}\right)
\leq\Pr\left(A_{n,y}^c\right)+ \Pr\left(D_{n,t}\cap A_{n,y}\right)\nonumber\\
    \label{eqn::step0}
    &\leq   \left(\frac{c n }{K\VC(\sF)}\right)^{\VC(\sF)}e^{-2ny^2/K^2} +\Pr\left(D_{n,t}\cap A_{n,y}\right).
\end{align}

We now focus on the right-hand term above. 
%For a set $A\subset \rdd$, define $$B_t( A)=\{x\in\rdd\colon\ d(x,A)\leq t\},$$
%where $d(x,A)=\inf_{y\in A}\norm{x-y}$ is the usual point-to-set distance.
By definition, we have that
\begin{align}\label{eqn::step}
    \frac{1}{\beta}\log\Pr\left(D_{n,t}\cap A_{n,y}\right)&=\frac{1}{\beta}\log\int_{A_{n,y}}\frac{ \int_{B_t^c( E_{\phi,\mu})} \exp{\left(\beta\phi(x,\hat{\mu}_n)\right)}d\pi}{ \int_{\rdd} \exp{\left(\beta\phi(x,\hat{\mu}_n)\right)}d\pi}d\mu.
\end{align}
On $A_{n,y}$ it holds that
%\begin{equation}\label{eqn::step2}
    $\exp{\left(\beta\phi(x,\mu)\right)}\exp(-\beta y)\leq \exp{\left(\beta\phi(x,\hat{\mu}_n)\right)}\leq \exp{\left(\beta\phi(x,\mu)\right)}\exp{(\beta y)}.$
%\end{equation}
Applying this to the right-hand side of \eqref{eqn::step} results 
in
% \begin{align*}
%        \frac{1}{\beta}\log\Pr\left(D_{n,t}\cap A_{n,y}\right)&\leq \frac{1}{\beta}\log\int_{A_{n,y}}\frac{ \int_{B_t^c( E_{\phi,\mu})} \exp{\left(\beta\phi(x,\hat{\mu}_n)\right)}d\pi}{ \int_{\rdd} \exp{\left(\beta\phi(x,\hat{\mu}_n)\right)}d\pi}d\mu\\
%        &\leq  \frac{1}{\beta}\log\frac{ \int_{B_t^c( E_{\phi,\mu})} \exp{\left(\beta\phi(x,\mu)\right)}d\pi}{ \int_{\rdd} \exp{\left(\beta\phi(x,,\mu)\right)}d\pi}+\frac{1}{\beta}\log\Prr{A_{n,y}} +2y\\
%    &\leq    \frac{1}{\beta}\log\frac{ \int_{B_t^c( E_{\phi,\mu})} \exp{\left(\beta\phi(x,\mu)\right)}d\pi}{ \int_{\rdd} \exp{\left(\beta\phi(x,\mu)\right)}d\pi}+2y.
% \end{align*}
% Thus, we have that
\begin{align}\label{eqn::logp}
     \frac{1}{\beta}\log\Pr\left(D_{n,t}\cap A_{n,y}\right)&\leq \frac{1}{\beta}\log\frac{ \int_{B_t^c( E_{\phi,\mu})} \exp{\left(\beta\phi(x,\mu)\right)}d\pi}{ \int_{\rdd} \exp{\left(\beta\phi(x,\mu)\right)}d\pi}+2y.
\end{align}
% \textcolor{red}{This is still true for red version.}

We now bound $\E{\pi}{\exp{\left(\beta\phi(x,\mu)\right)}}$ below. 
For any $r>0$, Condition \ref{cond::phi_uc} implies that
\begin{align*}
    \frac{1}{\beta}\log \int_{\rdd} \exp{\left(\beta\phi(x,\mu)\right)}d\pi&\geq \frac{1}{\beta}\log \int_{B_r(E_{\phi,\mu})} \exp{\left(\beta\phi(x,\mu)\right)}d\pi\nonumber\\
    &=\phi(\theta_0,\mu)+\frac{1}{\beta}\log \int_{B_r(E_{\phi,\mu})} \exp{\left(\beta(\phi(x,\mu)-\phi(\theta_0,\mu))\right)}d\pi\nonumber\\
  %  \label{eqn::step3}
    &\geq\phi(\theta_0,\mu)-\omega(r) +\beta^{-1}\log\pi(B_r(E_{\phi,\mu})).
\end{align*}
Maximizing the right-hand side of the above, and recalling the definition of the calibration function, $\psi$, from \eqref{eqn::calibration-func}, yields 
\begin{equation*}
 \frac{1}{\beta}\log \int_{\rdd} \exp{\left(\beta\phi(x,\mu)\right)}d\pi
    \geq\phi(\theta_0,\mu)-\psi(\beta)/\beta.
\end{equation*}
% Define $\alpha(t)=\phi(\theta_0)-\sup_{x\in B_t^c( E_{\phi,\mu})} \left(\phi(x) \right)$. 
Plugging this lower bound into \eqref{eqn::logp} yields
\begin{align*}
   \frac{1}{\beta}\log\Pr\left(D_{n,t}\cap A_{n,y}\right)&\leq\frac{1}{\beta}\log \int_{B_t^c( E_{\phi,\mu})} \exp{\left(\beta\phi(x,\mu)\right)}d\pi -\phi(\theta_0,\mu)+\psi(\beta)/\beta+2y\\
    &\leq \sup_{x\in B_t^c( E_{\phi,\mu})} \phi(x) -\phi(\theta_0)+\frac{1}{\beta}\log\pi(B_t^c( E_{\phi,\mu}))+\psi(\beta)/\beta+2y\\
    &= -\alpha(t)- I(t)/\beta+ \psi(\beta)/\beta+2y,
\end{align*}
where the last line follows from the definitions of $\alpha(t)$ and $I(t)$. 
Rewriting the last inequality gives that
\begin{align}\label{eqn::step4}
\Pr\left(D_{n,t}\cap A_{n,y}\right)&\leq  \exp\left(-\beta\alpha(t)-I(t)+\psi(\beta)+2 \beta y\right).
\end{align}

% Let $c_4=c_1^*/16$. 
Now, setting $y=f(t,\beta)=\alpha(t)/2 \vee I(t)/2\beta$ and plugging  \eqref{eqn::step4} into \eqref{eqn::step0} results in
\begin{align}
    \Pr\left(\max_{\theta_0\in E_{\phi,\mu}}\norm{\tilde{\theta}_n-\theta_0}>t\right)&\leq  \left(\frac{c n}{\VC(\sF)}\right)^{\VC(\sF)}e^{-nf(t,\beta)^2/(8K^2)} +\Pr\left(D_{n,t}\cap A_{n,f(t,\beta)}\right)\nonumber\\
    \label{eqn::step5}
    &\leq \left(\frac{c n}{\VC(\sF)}\right)^{\VC(\sF)}e^{-nf(t,\beta)^2/8K^2}+ e^{-\beta\alpha(t)-I(t)+ \psi(\beta)+2\beta f(t,\beta)}.
\end{align}
Now, note that 
$\alpha(t)+I(t)/\beta-f(t,\beta)\geq \alpha(t)/2+ I(t)/2\beta.$
Plugging this inequality into \eqref{eqn::step5} yields the desired result.
% results in
% \begin{align*}
%    \Pr\left(\max_{\theta_0\in E_{\phi,\mu}}\norm{\tilde{\theta}_n-\theta_0}>t\right)&\leq \left(\frac{c }{\VC(\sF)}\right)^{\VC(\sF)}e^{-nf(t,\beta)^2/(8K^2)} + e^{-\beta\alpha(t)/2-I(t)/2+ \psi(\beta)}. 
% \end{align*}
% where the last line follows from the assumption that $\psi(\beta) < \alpha (t)/4\vee \rho(t,\beta)/4$. 
    % &\leq\left(\frac{c_2 n}{\VC(\sF)}\right)^{\VC(\sF)}e^{-c_4n f(t,\beta)^2} + e^{-\beta\alpha(t)/4-\beta\rho(t,\beta)/4},
\end{proof}

It is helpful to view the two terms in the above bounds separately. The first term measures the sample complexity of estimation and, in the absence of sampling considerations, recovers the concentration properties of classical (non-private) estimators. (In particular, the non-private bound can be obtained from the above by formally setting $\beta=\infty$.) The second term is novel and is a measure of the cost of sampling and encodes the trade-off between the prior, $\pi$, the parameter $\beta$, and the objective function $\alpha$.

% \subsection{Consistency of the exponential mechanism}
% It is useful to first introduce some notation??
% Note that this fact implies that one can instead set $\beta\leq \eps/2\GS_n(\phi)$ in order to ensure $\tilde{\theta}_n$ is $\eps$-differentially private, see 
Through taking $\beta=n\eps/2K$, Theorem~\ref{thm::main_result} immediately provides a consistency bound (and thereby a sample complexity bound) for the exponential mechanism:
\begin{multline}\label{eqn::private-conc-bound}
     \Pr\left(\max_{\theta_0\in E_{\phi,\mu}}\norm{\tilde{\theta}_n-\theta_0}\geq t\right)\leq \left(\frac{c_1 n}{\VC(\sF)}\right)^{\VC(\sF)}e^{-c_2n \left[\frac{1}{n\epsilon}I(t)\vee \alpha(t)/2K\right]^2}\\ + e^{-n\epsilon\alpha(t)/4K-I(t)/2+\psi(n\epsilon/2K)}.
\end{multline}

Let us pause here to interpret this result. The cost of consistent estimation is encoded in the first terms of \eqref{eqn::private-conc-bound} and, in particular, can be related to the non-private setting as well.
The cost of privacy is encoded in the second term in \eqref{eqn::private-conc-bound}.
Finally, we note that the second term in $\eqref{eqn::private-conc-bound}$ encodes a trade-off between an effect of the prior and the privacy parameter.

Equation \eqref{eqn::private-conc-bound} readily yields the sample complexity of an estimator drawn from the exponential mechanism.\footnote{To see this briefly, assume for simplicity that $\pi$ is the uniform measure on a hypercube whose side lengths are at most polynomial in $d$. 
In addition, assume that $\theta_0$ is well captured by this hypercube, i.e., $\theta_0$ is at least $e^{-d}$ far from any boundary of the cube. 
Suppose we want to estimate a maximizer of $\phi(\cdot,\mu)$ within some error level $t$ with probability at least $1-e^{-d}$. 
Lemma \ref*{lem::gen_sc_bnd} and Corollary~\ref{cor::main-sc-cor} give that it is sufficient for $\VC(\sF)$ to be polynomial in $d$ to ensure that the sample complexity is polynomial in $d$.} 
This yields Corollary~\ref{cor::main-sc-cor}, and, for instance, Lemma~\ref{lem::gen_devi_Bound} below. 
% We end this section with the following important remark.

\section{Proof of Theorem~\ref{thm::main-dev-bound}}\label{app::proofs}
We now prove Theorem~\ref{thm::main-dev-bound} with a series of lemmas.
Let us first note the following fact, whose proof is deferred to Appendix \ref*{app:tech-proofs}. 
\begin{lemma}\label{lem::deriv}
Let $a,b,c\in \re^+$. 
The function $h(x)=a x-b\log \left(c x /b\right)$ is increasing and positive for 
\begin{align*}
x\geq \frac{2b}{a}[\log\left(c/a\right)\vee 1].
\end{align*}
\end{lemma}
%%%%%%
%%%%%%%%%%%%%%%%%%%%%%%%
\begin{lemma}\label{lem::gen_devi_Bound}
Suppose that Conditions \ref{cond::phi_um} and \ref{cond::phi-k-reg} hold with $\phi=D$ for some $K>0$. 
In addition, suppose that the map $x\mapsto \D(x,\mu)$ is $L$-Lipschitz for some $L>0$.  
Then, there exists a universal constant $c>0$, such that for all $n,d\geq 1$ and all $0<\gamma<1$, with probability at least $1-\gamma$, it holds that
\begin{multline}\label{eqn::gen_devi_Bound}
  d(E_{\D,\mu},\tilde\theta_n)\lesssim   \alpha^{-1}\Bigg(cK\Bigg[\sqrt{\frac{\log(1/\gamma)\vee \VC(\sF)\log n}{n}}\\
 \bigvee\ \frac{\log\left(1/\gamma\right)\vee\left(1-\log\pi(B_{\frac{2K}{Ln\epsilon}}(E_{\phi,\mu}))\right)}{n\epsilon}\Bigg]\Bigg) . 
\end{multline}
\end{lemma}
\begin{proof}
Note that all of the conditions of Theorem~\ref{thm::main_result} hold, and so we can apply Theorem~\ref{thm::main_result} with $\beta=n\epsilon/2K$. 
% Proving the deviation bound then amounts to finding a lower bound on $t$ which ensures that
% \begin{align}\label{eqn::gen_devi_bound1}
%   \left(\frac{c_1 n}{\VC(\sF)}\right)^{\VC(\sF)}e^{-c_2n \left[2KI(t)/n\epsilon\vee \alpha(t)\right]^2/K^2} + e^{-n\epsilon\alpha(t)/2K-I(t)/2+\psi(n\epsilon/2K)}\leq \gamma. 
% \end{align}  
Proving the deviation bound then amounts to finding a lower bound on $t$ which ensures that:
\begin{align}\label{eqn::gen_devi_bound1}
  \left(\frac{c_1 n}{\VC(\sF)}\right)^{\VC(\sF)}e^{-c_2n \left[2KI(t)/n\epsilon\vee \alpha(t)\right]^2/K^2} + e^{-n\epsilon\alpha(t)/2K-I(t)/2+\psi(n\epsilon/2K)}\coloneqq I+II\leq \gamma. 
\end{align}
We will do this by finding $t$ such that both $I$ and $II$ are less than $\gamma/2$.
% We will focus on the two terms individually; we will find a lower bound on $t$ such that both 
% \begin{align}\label{eqn::gen_devi_target11}
%  \left(\frac{c_1 n}{\VC(\sF)}\right)^{\VC(\sF)}e^{-c_2n \left[2KI(t)/n\epsilon\vee \alpha(t)\right]^2/K^2}\leq \gamma/2,
% \end{align}
% and
% \begin{align}\label{eqn::gen_devi_target22}
% e^{-n\epsilon\alpha(t)/2K-I(t)/2+\psi(n\epsilon/2K)}\leq \gamma/2,
% \end{align}
% hold. 
% Together, \eqref{eqn::gen_devi_target11} and \eqref{eqn::gen_devi_target22} imply \eqref{eqn::gen_devi_Bound}. 
We start with finding $t$ such that $I\leq \gamma/2$. 
First, note that 
$$ \left(\frac{c_1 n}{\VC(\sF)}\right)^{\VC(\sF)}e^{-c_2n \left[2KI(t)/n\epsilon\vee \alpha(t)\right]^2/K^2}\leq  \left(\frac{c_1 n}{\VC(\sF)}\right)^{\VC(\sF)}e^{-c_2 n  \alpha(t)^2/K^2}.$$
If we can find a lower bound on $t$ which implies that
\begin{equation}\label{eqn::gen_devi_step00}
\left(\frac{c_1 n}{\VC(\sF)}\right)^{\VC(\sF)}\leq e^{c_2 n \alpha(t)^2/2K^2},
\end{equation}
then such a lower bound, together with the bound $t\gtrsim \alpha^{-1}\left(\sqrt{{K^2\log(2/\gamma)}/{n}}\right)$ yields that $I\leq \gamma/2$. 
It remains to find $t$ such that \eqref{eqn::gen_devi_step00} holds.
% \begin{equation}\label{eqn::gen_devi_step111}
%  t\gtrsim \alpha^{-1}\left(\sqrt{\frac{K^2\log(2/\gamma)}{n}}\right),
% \end{equation}

% then implies that
% \begin{equation*}%\label{eqn::gen_devi_step000}
%     \left(\frac{c_1 n}{\VC(\sF)}\right)^{\VC(\sF)}e^{-c_2n\alpha(t)^2/K^2 }\leq e^{-c_2n \alpha(t)^2/2K^2}.
% \end{equation*}

To this end, define the function $h(n)=c_2n \alpha(t)^2/2K^2-\VC(\sF)\log \left(c_1 n/\VC(\sF)\right).$
The condition \eqref{eqn::gen_devi_step00} is then equivalent to the condition $h(n)\geq 0$. 
Note that $h(n)$ is in the form of Lemma \ref{lem::deriv} with $a=c_2 \alpha(t)^2 /2K^2$, $b=\VC(\sF)$ and $c=c_1$. 
Applying Lemma \ref{lem::deriv} yields that $h(n)\geq 0$ provided that 
\begin{equation}\label{eqn::gen_devi_bnd00}
    %n\gtrsim K^2\frac{\VC(\sF)}{ \alpha(t)^2}.
     n\gtrsim K^2\frac{\VC(\sF)\log\left(\frac{K^2}{c\alpha(t)^2}\right)\vee 1}{ \alpha(t)^2},
\end{equation}
for some universal $c>0$. 
Now, provided that $\alpha^2(t)\gtrsim K^2/n$, \eqref{eqn::gen_devi_bnd00} is satisfied when 
\begin{equation*}
    \alpha(t)^2\gtrsim K^2{\VC(\sF)\log n}/{n}.
\end{equation*}
It follows that there exists a universal constant $c>0$ such that \eqref{eqn::gen_devi_step00} holds provided that
\begin{equation}\label{eqn::gen_devi_bnd22}
t\gtrsim \alpha^{-1}\left(cK\sqrt{\frac{\log(1/\gamma)\vee \VC(\sF)\log n}{n}}\right).
\end{equation}

We now turn to finding $t$ such that $II\leq \gamma/2$. 
The $II\leq \gamma/2$ is satisfied when
\begin{equation}\label{eqn::last_gen_devi}
\alpha(t)>4K\frac{\log\left(2/\gamma\right)\vee \psi(n\epsilon/2K)}{n\epsilon}>2K\frac{\log\left(2/\gamma\right)-I(t)/2+\psi(n\epsilon/2K)}{n\epsilon}. 
\end{equation}
% \begin{equation}\label{eqn::last_gen_devi}
% \alpha(t)>2K\frac{\log\left(2/\gamma\right)-I(t)/2+\psi(n\epsilon/2K)}{n\epsilon}>2K\frac{\log\left(2/\gamma\right)\vee \psi(n\epsilon/2K)}{n\epsilon}. 
% \end{equation}
It remains to simplify $\psi(n\epsilon/2K)$. 
To this end, given that by assumption $\D$ is $L$-Lipschitz, we have that $\omega(r)\leq L r$. Plugging this in to the definition of $\psi$ yields:
\begin{equation*}
\psi(n\epsilon/2K)\leq\min_{r>0}\left[\frac{n\epsilon}{2K}\cdot L\cdot r-\log\pi(B_{r}(E_{\phi,\mu}))\right].
\end{equation*}
Next, taking the specific choice $r=2K/Ln\epsilon$ yields 
    $\psi(n\epsilon/2K)\leq  1-\log\pi(B_{2K/Ln\epsilon}(E_{\phi,\mu})).$
Plugging this inequality into \eqref{eqn::last_gen_devi}, and combining the result with \eqref{eqn::gen_devi_bnd22} yields the desired result.
% that \eqref{eqn::gen_devi_bound1} is satisfied when 
% \begin{equation*}
%     t\gtrsim \alpha^{-1}\Bigg(cK\Bigg[\sqrt{\frac{\log(1/\gamma)\vee \VC(\sF)\log n}{n}}
%     \bigvee\  \frac{\log\left(1/\gamma\right)\vee\left(1-\log\pi(B_{2K/Ln\epsilon}(E_{\phi,\mu}))\right)}{n\epsilon}\Bigg]\Bigg)\ , 
% \end{equation*}
% for a universal constant $c>0$. 
\end{proof}
%%%%%%%%%%%%%%%%%%%%%%%%%%%%%%%%%%%%%%%%%
%%%%%%%%%%%%%%%%%%%%%%%%%%%%%%
\noindent We now note the following lemmas whose proofs are deferred to Appendix \ref*{app:tech-proofs}. 
\begin{lemma}\label{lem::norm_logpi}
If $\pi=\mathcal{N}(\theta_\pi, \sigma_\pi^2I)$ with $\sigma_\pi\geq 1/4$, then for all  $E\subset\rdd$, all $d>2$ and all $R\leq \sigma_\pi$ it holds that
\begin{equation}\label{sc_normal_1}
 -\log \pi(B_{R}(E))\lesssim \frac{d(E,\theta_\pi)^2}{\sigma_\pi^2}+ d\log\left(\frac{\sigma_\pi}{R}\vee d\right).
\end{equation}
\end{lemma}
%%%%%%%%%%%%%%%%%%%%%%%%%%%%%%%%%%%%%%%%%%%%%%%%%%%%%%%%%%%%%%%
\begin{lemma}\label{lem::cube_logpi}
Suppose that $\pi\propto \ind{x\in \cube_{R}(\theta_\pi)}$ where $E\subset\cube_{R}(\theta_\pi)$, then for all $E\subset\rdd$, $d> 0$ and all $r>0$ it holds that 
\begin{align*}
-\log \pi(B_{r}(E)) \lesssim  d\log\left(\frac{R}{d_{R,\theta_\pi}(E)\wedge r}\right).
\end{align*}
\end{lemma}
\noindent We can now prove Theorem~\ref{thm::main-dev-bound}. 
\begin{proof}[Proof of Theorem~\ref{thm::main-dev-bound}]
First note that by assumption, all the conditions of Lemma \ref{lem::gen_devi_Bound} are satisfied. 
To prove \eqref{devi_normal} and \eqref{devi_cube} we apply Lemma \ref{lem::norm_logpi} and Lemma \ref{lem::cube_logpi}, respectively, in conjunction with Lemma \ref{lem::gen_devi_Bound}. 
To apply Lemma \ref{lem::norm_logpi}, first, observe that together, the assumptions that $L>1$, $\sigma_\pi^2\geq 1/16$ and $n\geq 8K/\epsilon$ imply that $2K/Ln\epsilon\leq \sigma_\pi$. 
Applying Lemma \ref{lem::norm_logpi} with $R=2K/Ln\epsilon$ and $E=E_{\D,\mu}$ in conjunction with \eqref{eqn::gen_devi_Bound}, yields \eqref{devi_normal}. 
% that there exists a universal constant $c>0$ such that for all $t\geq 0$, all $d> 2$ and all $0<\gamma<1$, with probability at least $1-\gamma$, we have that
% \begin{multline*}
%     d(E_{\D,\mu},\tilde\theta_n)\lesssim   \alpha^{-1}\Bigg(cK\Bigg[\sqrt{\frac{\log(1/\gamma)\vee \VC(\sF)\log n}{n}}\\
%  \bigvee\ \frac{\log\left(1/\gamma\right)\vee\left(\frac{d(E,\theta_\pi)^2}{\sigma_\pi^2}+ d\log\left(\frac{\sigma_\pi Ln\epsilon}{K}\vee d\right)\right)}{n\epsilon}\Bigg]\Bigg) . 
% \end{multline*}
Applying Lemma \ref{lem::cube_logpi} with $r=2K/Ln\epsilon$ and $E=E_{\D,\mu}$ in conjunction with \eqref{eqn::gen_devi_Bound},  yields \eqref{devi_cube} as desired. 
% that there exists a universal constant $c>0$ such that for all $t\geq 0$, all $d\geq 1$ and all $0<\gamma<1$, with probability at least $1-\gamma$, we have that
% \begin{multline*}
%     d(E_{\D,\mu},\tilde\theta_n)\lesssim   \alpha^{-1}\Bigg(cK\Bigg[\sqrt{\frac{\log(1/\gamma)\vee \VC(\sF)\log n}{n}}\\
%  \bigvee\ \frac{\log\left(1/\gamma\right)\vee d\log\left(\frac{Rn}{nd_{R,\theta_\pi}(E_{\D,\mu})\wedge K/L}\right)}{n\epsilon}\Bigg]\Bigg) . 
% \end{multline*}
\end{proof}

% \subsection{Some well-known depth functions and their properties}\label{sec::depth-prop}
\subsection{Optimality of Theorem~\ref{thm::main-dev-bound}}
\begin{lemma}\label{lem::sharp}
The upper bound on the deviation of $\tilde{\theta}_n$ about $E_{\D,\mu}$ given in Theorem~\ref{thm::main-dev-bound} is sharp, up to logarithmic factors.
The upper bound on the sample complexity of $\tilde{\theta}_n$ given in Corollary~\ref{cor::main-sc-cor} is also sharp, up to logarithmic factors. 
% Precisely, suppose that $\D=\HD$, $\theta_0\in \cube_{R}(0)$, $\pi=\cube_{R}(0)$, that $\mu=\cN(\theta_0,\sigma^2I)$.
% In this case, for all $\sigma,R,\epsilon,t,d>0$ and all $0<\gamma<1$ it holds that $\norm{\theta_0-\tilde{\theta}_n}\leq t$ with probability at least $1-\gamma$ provided that
% $$n \geq c_1R\left[\frac{\log (1 / \gamma) \vee Rd}{(\Phi(R/\sigma)-1/2)^2t^2} \vee \frac{\log (1 / \gamma) \vee d \log \left(\frac{R^2}{t(\Phi(R/\sigma)-1/2)}\vee d\right)}{\epsilon t(\Phi(R/\sigma)-1/2)}\right].$$
\end{lemma}
\begin{proof}
% Recall that Theorem~6.5 of \citep{Kamath2018} is given by
% \begin{thm}
% For any $t\leq 1$ smaller than some absolute constant, any $(\epsilon,\delta)$-DP mechanism (for $\delta \leq O(\sqrt{d}R/n)$ which estimators a Gaussian measure (with mean $\theta_0\in \cube_{R}(0)$ and known covariance $\sigma^2I$)
% to accuracy $\leq t$ in total variation distance with probability $\geq 2/3$ requires $n~=~\Omega(\frac{d}{t\epsilon \log(dR)})$ samples.
% \end{thm}
We show that for a specific set of parameters, the upper bound given in Corollary~\ref{cor::main-sc-cor} matches the lower bound of Theorem~6.5 of \citep{Kamath2018} up to logarithmic factors in $d$ and $t^{-1}$. 
Specifically, set $\D=\HD$, $\theta_0\in \cube_{R}(0)$, $\pi=\cube_{R}(0)$, that $\mu=\cN(\theta_0,\sigma^2I)$. 
For $t\geq R$, the sample complexity is $1$. 
We then focus on the case where $t\leq R$. 

We now check the conditions of Corollary~\ref{cor::main-sc-cor} under these parameters. 
First, $\sup_x\D(x,\mu)=\D(\theta_0,\mu)$ and so Condition \ref{cond::phi_um} is satisfied. 
Next, $\HD$ satisfies Condition \ref{cond::phi-k-reg} with $K=1$ and $\VC(\sF)=d+2$, as shown in Lemma \ref*{lem::depth_kf}. 
Furthermore, we have that $\HD(\cdot,\mu)$ is $\sigma^{-1}$-Lipschitz. 
Therefore, the conditions of Corollary~\ref{cor::main-sc-cor} are satisfied. Plugging these values into Corollary~\ref{cor::main-sc-cor} and setting $\gamma=1/3$, as in Theorem~6.5 of \citep{Kamath2018}, yields
% $$  n\gtrsim \frac{\log(1/\gamma)\vee d\log\left(\frac{K}{\alpha(t)}\vee e\right)}{\alpha(t)^2}\vee \frac{\log\left(1/\gamma\right)\vee d\log\left(\frac{ R/\sigma}{d_{R,\theta_\pi}(\theta_0)\wedge \alpha(t)}\vee d\right)}{\alpha(t)\epsilon} .$$
% Setting $\gamma=1/3$, as in Theorem~6.5 of \citep{Kamath2018}, results in 
\begin{equation}\label{eqn::opt_int}
     n\gtrsim \frac{ d\log\left(\frac{K}{\alpha(t)}\vee e\right)}{\alpha(t)^2}\vee \frac{ d\log\left(\frac{ R/\sigma}{d_{R,\theta_\pi}(\theta_0)\wedge \alpha(t)}\vee d\right)}{\alpha(t)\epsilon} .
\end{equation}
Note that in this case, we have that $\alpha(t)=\Phi(t/\sigma)-1/2.$
Let $a(R)=(\Phi(R/\sigma)-1/2)/R$. 
Now, since $t\leq R$, concavity of $\Phi$ on $\re^+$ implies that $\alpha(t)\geq a(R)t$.
% \begin{equation}\label{eqn::concave}
%     \alpha(t)=\Phi(t/\sigma)-1/2\geq \left(\frac{\Phi(R/\sigma)-1/2}{R}\right)t=a(R)t.
% \end{equation}
Applying this fact in conjunction with \eqref{eqn::opt_int} results in
\begin{equation}\label{eqn::optimal_dev_bound}
    n\gtrsim \frac{ d\log\left(\frac{1}{a(R)t}\vee e\right)}{a(R)^2\cdot t^2}\vee \frac{ d\log\left(\frac{ R/\sigma}{d_{R,\theta_\pi}(\theta_0)\wedge a(R)t}\vee d\right)}{a(R)\cdot t\cdot\epsilon},
\end{equation}
which matches the bound given in Theorem~6.5 of \citep{Kamath2018} up to logarithmic factors in $d$ and $t^{-1}$. 
A simple manipulation of the terms in \eqref{eqn::optimal_dev_bound} gives an equivalent deviations bound, which matches the minimax lower bounds given by \citet{Cai2019}, up to logarithmic factors, if $\delta$ is taken to be $n^{-k}$ for some $k>0$.  
\end{proof}

\section{Inverse discrepancy function for small $t$}\label{sec::linear}
\kr{As stated earlier, when applying Theorem \ref{thm::main-dev-bound}, if $\alpha^{-1}$ is often at most linear, at least for small $t$, then this is very convenient technically, in that it simplifies the bounds given in Theorem \ref{thm::main-dev-bound}. 
This is because, in that case, we need not be concerned about inverting $\alpha$ directly. 
Indeed, if there exists $C_\mu>0$ such that $\alpha(t)\geq C_\mu t$ for small $t$, then this implies that the $\alpha^{-1}(t)$ in Theorem \ref{thm::main-dev-bound} can be replaced by $t/C_\mu$, thus, simplifying the bound greatly. 
The purpose of this section is to give conditions for there to exist $C_\mu>0$ such that $\alpha(t)\geq C_\mu t$ for small $t$.}
We first give general conditions under which $\alpha(t)\geq C_\mu t$ for an arbitrary depth function.
\begin{lemma}\label{lemm:gen_alpha_bound}
Suppose that $\sup_{x\in\rdd} \D(x,\mu)=\D(\theta_0,\mu)$, $\D(\cdot,\mu)$ is translation invariant, decreasing along rays, continuous and continuously differentiable on $B_r(\theta_0)$ for some $r>0$. 
In addition, suppose that there exists  $0<a<R<r$ such that $u_*=\argmax_{u\in\bS^{d-1}}\D(u\cdot t,\mu)$ is constant in $t$ for $a\leq t\leq R$. 
Then, if for $a\leq t\leq R$, it holds that  $$g^*(t)=u_{*}^\top\left[\frac{d}{dx}\D(x,\mu)\big|_{x=t\cdot u_{*}}\right]$$ is non-decreasing on $[a,R]$, we have that $\alpha(t)\geq t\alpha(R)/R$. 
\end{lemma}

\noindent Note that the proofs of results in this section are deferred to Appendix \ref*{app:tech-proofs}. 
The conditions of Lemma \ref{lemm:gen_alpha_bound} will hold when $\mu$ and $\D$ are such that the contours of the depth function are not changing in shape as we move away from the center. 
For example, the conditions of Lemma \ref{lemm:gen_alpha_bound} hold for the class of elliptical distributions and the integrated depths. 
We give a more general version of this result below, using Lemma \ref{lemm:gen_alpha_bound} to show a general bound for integrated depth functions. 
% Recall that a measure $\mu\in\cM_1(\rdd)$ is centrally symmetric about a point $x\in\rdd$ if $X-x\eqd x-X$ for $X\sim \mu$. 
In addition, recall that $\nu$ denotes the uniform measure on the $(d-1)$-dimensional unit sphere. 
\begin{lemma}\label{lem::alpha_linear_symm}
Suppose $\mu$ is centrally symmetric about $\theta_0$ and $\D(x,\mu)=\int_{\bS^{d-1}} g(x^\top u,\mu_u) d\nu$ where for all $u\in \bS^{d-1}$, we have that
\begin{enumerate}
    \item $g(ay+b,a\mu_u+b)=g(y,\mu_u)$,
    \item $g(\cdot,\mu_u)$ are bounded, continuous and continuously differentiable,
    \item $g(\theta_0^\top u,\mu_u)=0$ and $g(\theta_0^\top u-y,\mu_u)=g(\theta_0^\top u+y,\mu_u)$,
    \item there exists $R>0$ such that $g(\theta_0^\top u+y,\mu_u)$ is non-increasing for $0\leq y\leq R$ and
    \item there exists $v\in \bS^{d-1}$ such that $g(c(\theta_0+v)^\top u,\mu_u)\geq g(c(\theta_0+w)^\top u,\mu_u)$ for all $c>0$ and $w\in \bS^{d-1}$,
\end{enumerate}
then for $t\in[0,R]$ it holds that
\begin{equation*}
    \alpha(t) \geq t\alpha(R) /{R}.
\end{equation*}
\end{lemma}

Observe that Items 1-4 are satisfied for the smoothed integrated dual depth (Definition~\ref{def::sidd}) and the integrated-rank-weighted depth (Definition~\ref{def::irw}) provided the projected cumulative distribution functions $h(y,u)=\mu\left(X^\top u\leq y \right)$ are continuous and continuously differentiable in $y$ for all $u\in\bS^{d-1}$. Item 5 depends on the distribution, but is fairly general. For instance, it is easy to see that item 5 is satisfied by the class of elliptical distributions, or anytime $\mu$ is such that the depth function has nested, convex contours.

In the univariate setting, deviations bounds for the univariate median often only require a lower bound on the density of the population measure in a neighborhood of the median. 
We next show a result of this flavor, for the halfspace depth. 
Recall that a measure $\mu\in\cM_1(\rdd)$ is halfspace symmetric about $\theta_0\in\rdd$ if for every closed halfspace $H$ which contains $\theta_0$, it holds that $\mu(X\in H)\geq 1/2$ \citep{Zuo2000}. 
\begin{lemma}\label{lem::alpha_linear}
Suppose that $\D=\HD$, $\argmax_{x\in\rdd}\HD(x,\mu)=\theta_0$ and suppose that there is some $R>0$ such that for all $u\in\bS^{d-1}$, $\mu_u$ are halfspace symmetric and absolutely continuous. 
If there exists $C_\mu>0$ such that for all $u\in\bS^{d-1}$, $\inf_{\theta_0^\top u\pm R}f_{\mu_u}(x)\geq C_\mu$, then for all $\theta_0^\top u- R\leq t\leq\theta_0^\top u+ R$, we have that $\alpha(t) \geq tC_\mu.$
\end{lemma}
%%%%%%%%%%%%%%%%%%%
%%%%%%%%%%%%%%%%%%%
\section{Depth functions and their properties}\label{sec::depth}
% \end{proof}
For the convenience of the reader, we collect here some basic facts about depth functions. 
% We begin by recalling the definition of depth functions. (We work with a slightly weaker notion of depth than is common in the literature to unify notation. See Remark~\ref{rem::depth_def}) 
We first review several common depth functions in the literature and their properties. We then introduce a new depth function, the smoothed integrated dual depth, which has desirable properties from a computational perspective and is used in our simulations. For proofs of the results from this section, see Appendix~\ref*{sec::depth_proofs}. 
% \section{Well-known depth functions and their properties}\label{sec::depth-prop}
% We now define several depth functions discussed in this paper, after which we give a short summary of their properties.

% For $X\sim\mu$, define 
% \begin{equation}\label{eqn::proj_dist}
%    F(x,u,\mu)=\mu(\ind{X^\top u\leq x^\top u}) 
% \end{equation}
% and $F(x-,u,\mu)=\mu(\ind{X^\top u< x^\top u}).$ 
% Halfspace depth, otherwise known as Tukey depth \citep{Tukey1974}, is defined as follows:
% \begin{definition}[Halfspace depth] \label{def::hs}
% % Let $\bS^{d-1}= \{u\in \rdd\colon \ \norm{u}=1\}$ be the set of unit vectors in $\rdd$. 
% The halfspace depth $\HD$ of a point $x\in \rdd$ with respect to $\mu\in\cM_1(\rdd)$ is
% \begin{equation}\label{eqn::cdf_u}
%     \HD (x,\mu)=\inf_{\norm{u}=1} F(x,u,\mu).
% \end{equation}
% \end{definition}
% The halfspace depth of a point $x\in\rdd$ is the minimum probability mass contained in a closed halfspace containing $x$. 
One classical depth function is the simplicial depth \citep{Liu1988, liu1990}. 
Let $\Delta(x_1, \ldots,x_{d+1})$ be the simplex in $\rdd$ with vertices $x_1, \ldots,x_{d+1}$. 
%Simplicial depth is defined as follows:
\begin{definition}[Simplicial Depth] \label{def::smd}
Suppose that $Y_1, \ldots,Y_{d+1}$ are i.i.d.\ from $\mu~\in~\cM_1(\rdd)$. The simplicial depth of a point $x\in \rdd$ with respect to $\mu$ is
$$\SMD (x,\mu)= \Pr(\Delta(Y_1, \ldots,Y_{d+1})\ni x).$$
\end{definition}
\noindent 
The simplicial depth of a point $x\in\rdd$ is the probability that a random simplex defined by $d+1$ draws from $\mu$ contains $x$. 
The next depth function we introduce is spatial depth \citep{Vardi2000, Serfling2002}, which is maximized at the spatial median \citep{Vardi2000}. 
Define the spatial sign function some as $\sff(y) =  y/\norm{y}$  (with $\sff(0)=0$). 
We can then define the spatial rank function as $\E{\mu}{\sff(x-X)}$. 
The norm of the spatial rank of $x\in\rdd$ is a measure of the outlyingness of $x$ with respect to $\mu$. 
With this in mind, the spatial depth is defined as follows:
\begin{definition}
The spatial depth $\SD$ of a point $x\in \rdd$ with respect to  $\mu\in\cM_1(\rdd)$ is
$\SD (x,\mu)=1-\norm{\E{\mu}{\sff(x-X)}}.$
\end{definition}
We will see below that the spatial depth is $(K,\sF)$-regular, however, $K=O(d)$.  
Replacing the norm with its square eliminates this inconvenience, and still yields a depth function which is maximized at the spatial median. 
Therefore, we define the modified spatial depth as follows:
\begin{definition}
The modified spatial depth $\MSD$ of a point $x\in \rdd$ with respect to $\mu\in\cM_1(\rdd)$ is
$\MSD (x,\mu)=1-\norm{\E{\mu}{\sff(x-X)}}^2.$
\end{definition}
Lastly, we define the integrated depth functions. 
Here, the idea is to define the depth of a point $x\in\rdd$ as the average, univariate depth of $x^{\top} u$ over all directions $u$. 
The first integrated depth function\footnote{The first integrated multivariate depth function, to be precise. The first integrated depth was developed for functional data by \cite{Fraiman2001}.} is the integrated dual depth of \citet{Cuevas2009}, which, letting $\nu$ be the uniform measure on $\bS^{d-1}$, is defined as follows:
\begin{definition}[Integrated Dual Depth]\label{def::idd} 
The integrated dual depth  of a point $x\in \rdd$ with respect to $\mu\in\cM_1(\rdd)$ is
\begin{equation}\label{eqn::idd}
    \IDD (x,\mu)=\int_{\bS^{d-1}}F(x,u,\mu)\left(1-F(x,u,\mu)\right) d\nu(u).
\end{equation}
\end{definition}
The integrated dual depth of $x\in\rdd$ is the average univariate simplicial depth of $x^{\top} u$ over all directions $u\in \bS^{d-1}$. 
The integrated rank-weighted depth \citep{RAMSAY201951} replaces the simplicial depth in \eqref{eqn::idd} with a slightly modified version of halfspace depth. 
For $X\sim\mu$, define  $F(x-,u,\mu)=\mu(X^\top u< x^\top u).$ 
\begin{definition}[Integrated Rank-Weighted Depth]\label{def::irw} The integrated rank-weighted depth of a point $x\in \rdd$ with respect to $\mu\in\cM_1(\rdd)$ is
$\IRW (x,\mu)=2\E{\nu}{F(x,U,\mu)\wedge (1-F(x,-U,\mu))}.$
\end{definition}

We now turn to a brief discussion of these functions satisfy Definition~\ref{def::depth}. To this end, recall the following classes of measures. 
% Let $\sC\subset\cM_1(\rdd)$ denote the set of centrally symmetric measures over $\rdd$. 
Recall that a measure $\mu$ is \emph{angularly symmetric} about a point $\theta_0\in\rdd$ if $\frac{X-\theta_0}{\norm{X-\theta_0}}\eqd \frac{\theta_0-X}{\norm{X-\theta_0}}$ for $X\sim \mu$ \citep{liu1990}.
Let $\sN\subset\cM_1(\rdd)$ denote the set of angularly symmetric measures over $\rdd$ that are absolutely continuous with respect to the Lebesgue measure. 
The following proposition collects existing results in the literature: 
\begin{proposition}\label{prop::depth_prop}
The aforementioned depth functions satisfy the following:
\begin{itemize}
    \item Halfspace depth is a depth function with admissible set $\cM_1(\rdd)$. 
    \item Simplicial depth is a depth function with admissible set $\sN$. 
    \item Spatial depth and the modified spatial depth satisfy properties (1), (2) and (4) for all $\mu~\in~\cM_1(\rdd)$. 
    \item Integrated dual depth and integrated rank-weighted depth are depth functions with admissible set $\sC$. Furthermore, they satisfy properties (1), (2) and (4) for any $\mu\in\cM_1(\rdd)$. 
\end{itemize}
\end{proposition}
\begin{remark}\label{rem::spat}
Spatial depth does not satisfy Definition \ref{def::depth}, as property (3) fails. However, spatial depth satisfies spatial angle monotonicity and spatial rank monotonicity \cite[][see Section~4.3]{SerflingDepthFO}, which together provide a weaker alternative to property (3). 
%This allows one to interpret $\SD$ as a measure of depth. 
Simplicial depth and halfspace depth are affine invariant, which is a stronger version of property (1). 
\end{remark}
\begin{remark}
Both the halfspace depth and the simplicial depth function have an exact computational running time of $O(n^{d-1})$, see, e.g., \citep{2014arXiv1411.6927D,2015arXiv151204856A} and the references therein. 
By contrast, the spatial depth and the integrated depth functions are computable in high dimensions \citep{Chaudhuri1996,RAMSAY201951}. 
For instance, the integrated depth functions can be approximated easily via sampling, say, $m$ points from $\nu$. 
One can then approximate the sample depth value of a given point in $O(mnd)$ time \citep{RAMSAY201951}. \kr{Going further, it is easy to show via Talagrand's inequality \citep{Talagrand1994,sen2018gentle} that if one wants to approximate a sample depth value to within error $t$, with $>0.99$ probability then we need $m\gtrsim nd\log(1/t)/t^2$ unit vectors, which results in $O(n^2d^2\log(1/t)/t^2)$ time.}
\end{remark}

% see Section~4.3 of SerflingDepthFO for the discussion

The following theorem demonstrates that the depth functions defined in this section satisfy the conditions of Theorem~\ref{thm::main_result}. 
Let $\mu_u$ be the law of $X^{\top}u$ if $X\sim \mu$. 
\begin{theorem}\label{thm::depth_cond}
The aforementioned depth functions satisfy the following:
\begin{itemize}
    \item Conditions \ref{cond::phi_um} and \ref{cond::phi-k-reg}  hold for each of the depth functions defined in Section~\ref{sec::depth}. 
    \item Suppose that $\pi$ is absolutely continuous and that for all $u\in \bS^{d-1}$, the measure $\mu_u$ is absolutely continuous with density $f_u$ and $\sup_{\norm{u}=1}\norm{f_u}_\infty=L  <\infty$. 
Then there exists a universal constant $C>0$ such that for all $s\in (0,\infty]$, halfspace, simplicial, integrated rank-weighted and the smoothed integrated dual depth\footnote{The smoothed integrated dual depth is introduced in Section~\ref{sec::sm-idd} below and equals the integrated dual depth when $s=\infty$.} satisfy Condition \ref{cond::phi_uc} with $\omega(r)=C\cdot L\cdot r$. 
\item Suppose that $\mu$ has a bounded density and that $\sup_y\E{}{\norm{y-X}^{-1}}=L'\leq \infty$. Then modified spatial depth satisfies Condition \ref{cond::phi_uc} with $\omega(r)=2\sqrt{d}\cdot L'\cdot r$ and spatial depth satisfies Condition \ref{cond::phi_uc} with $\omega(r)=2\cdot L'\cdot r$.
%\item Condition \ref{cond::phi-k-reg} holds for each of the depth functions defined in Section~\ref{sec::depth}. 
\end{itemize}
\end{theorem}
\noindent The above results are also summarized in Table \ref{table::cond}.  
%%%%%%%%%%%%%%%%%%%%%%%%%%%%%%%%%%%%%%%
\subsection{The smoothed integrated dual depth}\label{sec::sm-idd}
One issue with many of the aforementioned depth functions is that their empirical versions contain indicator functions, which are non-smooth.  
% integrated dual depth of \citet{Cuevas2009} is defined as 
% \begin{equation*}
%     \IDD (x,\hat{\mu}_n,s)=\int_{\bS^{d-1}}\frac{1}{n}\sumn\ind{X_i^\top u\leq x^\top u}\left(1-\frac{1}{n}\sumn\ind{X_i^\top u\leq x^\top u}\right) d\nu(u),
% \end{equation*}
% which is clearly non-smooth. 
This fact is inconvenient from an optimization perspective, e.g., we cannot apply gradient descent to estimate the median. 
In addition, this non-smoothness implies that the empirical depth functions will possess flat contours. 
This can cause the estimated medians to perform poorly when the population measure has atoms near its center \citep{lalanne2023private}. 
Therefore, to resolve this issue, recall that
$\ind{X_i^\top u\leq x^\top u}\approx \expt\left(s(x-X)^\top u\right),$
for large $s$, where $\expt$ is the usual sigmoid function: $\expt(x)=(1+e^{-x})^{-1}$. 
This motivates the following definition of depth, which we use in our simulation study from Section~\ref{sec::simu}. 
Recall that $\nu$ denotes the uniform measure on the $(d-1)$-dimensional unit sphere. 
\begin{definition}[Smoothed Integrated Dual Depth] \label{def::sidd}
The smoothed integrated dual depth with smoothing parameter $s>0$ of a point $x\in \rdd$ with respect to $\mu\in\cM_1(\rdd)$ is
\begin{equation}\label{eqn::sidd}
    \IDD (x,\mu,s)=\int_{\bS^{d-1}}\E{}{\expt\left(s(x-X)^\top u\right)}\left(1-\E{}{\expt\left(s(x-X)^\top u\right)}\right) d\nu(u).
\end{equation}
\end{definition}
\noindent As $s\rightarrow\infty$, the smoothed integrated dual depth converges to the integrated dual depth. 
However, the smoothed integrated dual depth is convenient in that its empirical version is differentiable. 
Furthermore, the smoothed integrated dual depth is a depth with admissible set $\sC$.
%%%%%%%%%%%%%%%%%%%%%%%%%%%%%%%%%%%%%%%%%%%%%%%%%%%
\begin{proposition}\label{prop::idd_beta}
For any $s>0$, the smoothed integrated dual depth is a depth function with admissible set $\sC$. In addition, properties (1) and (4) hold for any $\mu\in \cM_1(\rdd)$. 
\end{proposition}
\noindent 
In the simulations presented in Section~\ref{sec::simu}, we approximate \eqref{eqn::sidd} with 
\begin{equation*}
    \widehat{\IDD} (x,\mu,s)=\frac{1}{M}\sum_{m=1}^M\E{}{\expt\left(s(x-X)^\top u_m\right)}\left(1-\E{}{\expt\left(s(x-X)^\top u_m\right)}\right) ,
\end{equation*}
where $u_1,\ldots, u_M$ are drawn i.i.d.\ from $\nu$. 
% Define the uniform approximation error as $\mathcal{E}_M=\sup_x\left|\widehat{\IDD} (x,\hat{\mu}_n,s)-\IDD(x,\hat{\mu}_n,s)\right|$
One can show that the number of unit vectors needed to maintain a fixed error level $t>0$ with high probability $1-\gamma$ is polynomial in $n$ and $d$. 
\begin{proposition}\label{prop::app_dep}
Let $u_1,\ldots,u_M$ be i.i.d.\ from $\nu$. Then for all $d,s,t>0$ and all $0<\gamma<1$, there exists a universal constant $c_1>0$ such that $\sup_x\left|\widehat{\IDD} (x,\hat{\mu}_n,s)-\IDD(x,\hat{\mu}_n,s)\right|\leq t,$
with probability $1-\gamma$, provided that 
$$M\geq c_1[\log(1/\gamma)\vee dn\log( {1}/{t}\vee e)]t^{-2}.$$
\end{proposition}
\kr{Observe that Proposition~\ref{prop::app_dep} is a sample complexity style result, where instead of bounding the number of samples needed for achieve some error tolerance $t>0$, we bound the number of simulated unit vectors needed to achieve error tolerance $t>0$. Specifically, the result says that we need, roughly, $O(nd/t^2)$ unit vectors for our approximation to be within $t$ of $\IDD(x,\hat{\mu}_n,s)$, with probability $>0.99$.}
% Condition \ref{cond::phi-k-reg} for the same pair $(K,\sF)$ as the integrated dual depth. 
%%%%%%%%%%%%%%%%%%%%%%%%%%%%%%%%%%%%%%%%%%%
Lastly, note that the median estimators are relatively insensitive to the choice of $s$. 
For $s\geq 10$, the specific value of $s$ appears to have minimal effect on the estimation error and convergence of the gradient descent algorithm, see Appendix \ref*{sec::s_param} for more details. We recommend choosing $s=100$. 
%%%%%%%%%%%%%%%%%%%%%%%%%%%%%%%%%%%%%%%%%%%%%
\section{Computing private depth values}\label{sec::computing-depths}
%%%%%%%%%%%%%%%%%%%%%%%%%%%%%%%%%%%%%%%%%%%%%%
The depth values themselves can be of interest. 
For example, they are used in visualization and a host of inference procedures \citep{Li2004, Chenouri2011}. 
The theoretical analysis from the previous sections yields simple differentially private methods for estimating depth values. 
If we want to compute the depth of a known point, then we can use an additive noise mechanism, such as the Laplace mechanism or the Gaussian mechanism \citep{Dwork2006}, i.e., we can just add noise that is calibrated to ensure differential privacy. 
\kr{For example, it follows from privacy of the Laplace mechanism \citep{Dwork2006} that if $\D$ is $(K,\cF)$-regular, then for all $x\in\rdd$ it holds that
$\tilde\D(x,\hat{\mu}_n)=\D(x,\hat{\mu}_n)+W_1K/n\epsilon,$
$W_1\sim\lap(1)$, is $\epsilon$-differentially private.}
% Let $\norm{\cdot}_1$ and $\norm{\cdot}_2$ denote the Euclidean $L_1$ and $L_2$ norms, respectively. 
% Define the global 1-sensitivity and the global 2-sensitivity of a statistic $T$ to be
% $$\GS_{n,1}(T)=\sup_{\substack{\sam{X}{n}\in \mathbf{D}_{n\times d},\\ \tilde{\mu}_n\in \widetilde{\cM}(\sam{X}{n})}}\norm{T(\hat{\mu}_n)-T(\tilde{\mu}_n)}_1\hspace{1em}\text{and} \hspace{1em}\GS_{n,2}(T)=\sup_{\substack{\sam{X}{n}\in \mathbf{D}_{n\times d},\\ \tilde{\mu}_n\in \widetilde{\cM}(\sam{X}{n})}}\norm{T(\hat{\mu}_n)-T(\tilde{\mu}_n)}_2.$$
% For example, for a non-private statistic $T(\hat{\mu}_n)$, the Laplace mechanism is defined as
% $\widetilde{T}(\hat{\mu}_n)=T(\hat{\mu}_n)+W_1\GS_n(T)/\epsilon,$
% where $W_1\sim\lap(1)$. 
% It is well known that $\widetilde{T}(\hat{\mu}_n)$ is $\epsilon$-differentially private \citep{Dwork2006}. 
Theorem~\ref{thm::main_result} then implies the concentration of private depth values generated from the Laplace mechanism.
%%%%%%%%%%%%%%%%%%
%%%%%%%%%%%%%%%%%%
\begin{corollary}\label{thm::gshs}
Suppose that $\D$ satisfies Condition \ref{cond::phi-k-reg}. 
For $x\in \rdd$ chosen independently of the data, 
$\widetilde{\D}(x,\hat{\mu}_n)=\D(x,\hat{\mu}_n)+W_1k/n\epsilon,$
is $\epsilon$-differentially private. 
In addition, there exists universal constants $c_1,c_2>0$ such that for all $t\geq 0$ it holds that
$$\Pr(|\widetilde{\D}(x,\hat{\mu}_n)-\D(x,\mu)|>t)\leq \left({c_1n}/{\VC(\sF)}\right)^{ \VC(\sF)}e^{-c_2n t^2/K^2}+e^{-n\epsilon t/2K}.$$
\end{corollary}
\noindent Corollary~\ref{thm::gshs} suggests taking $\epsilon\propto t/K$ when simulating private depth values. 
%%%%%%%%%%%%%%%%%%%%%%%%%%%%%%%%%%%%%%%%%%%%%%%%%%%%%
An obvious problem of interest is estimating the vector of depth values at the sample points, i.e., can we estimate 
%\begin{equation*}
  $\widehat{\mathbf{D}}(\hat{\mu}_n)=\left( \D(X_1,\hat{\mu}_n),\ldots,\D(X_n,\hat{\mu}_n)\right)$
privately? 
% $\norm{\cdot}_1$ and $\norm{\cdot}_2$ denote the Euclidean $L_1$ and $L_2$ norms, respectively. 
\kr{Define the global 1-sensitivity of $\widehat{\mathbf{D}}(\hat{\mu}_n)$ to be
$$\GS_{n}(\widehat{\mathbf{D}}(\hat{\mu}_n))=\sup_{\substack{\sam{X}{n}\in \mathbf{D}_{n\times d},\\ \tilde{\mu}_n\in \widetilde{\cM}(\sam{X}{n})}}\norm{\widehat{\mathbf{D}}(\hat{\mu}_n)-\widehat{\mathbf{D}}(\tilde{\mu}_n)}_1,$$
where $\norm{\cdot}_1$ denotes the $L_1$ norm. 
In general, $\GS_n(\widehat{\mathbf{D}}(\hat{\mu}_n))=O(1)$, see Lemma \ref*{lem::gs_o1}. 
As a result, if $\widetilde{\mathbf{D}}$ is the vector of private depth values generated from the Laplace mechanism then, $||\widetilde{\mathbf{D}}(\hat{\mu}_n)-\widehat{\mathbf{D}}(\hat\mu_n)||\gtrsim n/\epsilon$ with high probability; the level of noise is greater than that of the sampling error. Here, $\norm{\cdot}$ denotes the $L_2$ norm.}
This result is intuitive; these vectors reveal more information about the population as $n$ grows, which differs markedly from the single depth value case, where the amount of information received is fixed in $n$. 
In fact, for large $n$ the vector of depth values at the sample points contains a significant amount of information about $\mu$; the depth function can, under certain conditions, characterize the distribution of the input measure $\mu$ \citep[see][and the references therein]{Struyf1999, Nagy2018}. 
In order to release so much information about the population privately, we must inject non-negligible noise. 
We cannot, then, simply plug in the $n$ private sample depth values into an inference procedure and proceed. 
An interesting topic for future research is how to extend depth-based inference procedures to the private setting. 

% \kr{In summary, we have just demonstrated that computing a private vector of the sample depth values at the sampled points themselves involves adding a destructive amount of noise. 
% The issue is that we are using the sample points both as the argument in the depth function and in the construction of the sample depth function. 
% In order to mitigate this issue, we can instead sample say, $N$ points from the exponential mechanism for the median given in Section~\ref{sec::rob-med-est} using a very small $\epsilon$, say $\epsilon'$. Call these points $Z_1,\ldots,Z_N$. This will give points similar in nature to the original sample. 
% We can then compute the vector $\widetilde{\D}(Z_1,\hat{\mu}_n),\ldots,\widetilde{\D}(Z_N,\hat{\mu}_n)$
% We can then either apply the Laplace mechanism to those points, or compute the sample depth function from those private points. 
% , but eliminates the problem of using the sampled points themselves in the vector of depths calculation
% The performance of these procedures in specific inference settings is an interesting topic of future research.}
% We cannot, then, simply plug in the $n$ private sample depth values into an inference procedure and proceed. 
% Doing so results in 

%%%%%%%%%%%%%%%%%%%%%%%%%%%%%%%%%%%%%%

% Acknowledgements and Disclosure of Funding should go at the end, before appendices and references

\acks{The authors acknowledge Gautam Kamath for his helpful comments and discussion, which improved the paper. The authors acknowledge the support of the Natural Sciences and Engineering Research Council of Canada (NSERC). Cette recherche a \'et\'e financ\'ee par le Conseil de recherches en sciences naturelles et en g\'enie du Canada (CRSNG),  [DGECR-2023-00311, DGECR-2020-00199].}


% \bibliographystyle{abbrvnat}
\bibliography{main}

\newpage

\appendix
\section{Notation conventions}\label{app::notation}
\kr{We collect here some notation we will use throughout the paper. 
For sets $A,B$, we let $A \triangle B$ be the symmetric difference between $A$ and $B$. 
Given $v\in\rdd$, let $||v||$ denote the Euclidean norm, and for a function $f\colon \re\to\re$, let $||f||=\sup_{x\in\re} |f(x)|$.
For a set $A$ and $r>0$, let $B_r(A)=\{x\colon \min\limits_{y\in A}\norm{x-y}\leq r\}$. 
Let $\cube_{R}(y)$ be the $d$-dimensional cube of side-length $R$ centered at $y$ and let $d_{R,y}(x)$ denote the minimum distance from a point $x=(x_1,\,\dots,\,x_d)$ to a face of the cube $\cube_{R}(y)$: $d_{R,y}(x)=\min\limits_{1\leq i\leq d}|x_i-y_i\pm R/2|$.
Similarly, denote the minimum distance from a set $B\subset\rdd$ to a face of the cube $\cube_{R}(y)$ by $d_{R,y}(B)=\inf\limits_{x\in B}d_{R,y}(x)$ and let $d(x,B)=\inf\limits_{y\in B}\norm{x-y}$ denote the usual point-to-set distance. 
For two numbers $a,b\in\re$, we have that $a\wedge b=\min(a,b)$ and $a\vee b=\max(a,b)$ and we write $a\lesssim b$ ($a\gtrsim b$) if $a\leq C b$ ($a\geq C b$) for some universal constant $C>0$. 

For a space $S$, $\cM_1(S)$ denotes the space of probability measures on $S$. 
We say that a dataset of size $n\times d$ is a collection of $n$ points in $\rdd$, $\mathbb{X}_n=(X_\ell)_{\ell=1}^n$, with repetitions allowed. 
We assume that we have observed a dataset of size $n\times d$, which are realizations of $n$, independent and identically distributed random variables from a population measure $\mu\in \cM_1(\rdd)$. The empirical measure of these observations is denoted by $\hmu_n$. 
We let $\sB$ denote the space of Borel functions from $\rdd$ to $[0,1]$. For a family of functions, $\sF\subseteq\sB$ let $\VC(\sF)$ denote the Vapnik--Chervonenkis dimension of $\sF$. We sometimes use the following pseudometric on $\cM_1(\rdd)$,  $d_\sF(\mu,\nu) = \sup_{g\in\sF}|\int gd(\mu-\nu)|$, where $\mu,\nu\in\cM_1(\rdd)$. 

Next, we let $\cN(\mu,\Sigma)$ denote the Gaussian measure with mean $\mu$ and covariance matrix $\Sigma$. 
We let $\bS^{d-1}$ denote the $(d-1)$-dimensional unit sphere and we let $\nu$ be the uniform measure on $\bS^{d-1}$, where the dimension is clear from the context. 
We write $X\sim\mu$ if the random vector $X$ is drawn from $\mu\in \cM_1(\rdd)$. Furthermore, we let $\mu_u$ denote the law of $X^\top u$ for $u\in \bS^{d-1}$ and $f_u$ be the associated density function (when it exists). 
For two random vectors $X,Y$, we write $X\eqd Y$ if $X$ is equal in distribution to $Y$. 
For a random vector $X\sim \mu$, we write its expectation as $\E{\mu}{X}=\int Xd\mu$ and omit the measure, $\mu$, when it is clear from the context. 
Recall that a measure $\mu\in\cM_1(\rdd)$ is centrally symmetric about a point $x\in\rdd$ if $X-x\eqd x-X$ for $X\sim \mu$ and a measure $\mu$ is angularly symmetric about a point $\theta_0\in\rdd$ if $\frac{X-\theta_0}{\norm{X-\theta_0}}\eqd \frac{\theta_0-X}{\norm{X-\theta_0}}$ for $X\sim \mu$ \citep{liu1990}.
We let $\sC\subset\cM_1(\rdd)$ denote the set of centrally symmetric measures over $\rdd$ and we let $\sN\subset\cM_1(\rdd)$ denote the set of angularly symmetric measures over $\rdd$ that are absolutely continuous with respect to the Lebesgue measure.} 
%%%%%%%%%%%%%%%%%%%%%%%%%%%%%%%%%%%%%%%%%
\section{Technical proofs}\label{app:tech-proofs}
\begin{proof}[Proof of Corollary \ref*{cor::main-sc-cor}]
First note that by assumption, all of the conditions of Lemma \ref*{lem::gen_sc_bnd} are satisfied. 
% Applying Lemma \ref*{lem::gen_sc_bnd} yields that there exists a universal constant $c>0$ such that for all $t\geq 0$, all $d\geq 1$ and all $0<\gamma<1$, we have that $\norm{\tilde\theta_n-E_{\phi,\mu}}<t$ with probability at least $1-\gamma$ provided 
% \begin{equation}\label{repeat}
%     n\gtrsim \left[K^2\frac{\log(1/\gamma)\vee (\VC(\sF)\log(\frac{K}{c\alpha(t)})\vee 1)}{\alpha(t)^2}\right]\bigvee \left[    K\frac{\log\left(1/\gamma\right)\vee -\log\pi(B_{\alpha(t)/4L}(E_{\phi,\mu}))}{\epsilon\alpha(t)}\right].
% \end{equation}
% {eqn::gen_sc_bnd}
To prove \eqrefplain{sc_normal} and \eqrefplain{sc_cube} we apply Lemma \ref*{lem::norm_logpi} and Lemma \ref*{lem::cube_logpi}, respectively, in conjunction with Lemma \ref*{lem::gen_sc_bnd}. 
To apply Lemma \ref*{lem::norm_logpi}, first, observe that together, the assumptions that $L>1$, $\sigma_\pi^2\geq 1/16$ and $\D(x;\mu)\leq 1$ imply that $\alpha(t)/4L\sigma_\pi\leq 1$. 
Applying Lemma \ref*{lem::norm_logpi} with $R=\alpha(t)/4L$ and $E=E_{\D,\mu}$ in conjunction with \eqref{eqn::gen_sc_bnd} yields that there exists a universal constant $c>0$ such that for all $t\geq 0$, all $d> 2$ and all $0<\gamma<1$, we have that $d(E_{\D,\mu},\tilde\theta_n)<t$ with probability at least $1-\gamma$ provided that
{\small
$$n\gtrsim K^2\left[\frac{\log(1/\gamma)\vee (\VC(\sF)\log(\frac{K}{c\alpha(t)})\vee 1)}{\alpha(t)^2}\right] \bigvee \left[   K \frac{\log\left(1/\gamma\right)\vee \frac{d(E_{\D,\mu},\theta_\pi)^2}{\sigma^2_\pi} \vee d\log\left(\frac{L\sigma_\pi }{\alpha(t)}\vee  d  \right)}{\epsilon\alpha(t)}\right].$$}
To show \eqrefplain{sc_cube}, we apply Lemma \ref*{lem::cube_logpi} with $r=\alpha(t)/4L$ and $E=E_{\D,\mu}$ in conjunction with \eqref{eqn::gen_sc_bnd}, which yields that there exists a universal constant $c_1>0$ such that for all $t\geq 0$, all $d\geq 1$ and all $0<\gamma<1$, we have that $d(E_{\D,\mu},\tilde\theta_n)<t$ with probability at least $1-\gamma$ provided that
{\small
\begin{align*}
n\gtrsim K^2\left[\frac{\log(1/\gamma)\vee (\VC(\sF)\log(\frac{K}{c\alpha(t)})\vee 1)}{\alpha(t)^2}\right] \bigvee \left[  K\frac{\log\left(1/\gamma\right)\vee d\log\left(\frac{R}{d_{R,\theta_\pi}(E_{\D,\mu})\wedge \alpha(t)/L}\right)}{\alpha(t)\epsilon} \right]. &\qedhere
\end{align*}}
\end{proof}
%%%%%%%%%%%%%%%%%%%%%%%
\begin{proof}[Proof of Lemma~\ref*{lemm:gen_alpha_bound}]
Without loss of generality, assume that $\theta_0=0$. 
It suffices to show that $\alpha(t)$ is concave on $[a,R]$. 
This holds if and only if $g(t)=\sup_{\norm{x}\geq t}\D(x,\mu)$ is convex on $[0,R]$. 
By the decreasing along rays property and the fact that $\D$ is maximized at $0$, it holds that $g(t)=\sup_{\norm{u}=1}\D(t\cdot u,\mu)$. 
Next, let $u_{*,t}=\argmax_{\norm{u}=1}\D(t\cdot u,\mu)$. 
Observe that by assumption $u_{*,t}=u_{*}$, i.e., $u_{*,t}$ is constant in $t$. 
Therefore, $\sup_{\norm{u}=1}\D(t\cdot u,\mu)=\D(t\cdot u_{*},\mu)$. 
Next, it holds that 
\begin{equation*}
    \frac{d}{dt}g(t)=\frac{d}{dt}\D(t\cdot u_{*},\mu)=u_{*}^\top\left[\frac{d}{dx}\D(x,\mu)\big|_{x=t\cdot u_{*}}\right],
\end{equation*}
which is non-decreasing by assumption. 
This implies that $g(t)$ is convex on $[a,R]$. 
It follows that $\alpha(t)\geq t\alpha(R)/R.$
% Next, using an envelope theorem, and the fact that $\D$ is continuously differentiable, we have that 
% $$\nabla \sup_{\norm{u}=1}\D(t\cdot u,\mu)=\sup_{\norm{u}=1}\D(t\cdot u,\mu)$$
\end{proof}
\begin{proof}[Proof of Lemma~\ref*{lem::alpha_linear_symm}]
Again, without loss of generality, take $\theta_0=0$. 
We have that 
\begin{align*}
        \frac{d}{dx}\int_{\bS^{d-1}}g(x^\top u,\mu_u)d\nu&= \int_{\bS^{d-1}}\frac{d}{dx}g(x^\top u,\mu_u)d\nu= \int_{\bS^{d-1}}g^{(1)}(x^\top u,\mu_u)\ u\ d\nu.
\end{align*}
Now, 
\begin{align*}
    u_{*}^\top\left[\frac{d}{dx}\D(x,\mu)\big|_{x=t\cdot u_{*}}\right]&=\int_{\bS^{d-1}}u_{*}^\top u g^{(1)}(t\cdot u_{*}^\top u,\mu_u)d\nu=2\int_{\bS^{d-1}}|u_{*}^\top u| g^{(1)}(t\cdot |u_{*}^\top u|,\mu_u)d\nu,
\end{align*}
which is non-decreasing by the assumptions on $g$. 
Now, assumptions 1--5 and the preceding observation imply the conditions of Lemma \ref*{lemm:gen_alpha_bound} are satisfied, and the result follows. 
\end{proof}
\begin{proof}[Proof of Lemma~\ref*{lem::alpha_linear}]
Given that $\HD$ is translation invariant, there is no loss of generality in assuming that $\theta_0=0$. 
First, by definition and the decreasing along rays property of $\HD$, we have that 
\begin{equation}\label{eqn::alpha_dfn2}
        \alpha(t) = \HD(0,\mu) - \sup_{\norm{x}\geq  t} \HD(x,\mu)=1/2 - \sup_{\norm{x}= t} \inf_{\norm{u}=1}F(x,u,\mu).
\end{equation}
Next, we bound the term $\sup_{\norm{x}= t} \inf_{\norm{u}=1}F(x,u,\mu)$ above. 
To this end, we have that 
\begin{equation*}
        \sup_{\norm{x}=t} \inf_{\norm{u}=1}F(x,u,\mu)\leq \sup_{\norm{u}=1}\left(\frac{1}{2}-\left|\frac{1}{2}-h(t,u)\right|\right).
\end{equation*}
This inequality, in combination with \eqref{eqn::alpha_dfn2} and the mean value theorem yields that for all $0\leq t\leq R$, it holds
\begin{equation*}
        \alpha(t) \geq \left|\frac{1}{2}-h(t,u)\right|\geq tC_\mu. \qedhere
\end{equation*}
\end{proof}
%%%%%%%%%%%%%%%%%%%%%%%%%%%%%%%%
%%%%%%%%%%%%%%%%%%%%%%%%%%%%%%%%%%%%%%%
\begin{proof}[Proof of Lemma \ref*{lem::deriv}]
First, we use the derivative test to determine when the function $h(x)$ is increasing. 
Doing so yields that $h(x)$ is increasing for $x\geq b/a$. 
Now, it remains to find $x\geq b/a$ such that $h(x)\geq 0$. 
In other words, for $y_r=r b/a,$ we want to find $r\geq 1$ such that $h(y_r)\geq 0$. 
By definition,
\begin{align*}
    h(y_r)&=b\left(r-\log \left(cr/a\right)\right).
\end{align*}
It follows that $h(y_r)\geq 0$ is satisfied for $r$ such that $r-\log r\geq \log (c/a).$ 
Applying the inequality $x-\log x\geq x/2$ for $x\geq 1$ yields that $r-\log r\geq \log (c/a)$ is satisfied when 
\begin{align*}
  r\geq 2\log \left(c/a\right).
\end{align*}
Therefore, $h(x)$ is positive and increasing for
\begin{align*}
x\geq 2\frac{b}{a}\log \left(\frac{c}{a}\vee \sqrt{e}\right).%&\qedhere
\end{align*}
\end{proof}
\begin{proof}[Proof of Lemma~\ref*{lem::norm_logpi}]
% The next step is to lower bound the quantity $\log\pi(B_{C\alpha(t)}(E_{\phi,\mu}))$. 
% We first prove \eqref{sc_normal}. 
Let $R_*=R^2/\sigma^2_\pi$, $\theta_0\in E$ and let $W\sim \chi^2_d(\norm{\theta_0-\theta_\pi}^2/\sigma_\pi^2)$. 
We can lower bound $\pi(B_{R}(E))$ as follows:
\begin{equation}\label{eqn::sc_gp_1}
   \pi(B_{R}(E))\geq \pi(\norm{X-\theta_0}\leq R)=\Prr{W\leq R_*}.  
\end{equation}
Let $G(x,k,\lambda)$ be the cumulative distribution function for the non-central chi-squared measure with $k$ degrees of freedom and non-centrality parameter $\lambda$. 
Recall that for all $x,k,\lambda>0$, it holds that
\begin{align*}
    G(x,k,\lambda)&= e^{-\lambda^2}\sum_{j=0}^\infty \frac{(\lambda/2)^j}{j!} G(x,k+2j)\geq  e^{-\lambda^2} G(x,k,0).
\end{align*}
Therefore, 
\begin{equation}\label{eqn::sc_gp_2}
    \Prr{W\leq R_*}\geq e^{-\norm{\theta_0-\theta_\pi}^2/2\sigma_\pi^2}G(R_*,d,0). 
\end{equation}
%see https://en.wikipedia.org/wiki/Noncentral_chi-squared_distribution#Cumulative_distribution_function
We now lower bound $G(x,d,0)$ for $x\leq 1$. Let $\gamma$ be the lower incomplete gamma function. 
It follows from the properties of the gamma function that
\begin{align*}
    -\log G(x,d,0)&= -\log \gamma(x/2,d/2)+
    \log\Gamma(d/2)\lesssim -\log \gamma(x/2,d/2)+d\log d.
\end{align*}
Note that the assumption $d>2$ implies that $\sup_y \gamma(y,d/2)\geq 1$. 
This implies that $y^{d/2-1}e^{-y/2}$ is increasing on $(0,x)$. 
Thus, for any $0<r<1$ it holds that
\begin{align*}
     -\log G(x,d,0)&\lesssim -\log \int_{0}^{x}(y/2)^{d/2-1}e^{-y/2}/2\ dy
  +d\log d\\
    &\lesssim-\log (x-r)(r/2)^{d/2-1}/2+r/2+d\log d.
\end{align*}
Setting $r=x/2$ and using the fact that $x\leq 1$ results in
\begin{align}\label{eqn::chi_lb}
    -\log G(x,d,0)&\lesssim \frac{x}{4}+ \frac{d}{2}\log\left(\frac{4}{x}\right)+d\log d\lesssim  d\log\left(\frac{4}{x}\vee d\right),
    % &\lesssim d\log\left(\frac{4}{C^2\alpha_*(t)^2}\vee d\right),
\end{align}
Note that $R_*^2=R^2/\sigma^2_\pi\leq 1$ by assumption. 
Thus, applying \eqref{eqn::chi_lb} results in
\begin{equation}\label{eqn::chi_bnd}
 -\log G(R,d,0)\lesssim d\log\left(\frac{\sigma_\pi}{R}\vee d\right).
\end{equation}
Combining \eqref{eqn::sc_gp_1}, \eqref{eqn::sc_gp_2} and \eqref{eqn::chi_bnd} implies that
\begin{align*}
 -\log \pi(B_{R}(E))&\lesssim \frac{\norm{\theta_0-\theta_\pi}^2}{\sigma_\pi^2}+ d\log\left(\frac{\sigma_\pi}{R}\vee d\right).
\end{align*}
Now, note that this bound holds for all $\theta_0\in E$, therefore, it holds that 
\begin{align*}
 -\log \pi(B_{R}(E))&\lesssim \frac{d(E,\theta_\pi)^2}{\sigma_\pi^2}+ d\log\left(\frac{\sigma_\pi}{R}\vee d\right). \qedhere
\end{align*}
\end{proof}
%%%%%%%%%%%%%%%%%%%
%%%%%%%%%%%%%%%%%
\begin{proof}[Proof of Lemma~\ref*{lem::cube_logpi}]
First, note that for any ball $B_{r}(x)$ such that $B_{r}(x)\subset \cube_{R}(\theta_\pi)$, we have that 
\begin{align*}
    \pi(B_{r}(x))= \frac{r^{d}\pi^{d/2}}{R^d\Gamma(d/2+1)}\geq  \left(\frac{r}{R}\right)^{d}\left( d/2+1\right)^{d/2+1}\pi^{d/2} e^{-d/2}\geq \left(\frac{r}{R}\right)^{d}\left(\frac{\pi d}{2e}\right)^{d/2}.
\end{align*}
%to lower bound gamma functions recall https://functions.wolfram.com/GammaBetaErf/Gamma/29/
Taking the negated logarithm of both sides implies that 
\begin{equation}\label{eqn::subset_bound}
-\log \pi(B_{r}(x))\lesssim  d\log\left(\frac{R}{r}\right)-d\log\pi d/2+d\lesssim d\log\left(\frac{R}{r}\right).
\end{equation}

If there exists $\theta_0\in  E$ such that $B_{r}(\theta_0)\subset\cube_{R}(\theta_\pi)$ then \eqref{eqn::subset_bound} implies that 
\begin{equation}\label{eqn::case1}
-\log \pi(B_{r}(\theta_0))\lesssim d\log\left(\frac{R}{r}\right).
\end{equation}
If instead there exists $\theta_0\in  E$ such that $\cube_{R}(\theta_\pi)\subset B_{r}(\theta_0)$ then
\begin{equation}\label{eqn::case2}
-\log \pi(B_{r}(\theta_0))=0.
\end{equation}
Suppose finally that for all $\theta_0\in E$ it holds that both $B_{r}(\theta_0)\not\subset\cube_{R}(\theta_\pi)$ and $\cube_{R}(\theta_\pi)\not\subset B_{r}(\theta_0)$. 
Let $\theta_{0,j}$ and $\theta_{p,j}$ be the $j^{th}$ coordinate of $\theta_0$ and $\theta_\pi$ respectively. 
where
%%%%The radius is the closest wall of the cube, check all the coordinates of median and minimize over distance to theta_p pm k/2
\begin{align*}
    r=\min_{1\leq j\leq d}\left(|\theta_{0,j}-\theta_{p,j}-R/2| \wedge |\theta_{0,j}-\theta_{p,j}+R/2|\right)=d_{R,\theta_\pi}(\theta_0).
\end{align*}
Now, by assumption, there exists $\theta_0\in  E$ such that $\theta_0\in \cube_{R}$. 
For such $\theta_0\in \cube_{R}$, the definition of $d_{R,\theta_\pi}$ immediately implies that $B_{d_{R,\theta_\pi}(\theta_0)}(\theta_0)\subset \cube_{R}(\theta_\pi)$. 
Therefore, \eqref{eqn::subset_bound} implies that
\begin{equation}\label{eqn::case3}
-\log \pi(B_{r}(E))\leq -\log \pi(B_{r}(\theta_0)) \lesssim d\log\left(\frac{R}{d_{R,\theta_\pi}(\theta_0)}\right).
\end{equation}
It follows from \eqref{eqn::case1}, \eqref{eqn::case2} and \eqref{eqn::case3} that for all $d> 0$ and all $r$ it holds that 
\begin{equation}\label{eqn::any}
-\log \pi(B_{r}(E))\lesssim  d\log\left(\frac{R}{d_{R,\theta_\pi}(\theta_0)\wedge r}\right).
\end{equation}
Note that \eqref{eqn::any} holds for all $\theta_0\in E$, thus, for all $d> 0$ and all $r>0$ it holds that 
\begin{equation*}
-\log \pi(B_{r}(E))\lesssim d\log\left(\frac{R}{d_{R,\theta_\pi}(E)\wedge r}\right). \qedhere
\end{equation*}
\end{proof}
%%%%%%%%%%%%%%%%%%%
%%%%%%%%%%%%%%%%%%%%%%%%%%%%%%%%%%%%%%%%%%%%%%%%
\begin{lemma}\label{lem::gen_sc_bnd}
Suppose that Conditions \ref*{cond::phi_um} and \ref*{cond::phi-k-reg} hold with $\phi=D$ for some $K>0$. 
In addition, suppose that the map $x\mapsto \D(x,\mu)$ is $L$-Lipschitz for some $L>0$.  
There is a universal $c>0$ such that for all $t\geq 0$, all $d\geq 1$ and all $0<\gamma<1$, we have that $\norm{\tilde\theta_n-E_{\D,\mu}}<t$ with probability at least $1-\gamma$ provided that
\begin{equation}\label{eqn::gen_sc_bnd}
    n\gtrsim \left[K^2\frac{\log(1/\gamma)\vee (\VC(\sF)\log(\frac{K}{c\alpha(t)})\vee 1)}{\alpha(t)^2}\right]\bigvee \left[    K\frac{\log\left(1/\gamma\right)\vee -\log\pi(B_{\alpha(t)/4L}(E_{\D,\mu}))}{\epsilon\alpha(t)}\right]. 
\end{equation}
\end{lemma}
\begin{proof}
The proof is similar to that of Lemma~\ref{lem::gen_devi_Bound}. 
Note that all of the conditions of Theorem~\ref*{thm::main_result} hold, and so we can apply Theorem~\ref*{thm::main_result} with $\beta=n\epsilon/2K$. 
Proving the sample complexity bound then amounts to finding $n$ which ensures that $I+II\leq \gamma$, where $I,II$ are defined in the proof of Lemma~\ref{lem::gen_devi_Bound}. 
Now, following the same argument in the proof of Lemma~\ref{lem::gen_devi_Bound}, see the arguments between \eqref{eqn::gen_devi_step00} and \eqref{eqn::gen_devi_bnd22}, we have that $I\leq \gamma/2$ whenever
\begin{equation}\label{eqn::sc_step111}
  n\gtrsim \frac{K^2\log(2/\gamma)}{\alpha(t)^2} \bigvee K^2\frac{\log(1/\gamma)\vee \VC(\sF)\log(\frac{K}{c\alpha(t)})\vee 1}{\alpha(t)^2}.
\end{equation}
Furthermore, following \eqref{eqn::last_gen_devi}, $II\leq \gamma/2$ when
\begin{equation}\label{eqn::sc_step112}
     n>4K\frac{\log\left(2/\gamma\right)\vee\psi(n\epsilon/2K)}{\epsilon\alpha(t)},
     % n>4K\frac{\log\left(2/\gamma\right)\vee\psi(n\epsilon/2K)}{\epsilon\alpha(t)}>2K\frac{\log\left(2/\gamma\right)-I(t)/2+\psi(n\epsilon/2K)}{\epsilon\alpha(t)}. 
\end{equation}
from which it then remains to simplify the bound 
% $n>4K\psi(n\epsilon/2K)/\epsilon\alpha(t)$, where we show that 
% In order to simplify this bound, we require that
\begin{align}\label{eqn::goal0}
  \psi(n\epsilon/2K)\leq   n\epsilon\alpha(t)/4K.
\end{align}
Given that by assumption $\D$ is $L$-Lipschitz, we have that $\omega(r)\leq L r$. Plugging this in to the definition of $\psi$ from \eqrefplain{eqn::calibration-func}, we have that
$$\psi(n\epsilon/2K)\leq \min_{r>0}\left\{\frac{n\epsilon}{2K}\cdot L\cdot r-\log\pi(B_{r}(E_{\phi,\mu}))\right\}.$$
By taking the specific choice $r=\alpha(t)/4L$, we see that a sufficient condition for \eqref{eqn::goal0} to hold is
% \label{eqn::goal}
\begin{equation}
        \frac{n\epsilon\alpha(t)}{8K}-\log\pi(B_{\alpha(t)/4L}(E_{\D,\mu}))\leq      \frac{n\epsilon\alpha(t)}{4K},
\end{equation}
which is equivalent to
\begin{equation}\label{eqn::goal2}
n>\frac{-4K\log\pi(B_{\alpha(t)/4L}(E_{\D,\mu}))}{\epsilon\alpha(t)}.
\end{equation}
Combining \eqref{eqn::sc_step111}, \eqref{eqn::sc_step112} and \eqref{eqn::goal2} yields the desired result. 
% that \eqref{eqn::bound1} is satisfied when 
% \begin{align*}
%     n\gtrsim  \left[K^2\frac{\log(1/\gamma)\vee \VC(\sF)\log(\frac{K}{c\alpha(t)})\vee 1}{\alpha(t)^2}\right]\bigvee      \left[K\frac{\log\left(1/\gamma\right)\vee (-\log\pi(B_{\alpha(t)/4L}(E_{\D,\mu}))}{\epsilon\alpha(t)}\right]. 
% \end{align*}
\end{proof}
\section{On the smoothing parameter in the smoothed integrated dual depth}\label{sec::s_param}
%%%%%%%%%%%%%%%%%%%%%%
We executed a small simulation study to assess the effect of the smoothing parameter $s$ on the non-private estimation error of the smoothed integrated dual depth median, as well as the effect of $s$ on the convergence of the gradient descent algorithm. 
The simulation set up was the same as described in Section \ref*{sec::simu}. Figure \ref*{fig:s-ERMSE} shows the empirical root-mean-square error for different values of $s$ and the dimension $d$. We see that the choice of $s$ does not particularly depend on the dimension of the data. Furthermore, we see that choosing $s\geq 10$ produces the best results in terms balancing convergence and estimation error. 
Given that larger $s$ theoretically gives a closer approximation to the integrated dual depth, we recommend choosing $s=100$ in practice. 
\begin{figure}[t]
\centering
\begin{minipage}[c]{1.86in}
    \includegraphics[width=1.86in]{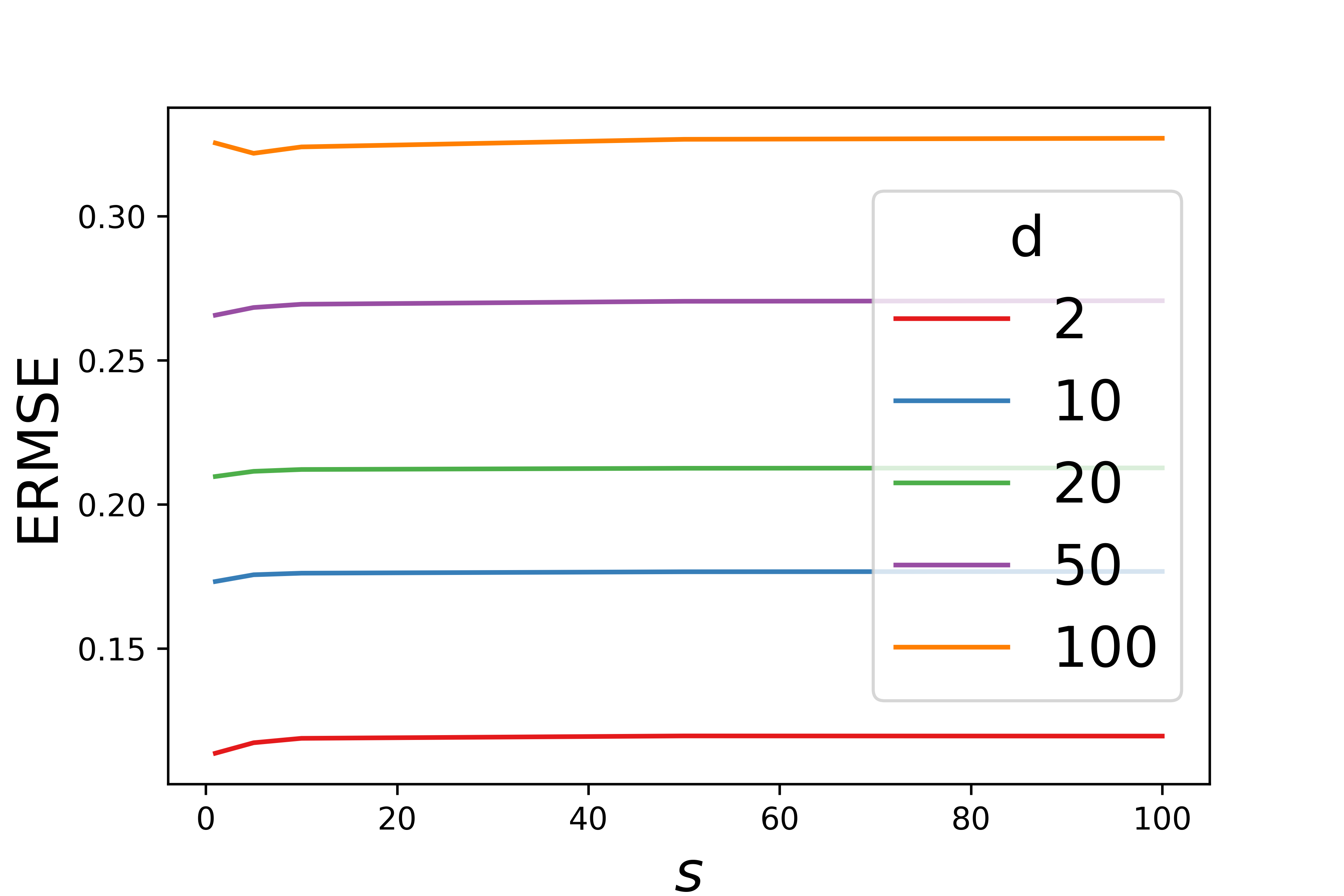}
\end{minipage}
\begin{minipage}[c]{1.86in}
    \includegraphics[width=1.86in]{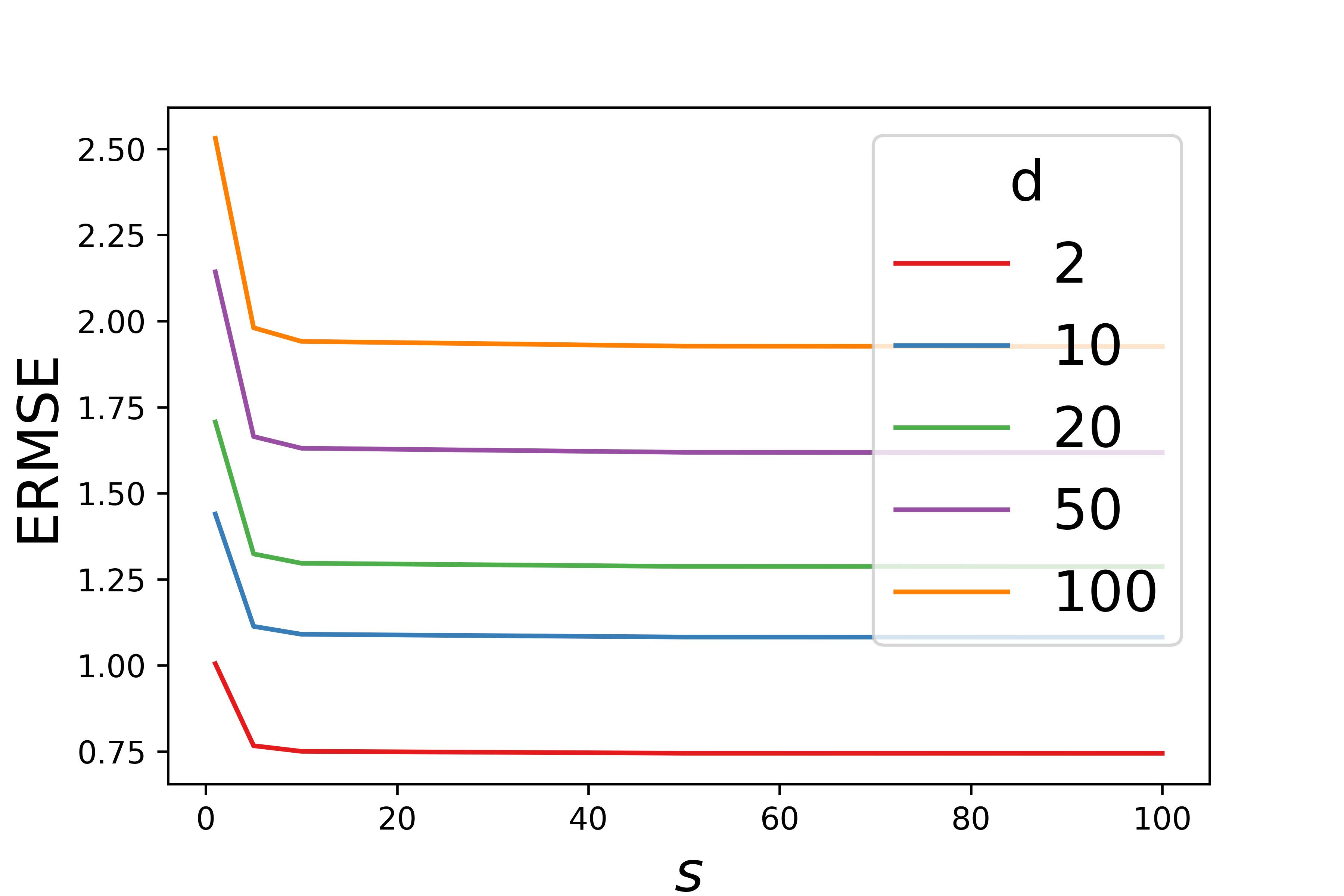}
\end{minipage}
\begin{minipage}[c]{1.86in}
\includegraphics[width=1.86in]{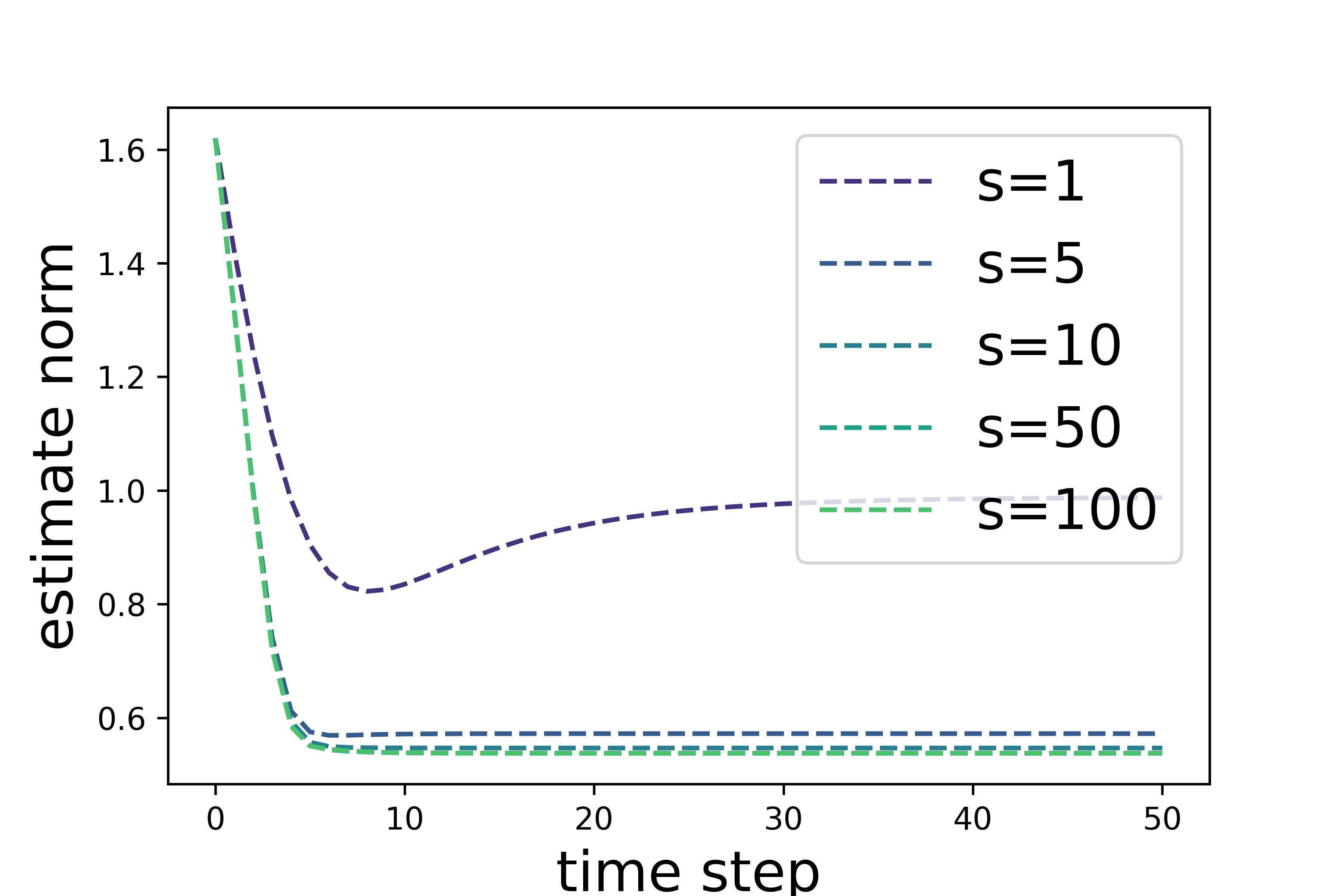}
\end{minipage}
    \caption{ERMSE of the non-private smoothed integrated dual depth median for different values of $s$ under Gaussian data (left) and contaminated Gaussian data (center) data. Convergence of one simulation run of the gradient descent algorithm in dimension $2$ is also presented (right). It can be seen that choosing $s$ large gives the lowest error and best convergence, though any value $s\geq 10$ produces similar results.}
    \label{fig:s-ERMSE}
\end{figure}
% \begin{figure}
% \centering
% \begin{minipage}[c]{0.45\textwidth}
% \end{minipage}
% \begin{minipage}[c]{0.45\textwidth}
%     \includegraphics{Plots/MSE_s_con.png}
% \end{minipage}
%     \caption{ERMSE of the non-private smoothed integrated dual depth median for different values of $s$. Choosing $s$ large gives the lowest error, though any value $s\geq 10$ produce similar results.}
%     \label{fig:s-ERMSE}
% \end{figure}
%%%%%%%%%%%%%%%%%%%%%%%%%%%%%%%%%%%%%%%%%%%%%%%%%%%%%%%%%%%%%%%%%%%%%%%%%%%%%%%%%%%%%%%%%%%%%%%%%%%%%%%%%%%%%%%%%%%%%%%%%%%%%%%%%%%%%%%%%%%%
%%%%%%%%%%%%%%%%%%%%%%%%%%%%%%
\section{Extra plot from simulations}
\begin{figure}[t]
    \centering
    \includegraphics[width=0.5\textwidth]{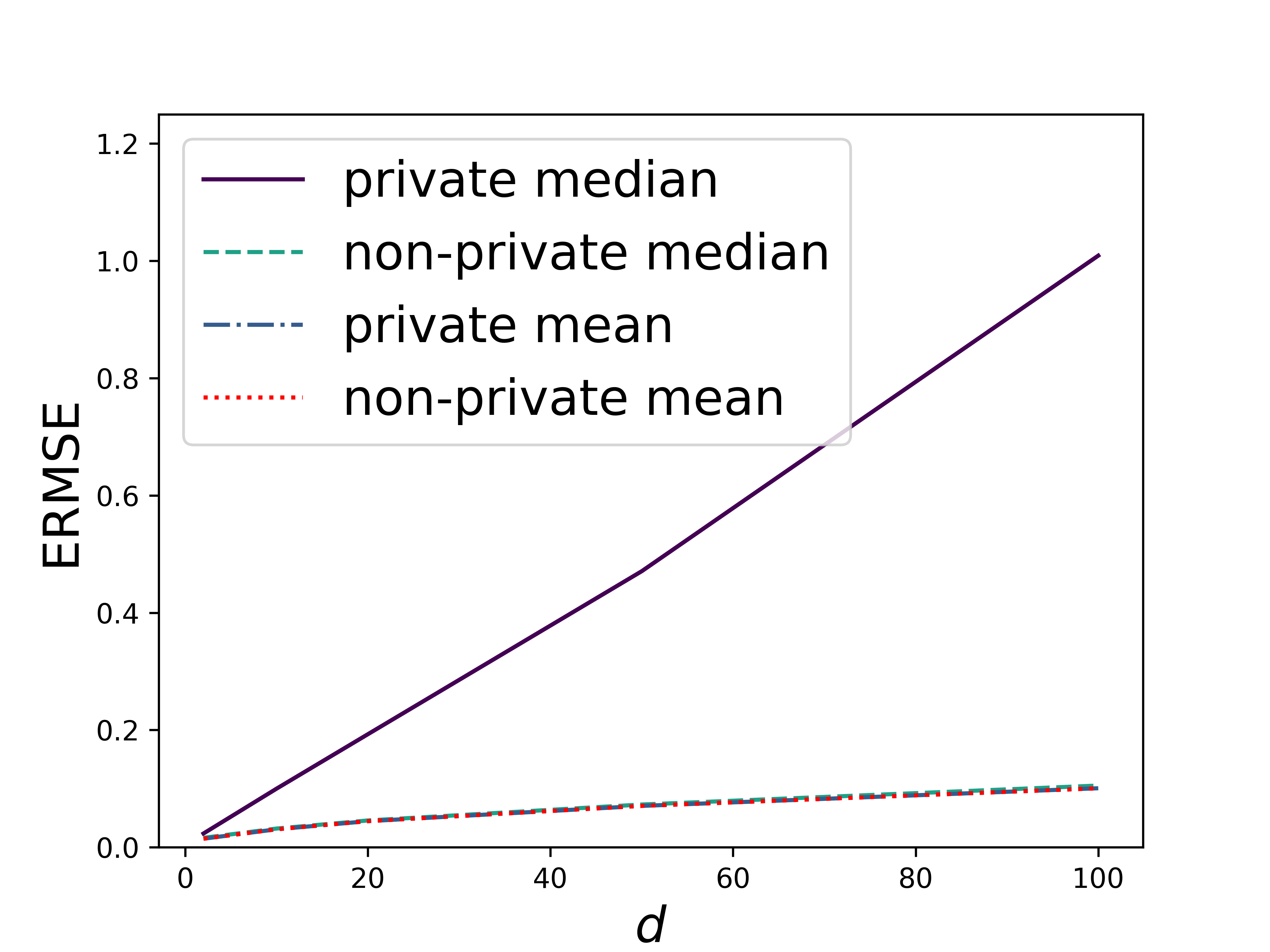}
\caption{Empirical root mean squared error (ERMSE) of the location estimators under Gaussian data. In the uncontaminated setting, the ERMSE of the private median grows at a slightly faster rate than its non-private counterpart and the mean estimators. This is due to the cost of privacy.}
    \label{fig:dimension2}
\end{figure}
\section{Auxiliary lemmas}
\begin{lemma}\label{lem::ub-gs-dp}
Suppose $U_n$ is an upper bound on $\GS_n(\phi)$.
Setting $\beta = \eps/2 U_n$ implies $\tilde{\theta}_n$ is $\eps$-differentially private. 
\end{lemma} 
\begin{proof}
From the definition of differential privacy, it is immediate that for any $0<\delta_1<\delta_2$, we have that $\delta_1$-differential privacy implies $\delta_2$-differential privacy. 
Now, it remains to show that $\tilde{\theta}_n$ is $\delta$-differentially private for some $\delta<\eps$.  
To see this, note that $\eps/2 U_n\leq \eps/2\GS_n(\phi)$, which implies that there exists $\delta<\eps$ such that $\eps/2 U_n= \delta/2\GS_n(\phi)$. 
This implies that $\tilde{\theta}_n$ is $\delta$-differentially private and the result follows.  
\end{proof}
\begin{lemma}\label{lem::GS}
If $\phi$ satisfies Condition \ref*{cond::phi-k-reg} then $\GS_n(\phi)\leq K/n$.
\end{lemma}
\begin{proof}
Definition \ref*{def::kf-reg} gives that 
\begin{multline*}
    \sup_{\substack{\sam{X}{n}\in \mathbf{D}_{n\times d},\\ \tilde{\mu}_n\in \widetilde{\cM}(\sam{X}{n})}}\sup_x|\phi(x,\hat{\mu}_n)-\phi(x,\tilde{\mu}_n)|\leq K\sup_{\substack{\sam{X}{n}\in \mathbf{D}_{n\times d},\\ \tilde{\mu}_n\in \widetilde{\cM}(\sam{X}{n})}}\sup_{g\in\sF}|\int g d(\hat{\mu}_n-\tilde{\mu}_n)|\\ \leq K\sup_{g\in\sF}\norm{g}/n=K/n. \qedhere
\end{multline*}
\end{proof}
\begin{lemma}\label{lem::incon}
Let $\D$ be the projection depth with location measure $\med$ and scale measure $\mad$, i.e., 
$$\D(x,\mu)=\frac{1}{1+O(x,\mu)}, \qquad \text{with}\qquad O(x,\mu)=\sup_{u\in \bS^{d-1}}|x^\top u-\med(\mu_u)|/\mad(\mu_u).$$
If $\epsilon=O(1)$ and $\pi$ is absolutely continuous, and $\pi\not\propto \ind{x\in E_{\D,\mu}}dx$, then $\GS_n(\D)=1$ and for a median drawn from the exponential mechanism based on $\D$ with $\beta=\epsilon/2\GS_n(\D)$, $\tilde{\theta}_n$, there exists $t>0$ such that
% $Q_{\mu,\beta}(d(E_{\mu,\phi},X)\geq t)\neq 0$, and
% \begin{equation*}
%    Q_{\hat{\mu}_n,\beta}(d(E_{\mu,\phi},\tilde{\theta}_n)\geq t)\rightarrow Q_{\mu,\beta}(d(E_{\mu,\phi},X)\geq t)\ \mu-a.s.\ .
% \end{equation*}
\begin{equation*}
    \limn\Prr{d(E_{\mu,\phi},\tilde{\theta}_n)\geq t}\neq 0.
\end{equation*}
% a median drawn from the exponential mechanism based on $\D$ with $\beta=\epsilon/2\GS_n(\D)$ is inconsistent $\mu$-a.s.\ . That is, 
\end{lemma}
\begin{proof}
The univariate sample median has infinite global sensitivity \citep{Brunel2020}, which implies that the projection depth with location measure $\med$ and scale measure $\mad$ has sensitivity 1.
It follows that $\beta=\eps/2$. 
We first show that there exists $t>0$ such that $Q_{\mu,\epsilon/2}(d(E_{\D,\mu},X)\geq t)>0$. 
Since $\pi$ is absolutely continuous and $\pi\not\propto \ind{x\in E_{\D,\mu}}dx$, there exists $A$ and $t>0$ such that $\inf_{x\in A}d(x,E_{\D,\mu})\geq t>0$ and $\pi(A)>0$. 
%Therefore $Q_{\mu,\epsilon/2}(d(E_{\mu,\phi},X)\geq t)> 0$. 
% \kr{[If $\pi$ is not absolutely continuous, then something weird could be happening if $E_{\mu,\phi}$ is open?] [AJ: reall? isnt this a set of argmax's? usu. for reasonable functions (def continuous but prob just upper semicontinuous?) this set is closed or even compact if you have a little more info.]} 
Since $ \D(x,\mu)\geq 0$, $Q_{\mu,\epsilon/2}(A)>0$, 
% \begin{align*}
%     Q_{\mu,\epsilon/2}(A)&=\int_A \exp(\eps \D(x,\mu)/2)d\pi>\pi(A)\neq 0.
% \end{align*}
so $Q_{\mu,\epsilon/2}(d(E_{\mu,\phi},X)\geq t)>0$ as well. 

Next, we show that $ \limn \Prr{d(E_{\D,\mu},\tilde{\theta}_n)\geq t}=Q_{\mu,\epsilon/2}(d(E_{\D,\mu},X)\geq t)$. 
To this end, we have that \citet{zuo2003} gives that $\D(x,\hat{\mu}_n)\rightarrow\D(x,\mu)\ \mu-a.s.$ as $n\rightarrow\infty$. 
Thus, $dQ_{\hat{\mu}_n,\epsilon/2}/d\pi\rightarrow dQ_{\mu,\epsilon/2}/d\pi\ \mu-a.s.$ as $n\rightarrow\infty$. 
This implies that, $\mu-a.s.$, $Q_{\hat{\mu}_n,\epsilon/2}\rightarrow Q_{\mu,\epsilon/2}$ in total variation  as $n\rightarrow\infty$ . 
% If $\pi$ is not proportional to $\ind{x\in E_{\D,\mu}}$, then there exists $t>0$ such that $Q_{\mu,\epsilon/2}(d(E_{\mu,\phi},X)\geq t)\neq 0$. 
% Thus, there exists $t>0$ such that
% $Q_{\mu,\epsilon/2}(d(E_{\mu,\phi},X)\geq t)\neq 0$, and, since $Q_{\hat{\mu}_n,\epsilon/2}\rightarrow Q_{\mu,\epsilon/2}\ \mu-a.s.$,
% \kr{[These are measures, in what topology is this convergence happening? Do you mean Weak convergence a.s.? Depending on how this is answerd the next step might need to be adjusted with some inequalities dpeneding on whether or not the set is open/closed. KR: I think it is total variation convergence, which implies weak convergence - fact check me though, my reasoning is that $\D(x,\hat{\mu}_n)\rightarrow\D(x,\mu)\ \mu-a.s.$ which implies that, by continuous mapping, the conditional density $dQ_{\hat{\mu}_n,\epsilon/2}/d\pi\rightarrow dQ_{\mu,\epsilon/2}/d\pi\ \mu-a.s.$ as $n\rightarrow\infty$, which implies TV convergence, $\mu-a.s.$. \ajedit{AJ: yeah seems right.}
% ]} 
We have that this implies
\begin{align*}
    \limn \Prr{d(E_{\mu,\phi},\tilde{\theta}_n)\geq t}&=\limn\E{}(Q_{\hat{\mu}_n,\epsilon/2}(d(E_{\mu,\phi},X)\geq t))\\
    &=\E{}(\limn Q_{\hat{\mu}_n,\epsilon/2}(d(E_{\mu,\phi},X)\geq t))\\
    &=Q_{\mu,\epsilon/2}(d(E_{\mu,\phi},X)\geq t))
    \neq 0,
\end{align*}
% \begin{equation*}
%    Q_{\hat{\mu}_n,\epsilon/2}(d(E_{\mu,\phi},\tilde{\theta}_n)\geq t)\rightarrow Q_{\mu,\epsilon/2}(d(E_{\mu,\phi},X)\geq t)\ \mu-a.s.\ ,
% \end{equation*}
% \kr{[Why is that postive/non-zero? - If $\pi\not\propto \ind{x\in E_{\D,\mu}}dx$, then there is some set $A$ such that $A\cap  E_{\D,\mu}=\emptyset$ and $\pi(A)>0$. The fact that $\pi$ is absolutely continuous yields that there exists $t>0$ such that $Q_{\mu,\epsilon/2}(d(E_{\mu,\phi},X)\geq t)\geq Q_{\mu,\epsilon/2}(A)$. 
% If $\pi$ is not absolutely continuous, then something weird could be happening if $E_{\mu,\phi}$ is open??? Also, $E_{\mu,\phi}$ is not guaranteed closed unless I put some restriction on $\mu$ to make $O$ continuous... so instead make $\pi$ ab con so the measure on the boundary has to be 0. 
% Now, because $0\leq \D(x,\mu)\leq 1$, we have that there exists $c>0$ such that
% \begin{align*}
%     Q_{\mu,\epsilon/2}(A)&=\int_A \exp(\eps \D(x,\mu)/2)d\pi>\pi(A)\neq 0.
% \end{align*}
% Therefore, there exists $t>0$ such that $Q_{\mu,\epsilon/2}(d(E_{\mu,\phi},X)\geq t)>0$.} 
which completes the proof.
% \begin{equation*}
%     \Prr{d(E_{\mu,\phi},\tilde{\theta}_n)\geq t}\rightarrow Q_{\mu,\beta}(d(E_{\mu,\phi},X)\geq t)\neq 0\ \mu-a.s.\ .  \qedhere
% \end{equation*}
\end{proof}
\begin{lemma}\label{lem::gs_o1}
Suppose that $\D$ is non-negative and satisfies properties (2) and (4) in Definition \ref*{def::depth}. 
Then there exists a universal constant $C>0$ such that  $\GS_n(\widehat{\mathbf{D}}(\hat{\mu}_n))\geq C$. 
Furthermore, if $\D$ satisfies Condition \ref*{cond::phi-k-reg}, then $\GS_n(\widehat{\mathbf{D}}(\hat{\mu}_n))=O(1)$. 
\end{lemma}
\begin{proof}
Consider $\tilde{\mu}_n$ to be centrally symmetric about $X_1$. 
Property (2) and non-negativity imply that $\D(X_1,\tilde{\mu}_n)=C>0$. 
Let $\tilde{\mu}_n(Y_1)$ be the empirical measure corresponding to $Y_1,\ldots, X_n$ with $Y_1\neq X_1$. 
Given that $\tilde{\mu}_n(Y_1)\in \widetilde{\cM}(\sam{X}{n})$, property (4) implies that
\begin{align*}
    \GS_n(\widehat{\mathbf{D}}(\hat{\mu}_n))&\geq \sup_{Y_1\in\rdd}| \D(Y_1,\tilde{\mu}_n(Y_1))- \D(X_1,\hat{\mu}_n)|\geq \D(X_1,\hat{\mu}_n)-\inf_{Y_1\in\rdd}\D(Y_1,\tilde{\mu}_n(Y_1))\geq C.
\end{align*}
Now, if $\D$ satisfies Condition \ref*{cond::phi-k-reg}, then Lemma \ref*{lem::GS} implies that 
\begin{align*}
    \GS_n(\widehat{\mathbf{D}}(\hat{\mu}_n))^2&\leq \frac{K^2}{n}+\sup_{\hat{\mu}_n,y}| \D(y,\hat{\mu}_n)- \D(X_1,\hat{\mu}_n)|^2\leq \frac{K^2}{n}+\norm{\D}=O(1). \qedhere
\end{align*}
\end{proof}
\begin{lemma}\label{lem::alpha_mg}
If $\mu=\cN(\theta_0,\Sigma)$ then for all $t>0$ and positive definite $\Sigma$, $\sup_{\norm{x}=t}\HD(x,\mu)= \Phi(-t/\sqrt{\lambda_1})$.
\end{lemma}
\begin{proof}
From the fact that $\HD$ is affine invariant, without loss of generality we can set $\theta_0=0$. 
From the fact that $\Phi$ is increasing, it holds that
\begin{multline*}
    \sup_{\norm{x}=t}\HD(x,\mu)= \sup_{\norm{x}=t}\inf_{\norm{u}=1}\Phi(\frac{x^\top u}{\sqrt{u'\Sigma u}})\\
    =1-\inf_{\norm{v}=1}\sup_{\norm{u}=1}\Phi(\frac{tv^\top u}{\sqrt{u'\Sigma u}})=1-\Phi(t\inf_{\norm{v}=1}\sup_{\norm{u}=1}\frac{v^\top u}{\sqrt{u'\Sigma u}}).
\end{multline*}
It suffices to compute $$\inf_{\norm{v}=1}\sup_{\norm{u}=1}\frac{(v^\top u)^2}{u'\Sigma u}.$$
First, note that 
\begin{align*}
    \inf_{\norm{v}=1}\sup_{\norm{u}=1}\frac{(v^\top u)^2}{u'\Sigma u}\geq \inf_{\norm{v}=1}\frac{(v^\top v)^2}{v'\Sigma v}=1/\lambda_1.
\end{align*}
Showing the reverse inequality, setting $v=e_1$ and setting $z=\Sigma^{1/2}u$ gives that 
{\small
\begin{align*}
\inf_{\norm{v}=1}\sup_{\norm{u}=1}\frac{(v^\top u)^2}{u'\Sigma u}
     &\leq  \sup_{\norm{u}=1}\frac{(e_1^\top u)^2}{u^\top \Sigma u}= \sup_{\norm{\Sigma^{-1/2}z}=1}\frac{(e_1^\top \Sigma^{-1/2}z)^2}{z^\top z}= \sup_{\norm{\Sigma^{-1/2}z}=1}(e_1^\top z/\norm{z})^2/\lambda_1 \leq\frac{1}{\lambda_1}.
\end{align*}}
Thus,
\begin{align*}
    \sup_{\norm{x}=t}\HD(x,\mu)&= 1-\Phi(t/\sqrt{\lambda_1})\geq \Phi(-t/\sqrt{\lambda_1}).
\end{align*}
\end{proof}
%%%%%%%%%%%%%%%%%%%%%%%%%%%
%%%%%%%%%%%%%%%%%%%%%%%%%%%%%%
\section{Proofs of the results from Section~\ref*{sec::depth}}\label{sec::depth_proofs}
We now prove the various results presented in Section~\ref*{sec::depth}. 
\begin{proof}[Proof of Proposition \ref*{prop::depth_prop}]
With the exception of the modified spatial depth, these results were shown by \citep{Zuo2000,Liu1988,Serfling2002,Cuevas2009, RAMSAY201951}, respectively. 
For the modified spatial depth, properties (1), (2) and (4) follow from the fact that properties (1), (2) and (4) for spatial depth. 
\end{proof}
\begin{proof}[Proof of Proposition \ref*{prop::idd_beta}]
We prove properties (1)-(4) in order. 
Let $A\in \re^{d\times d}$ be any orthogonal matrix and $b\in \rdd$. 
For any integrable function $h\colon\bS^{d-1}\rightarrow \re^+$ , it holds that
\begin{equation}\label{eqn::int_sphere}
    \int_{\bS^{d-1}} h(u) d\nu=\int_{\bS^{d-1}}  h(Au)d\nu.
\end{equation}
It also holds that
$$(Ax+b-AX-b)^\top u=A^\top(x-X)^\top u=A^\top u^\top(x-X).$$
Let $Y\sim A\mu+b$. 
Applying the above identity implies that
$$\E{}{\expt\left(s(Ax+b-Y)^\top u\right)}=\E{}{\expt\left(sA^{\top} u(x-X)\right)}.$$
The result follows from the definition of the smoothed integrated dual depth and \eqref{eqn::int_sphere}.

For property (2), by definition $\sup_x\IDD(x,\mu,s)\leq 1/4$. 
It remains to show $\IDD(\theta_0,\mu,s)= 1/4$. 
We will use the fact that 
\begin{equation}\label{eqn::prop2}
    (1-\expt(x-X))=\expt(X-x).
\end{equation}
% Recall that 
% \begin{align*}
%  \IDD (\theta;\mu,s)&=\int_{\bS^{d-1}}\left(\E{}{\expt\left(s(\theta-X)^\top u\right)}\right)\left(1-\E{}{\expt\left(s(\theta-X)^\top u\right)}\right) d\nu(u).
% \end{align*}
Now, consider the integrand in the definition of smoothed integrated dual depth at any fixed $u$. 
Central symmetry of $\mu$ and \eqref{eqn::prop2} yield that
\begin{align*}
\E{}{\expt\left(s(\theta_0-X)^\top u\right)}\left(1-\E{}{\expt\left(s(\theta_0-X)^\top u\right)}\right)&=\left(\E{}{\expt\left(s(X-\theta_0)^\top u\right)}\right)^2.
\end{align*}
Central symmetry of $\mu$ and \eqref{eqn::prop2} also yield that
\begin{align*}
\E{}{\expt\left(s(\theta_0-X)^\top u\right)}\left(1-\E{}{\expt\left(s(\theta_0-X)^\top u\right)}\right)&=\left(1-\E{}{\expt\left(s(X-\theta_0)^\top u\right)}\right)^2.
\end{align*}
These two equalities yield that 
\begin{align*}
\qquad \E{}{\expt\left(s(X-\theta_0)^\top u\right)}=1/2.
\end{align*}
% \begin{align*}
% \left(\E{}{\expt\left(s(X-\theta_0)^\top u\right)}\right)^2&=\left(1-\E{}{\expt\left(s(X-\theta_0)^\top u\right)}\right)^2 \qquad \implies\qquad \E{}{\expt\left(s(X-\theta_0)^\top u\right)}=1/2.
% \end{align*}
Given that this holds for any $u\in\bS^{d-1}$, it follows that $\IDD (\theta_0;\mu,s)= 1/4=\sup_x\IDD (x;\mu,s)$. 

For property (3), at any fixed $u$, central symmetry of $\mu$ and the fact that $\expt$ is increasing yields that either
\begin{align*}
\E{}{\expt\left(s(x-X)^\top u\right)}\leq \E{}{\expt\left(s(p\theta_0+(1-p)x-X)^\top u\right)}\leq 1/2,
\end{align*}
holds or 
\begin{align*}
1/2\leq \E{}{\expt\left(s(p\theta_0+(1-p)x-X)^\top u\right)}\leq \E{}{\expt\left(s(x-X)^\top u\right)},
\end{align*}
holds. 
The result follows immediately from the fact that $f(x)=x(1-x)$ is monotone on both $[0,1/2]$ and $[1/2,1]$. 

Lastly, we show property (4). For any fixed $u$, the fact that $\lim_{x\rightarrow\infty}\expt(x)\expt(-x)=0$ and the bounded convergence theorem imply that
\begin{align*}
\lim_{c\rightarrow\infty}  \IDD (cu;\mu,s)&=\lim_{c\rightarrow\infty}\int_{\bS^{d-1}}\left(\E{}{\expt\left(s(cu-X)^\top u\right)}\right)\left(1-\E{}{\expt\left(s(cu-X)^\top u\right)}\right) d\nu(u)\\	
&=\int_{\bS^{d-1}}\lim_{c\rightarrow\infty}\left(\E{}{\expt\left(s(cu-X)^\top u\right)}\right)\left(1-\E{}{\expt\left(s(cu-X)^\top u\right)}\right) d\nu(u)\\	
&=0. \qedhere
\end{align*} 
\end{proof}
%%%%%%%%%%%%%%%%%%%%%%%%%%%%%%%%%%%%%%%%
\begin{proof}[Proof of Proposition \ref*{prop::app_dep}]
Let 
$$Z=\sup_x\left|\widehat{\IDD} (x,\hat{\mu}_n,s)-\IDD(x,\hat{\mu}_n,s)\right|.$$
We first prove a concentration result for $Z$. 
To this end, let
\begin{align*}
G_{n}(x,u,s)&=\frac{1}{n}\sum_{i=1}^n\expt\left(s(x-X_i)^\top u\right)\\
W_1&= \sup_x\left|\frac{1}{M}\sum_{m=1}^M G_{n}(x,u_m,s)-\int_{\bS^{d-1}} G_{n}(x,u,s)d\nu(u)\right|\\
W_2&=\sup_x\left|\frac{1}{M}\sum_{m=1}^M G_{n}(x,u_m,s)^2-\int_{\bS^{d-1}} G_{n}(x,u,s)^2d\nu(u)\right|.
\end{align*}
From the fact that $Z\leq W_1+W_2$, we need only prove concentration results for $W_1$ and $W_2$. 
Let 
\begin{align*}
    \sG&=\left\{ \expt(y^\top u+b)/n\colon y\in\rdd,\ b\in \re\right\}. 
\end{align*}
Note that $\VC(\sG)\leq d+2$. 
Therefore, it follows from Theorem 7.12 in \citep{sen2018gentle} that there exists a universal constant $K>0$ such that for all $\epsilon\in(0,1)$, it holds that
$$ \sup_{Q}N(\epsilon/n,\sG,L_2(Q))\leq  \left(\frac{K}{\epsilon}\right)^{2d+4}.$$
% $$N(\epsilon/n,\sG,L_r(Q))\leq K(d+2)\left(\frac{8e}{\epsilon}\right)^{2d+4}\leq \left(\frac{K'}{\epsilon}\right)^{2d+4}.$$
Now, define $$\sumn\sG_i=\left\{ \frac{1}{n}\sumn\expt(y_i^\top u+b_i)\colon y_1,\ldots,y_n\in\rdd,\ ,b_1,\ldots,b_n\in \re\right\}.$$
It follows from Lemma 7.21 in \citep{sen2018gentle} that for all $\epsilon\in (0,1)$, it holds that
\begin{align*}
   \sup_{Q} N\left(2\epsilon,\sumn\sG_i,L_2(Q)\right) \leq \left(\frac{K}{\epsilon}\right)^{2dn+4n}. 
\end{align*}
It follows from \citep{Talagrand1994,sen2018gentle}, that there exists a universal constant $K'>0$ such that for all $t>0$, it holds that
\begin{align*}
    \Pr\left( \sqrt{M}W_1>t\right)&\leq \left(\frac{K't}{\sqrt{2dn+4n}}\right)^{2dn+4n}e^{-2t^2}.
\end{align*}
However, note that $|W_1|\leq 1/4$, therefore, for $t>\sqrt{M}/4$, it holds that $\Prr{\sqrt{M}W_1>t}=0$. 
Therefore, it holds that $$\left(\frac{K't}{\sqrt{2dn+4n}}\right)^{2dn+4n}\leq \left(\frac{K'\sqrt{M}}{4\sqrt{2dn+4n}}\right)^{2dn+4n},$$
with probability 1, and so,
\begin{align*}
    \Pr\left( W_1>t\right)&\leq \left(\frac{K'M}{4dn}\right)^{dn}e^{-2Mt^2}.
\end{align*}
The same logic applied to $W_2$ gives that there exists a universal constant $K''>0$ such that for all $t>0$, it holds that
\begin{align*}
    \Pr\left( W_2>t\right)&\leq \left(\frac{K''M}{dn}\right)^{dn}e^{-2Mt^2}.
\end{align*}
Therefore, there exists universal constants $K,K'>0$ such that for all $t>0$, it holds that
\begin{align*}
 \Prr{Z>t}&\leq \left(\frac{KM}{dn}\right)^{dn}e^{-K'Mt^2}. 
\end{align*}
Now, to prove the sample complexity bound, we want to find $M$ such that 
$$\left(\frac{KM}{dn}\right)^{dn}e^{-K'Mt^2}\leq \gamma,$$
which is equivalent to the inequality
\begin{equation}\label{eqn::M_step1}
    -dn\log (M/nd)+K'Mt^2\geq \log(1/\gamma)+dn\log\left(K\right).
\end{equation}
First, we show that $dn\log (M/nd)\leq K'Mt^2/2.$
Note that the function $h(M)=K'Mt^2/2-dn\log (M/nd)$ satisfies the conditions of Lemma \ref*{lem::deriv}. 
Applying Lemma \ref*{lem::deriv} yields that $h(M)\geq 0$ for 
\begin{equation}\label{eqn::M-1}
    M\geq 2\frac{2dn}{K't^2}\log\left(\frac{2}{K't^2}\vee e\right)\gtrsim \frac{dn}{t^2}\log\left(\frac{1}{t}\vee e\right).
\end{equation}
When \eqref{eqn::M-1} is satisfied, we have that \eqref{eqn::M_step1} is satisfied for 
\begin{equation}\label{eqn::M_step2}
   M\geq 2\frac{\log(1/\gamma)+dn\log K}{ K't^2}.
\end{equation}
Thus, for all $t>0$ and all $0<\gamma<1$, there exists a universal constant $c_1>0$ such that $$\sup_x\left|\widehat{\IDD} (x,\hat{\mu}_n,s)-\IDD(x,\hat{\mu}_n,s)\right|\leq t,$$
with probability $1-\gamma$, provided that 
\begin{equation*}
   M\geq c_1\frac{\log(1/\gamma)\vee dn\log\left( \frac{1}{t}\vee e\right)}{ t^2}.
\end{equation*}
% \qedhere
\end{proof}
We prove Theorem~\ref*{thm::depth_cond} with a series of Lemmas. 
\begin{lemma}\label{lem::c1_of_depth}
Condition \ref*{cond::phi_um} holds for each of the depth functions defined in Section~\ref*{sec::depth}. 
\end{lemma}
\begin{proof}
Recall that all of the depth functions satisfy property (4) given in Definition \ref*{def::depth}. 
Property (4) implies that there exists a compact set $E$ such that $\sup_{x\in\rdd}\D(x,\mu)=\sup_{x\in E}\D(x,\mu)$. 
It follows from boundedness of  $\D(x,\mu)$ and compactness of $E$ that there exists a point in $y\in E$ such that $\D(y,\mu)=\sup_{x\in E}\D(x,\mu)=\sup_{x\in\rdd}\D(x,\mu)$. 
\end{proof}
\begin{lemma}\label{lem::lip_continuity_of_depth_1}
Suppose that $\pi$ is absolutely continuous and that for all $u\in \bS^{d-1}$, the measure $\mu_u$ is absolutely continuous with density $f_u$ and $\sup_{\norm{u}=1}\norm{f_u}_\infty=L  <\infty$. 
Then there exists a universal constant $C>0$ such that for all $s\in (0,\infty]$, halfspace, simplicial, integrated rank-weighted and smoothed integrated dual depth satisfy Condition \ref*{cond::phi_uc} with $\omega(r)=C\cdot L\cdot r$ . 
\end{lemma}
\begin{proof}
% Let $L=\sup_{\norm{u}=1}\norm{f_u}_\infty$. 
We prove the result for each depth in turn.
\hfill\newline
\hfill\newline
\textbf{Halfspace depth:} 
Consider two points $x,y\in \rdd$. 
By assumption, $F(\cdot,u,\mu)$ is $L$-Lipschitz function a.e. .
% (it has a bounded derivative $\mu_u$ a.e.)
It follows that 
\begin{align*}
|\HD(x,\mu)-\HD(y;\mu)|&=|\inf_u F(x,u,\mu)- \inf_u F(y,u,\mu)| \\
&\leq  \sup_u | F(x,u,\mu)-F(y,u,\mu)|\\
&\leq  L\cdot \norm{x-y}\ a.e.\ .
\end{align*}
\hfill\newline
\textbf{Integrated rank-weighted depth:} For integrated rank-weighted depth, note that absolute continuity of $\mu_u$ implies that $F(x,u,\mu)=F(x-,u,\mu)$, where one recalls that $F(x-,u,\mu)=\mu(\{X^\top u< x^\top u\})$. 
With this in mind, for any two points $x,y\in \rdd$ the reverse triangle inequality yields that
\begin{align*}
|\IRW(x,\mu)-\IRW(y,\mu)|&\leq \int_{\bS^{d-1}} 2|F(x,u,\mu)-F(y,u,\mu)|d\nu(u)\leq  2L\cdot  \norm{x-y}\ a.e.\ .
\end{align*}
\hfill\newline
\textbf{Integrated dual depth:} For any $h:S\rightarrow [0,1]$, $h':S\rightarrow [0,1]$, for any $x\in S$ it holds that 
\begin{align}\label{eqn::3bd}
    |h(x)(1-h(x))-h'(x)(1-h'(x))|&\leq 3|h(x)-h'(x)|. 
\end{align}
Using this inequality, 
\begin{align}\label{eqn::iddlip1}
|\IDD(x,\mu,s)-\IDD(y,\mu,s)|&\leq 3\sup_u\left| \E{\mu}{\expt(s(x-X)^\top u)}-\E{\mu}{\expt(s(y-X)^\top u)}\right|.
\end{align}
It remains to bound $\left| \E{\mu}{\expt(s(x-X)^\top u)}-\E{\mu}{\expt(s(y-X)^\top u)}\right|$. 
To this end, first assume that $s\leq \infty$ and let $w,z\in \re$. 
Using the fact that $\sigma$ is a $1/4$-Lipschitz function, it holds that 
\begin{align*}
\left| \E{\mu_u}{\expt(s(w-X))}-\E{\mu_u}{\expt(s(z-X))}\right|&= \left|\int \expt(s(w-t))-\expt(s(z-t))f_u(t) dt\right|\\
&= \left|\frac{1}{s}\int \expt(v)-\expt(sz-sw+v)f_u(-v/s-w) dv\right|\\
&\leq|z-w|/4\int  \left|f_u(-v/s-w)\right| dv\\
&\leq |z-w|\norm{f_u}/4,
\end{align*}
where in the second line we use the substitution $v=s(w-t)$. 
The preceding inequality and \eqref{eqn::iddlip1} imply that
\begin{align*}
    |\IDD(x,\mu,s)-\IDD(y,\mu,s)|&\leq 3\sup_u\left| \E{\mu}{\expt(s(x-X)^\top u)}-\E{\mu}{\expt(s(y-X)^\top u)}\right|\\
    &\leq \frac{3}{4}L\norm{x-y}.
\end{align*}
Now assume that $s=\infty$. 
It follows from absolute continuity of $\mu_u$ and \eqref{eqn::iddlip1} that
\begin{align*}
|\IDD(x,\mu)-\IDD(y,\mu)|\leq 3\sup_u\left|F(x,u,\mu)-F(y,u,\mu)\right|\leq 3L\norm{x-y}\ a.e.
\end{align*}
\hfill\newline
\textbf{Simplicial depth:} For simplicial depth, we must show that $\Pr(x\in \Delta(X_1,\ldots,X_{d+1})$ is Lipschitz continuous. 
It is easy to begin with two dimensions. 
Consider $\Pr(x\in \Delta(X_1,X_2,X_{3}))-\Pr(y\in \Delta(X_1,X_2,X_{3}))$, as per \citep{liu1990}, we need to show that $\Pr(\overline{X_1X_2}\text{ intersects }\overline{xy})\leq L\cdot  \norm{x-y}$. 
In order for this event to occur, we must have that $X_1$ is above $\overline{xy}$ and $X_2$ is below $\overline{xy}$, but both are projected onto the line segment $\overline{xy}$ when projected onto the line running through $\overline{xy}$. 
Affine invariance of simplicial depth implies we can assume, without loss of generality, that $x$ and $y$ lie on the axis of the first coordinate. 
Let $x_1$ and $y_1$ be the first coordinates of $x$ and $y$. Suppose that $X_{11}$ is the first coordinate of $X_1$. 
It then follows from a.e. Lipschitz continuity of $F(x,u,\mu)$ that 
$$\Pr(\overline{X_1X_2}\text{ intersects }\overline{xy})\leq \Pr(x_1<X_{11}<y_1)\leq L\cdot  |x_1-y_1|\leq L\norm{x-y}\ a.e.\ .$$
In dimensions greater than two, a similar line of reasoning can be used. 
We can again assume, without loss of generality, that $x$ and $y$ lie on the axis of the first coordinate. 
It holds that $$\Pr(x\in \Delta(X_1,X_2,X_{3}))-\Pr(y\in \Delta(X_1,X_2,X_{3}))\leq \binom{d+1}{d}\Pr(A_d),$$ 
where $A_d$ is the event that the $d-1$-dimensional face of the random simplex, formed by $d$ points randomly drawn from $F$, intersects the line segment $\overline{xy}$. 
It is easy to see that 
\begin{align*}
    \Pr(A_d)\leq \Pr(x_1<X_{11}<y_1)\leq \norm{f_u}_\infty |x_1-y_1|\leq L\cdot  \norm{x-y}\ a.e.\ .
\end{align*}
 % \tag*{\qedhere}
 This completes the proof.
\end{proof}
%%%%%%%%%%%%%%%%%%%%
%%%%%%%%%%%%%%%%%%%%
\begin{lemma}\label{lem::lip_continuity_of_depth_2}
Suppose that $\mu$ has a bounded density and that $\sup_y\E{}{\norm{y-X}^{-1}}=L'\leq \infty$. Then the modified spatial depth satisfies Condition \ref*{cond::phi_uc} with $\omega(r)=2\sqrt{d}L'\cdot r$ and spatial depth satisfies Condition \ref*{cond::phi_uc} with $\omega(r)=2 L'\cdot r$. 
\end{lemma}
%%%%%%%%%%%%%%%%%%%%%%%%%%%%%%%%%%%%%%%%
\begin{proof}
Taking the derivative of $\MSD$ and using the assumption of a bounded density results in 
\begin{align*}
   \frac{d\MSD(y,\mu)}{dy}=- 2\Ee{\mu}{\sff(y-X)}\E{\mu}{\left[\frac{1_{d}}{\norm{y-X}}-\frac{(y-X)*(y-X)^\top 1_{d}}{\norm{y-X}^3}\right]}.
\end{align*}
Now,
\begin{align*}
    || \frac{d\MSD(y,\mu)}{dy}||\leq 2\sqrt{d}\sup_y\E{}{\norm{y-X}^{-1}}=2\sqrt{d}L'.
\end{align*}
Thus, the map $y\mapsto \MSD(y,\mu)$ is $2\sqrt{d}L'$-Lipschitz.

Define the spatial rank function $h(y,\mu)=\Ee{\mu}{\sff(y-X)}.$
Proposition~4.1 of \citep{Koltchinskii1997} says that if $\mu$ has a bounded density, then the map $y\mapsto h(y,\mu)$ is continuously differentiable in $\rdd$ with derivative 
\begin{align*}
    \frac{dh}{dy}=\int_{\{x \neq y\}} \frac{1}{\norm{y-x}}\left[I_d-\frac{(y-x)(y-x)^\top}{\norm{y-x}^2}\right] d\mu(x).
\end{align*}
Now,
\begin{align*}
    ||\frac{dh}{dy}||\leq 2\sup_y\E{}{\norm{y-X}^{-1}}=2L',
\end{align*}
where $\norm{\cdot}$ denotes the operator norm. 
Thus, the map $y\mapsto h(y,\mu)$ is $2L'$-Lipschitz, which implies that the map $y\mapsto \SD(y,\mu)$ is also $2L'$-Lipschitz.
\end{proof}
%%%%%%%%%%%%%%%%%%%%%%%%%%%%%%%%%
%%%%%%%%%%%%%%%%%%%%%%%%%%%%%%%%%
\begin{lemma}\label{lem::depth_kf}
Condition \ref*{cond::phi-k-reg} holds for each of the depth functions defined in Section~\ref*{sec::depth}. 
\end{lemma}
%%%%%%%%%%%%%%%%%%%%%%%%%%%%%%%%%%%%%%%%%%%%%%%%%%%%
%Don't delete this please
% Fix $x$ and assume without loss of generality that $F_{2,u}(x)\geq F_{1,u}(x)$. 
% Recall that there exists a minimizing sequence such that $\inf_u F_{1,u}(x)\leq F_{1,u}(y_k)\leq \inf_u F_{1,u}(y_k)+1/k$. 
% With this in mind, it follows that for all $k$,
% \begin{align*}
%     \inf_u F_{2,u}(x)-\inf_u F_{1,u}(x) & \leq  F_{2,u_k}(x)-F_{1,u_k}(x)+1/k\leq \sup_u|F_{2,u}(x)-F_{1,u}(x)|+1/k.
% \end{align*}
% Sending $k$ to infinity gives that $\sup_x|\inf_u F_{2,u}(x)-\inf_u F_{1,u}(x)|\leq \sup_{x,u}|F_{2,u}(x)-F_{1,u}(x)|$. 
\begin{proof}
We prove the result for each depth in turn.
%%%%%%%% Halfspace depth %%%%%%%%
\hfill\newline
\hfill\newline
\textbf{Halfspace depth:} Let 
$$\sF=\left\{\ind{X^\top u \leq y}\colon\ u\in \bS^{d-1},\ y\in \re\right\}$$
and recall that the set of subgraphs of $\sF$ is the set of closed halfspaces in $\rdd$. 
Furthermore, it holds that
\begin{align*}
   \sup_x|\HD(x,\mu_1)-\HD(x,\mu_2)|&\leq \sup_{g\in \sF}|\E{\mu_1}{g(X)}-\E{\mu_2}{g(X)}|.
\end{align*}
% \begin{align*}
%   |\HD(x,\mu_1)-\HD(x,\mu_2)|= |\inf_u F(x,u,\mu_1)-\inf_u F(x,u,\mu_2)|&\leq \sup_{u}|F(x,u,\mu_2)-F(x,u,\mu_2)|.
% %   \\
% %   &\leq \sup_{y\in \re,u\in \bS^{d-1}}|\E{\mu_1}{\ind{X^\top u \leq y }}-\E{\mu_2}{\ind{X^\top u \leq y}}|,
% \end{align*}
% Recall that $F(x,u,\mu_2)=\E{\mu_1}{\ind{X^\top u \leq y }}$ and define 
Thus, it follows $\HD$ is $(1,\sF)$-regular with $\VC(\sF)=d+2$. 
%%%%%%%% IRW depth %%%%%%%%
\hfill\newline
\hfill\newline
\textbf{Integrated rank-weighted-depth:} First observe that 
\begin{multline*}
  \sup_{x}  |\IRW(x,\mu_1)-\IRW(x,\mu_2)|\leq  4\sup_{u,x}   [\left|F(x,u,\mu_1 )-F(x,u,\mu_2)\right|\\
  \vee\left|F(x-,u,\mu_1 )-F(x-,u,\mu_2)\right|].
\end{multline*}
% \begin{align*}
%   \sup_{x}  |\IRW(x,\mu_1)-\IRW(x,\mu_2)|\leq  4\sup_{u,x}   \left|F(x,u,\mu_1 )-F(x,u,\mu_2)\right|\vee\left|F(x-,u,\mu_1 )-F(x-,u,\mu_2)\right|.
% \end{align*}
Define 
\begin{align*}
    \sG&=\left\{g\colon g(X)=\ind{X^\top u <y},\ u\in \bS^{d-1},\ y\in \re\right\}.
\end{align*}
We see that the integrated rank-weighted depth function is $(4,\sG\vee \sF)$-regular. Furthermore, since the subgraphs of $\sG$ are contained in those of $\sF$, we have that $\VC(\sG\vee \sF)=d+2$.
%%%%%%%% s-IDD depth %%%%%%%%
\hfill\newline
\hfill\newline
\textbf{Integrated dual depth:} 
If $s=\infty$ then it follows immediately from \eqref{eqn::3bd} and the analysis of halfspace depth that the integrated dual depth is $(3,\sF)$-regular with $\VC(\sF)=d+2$. 
Suppose that $s<\infty$. 
It follows from \eqref{eqn::3bd} that 
\begin{align*}
    \sup_{x}|\IDD(x,\mu_{2})-\IDD(x,\mu_{1})|\leq 3\sup_{y\in \re,u\in \bS^{d-1}}|\E{\mu_1}{\expt(s(x-X)^\top u)}-\E{\mu_2}{\expt(s(x-X)^\top u)}|.
\end{align*}
We have that 
\begin{align*}
  \sup_{y\in \re,u\in \bS^{d-1}}|\E{\mu_1}{\expt(s(x-X)^\top u)}-\E{\mu_2}{\expt(s(x-X)^\top u)}|\leq   \sup_{g\in \sF'}|\E{\mu_1}{g(X)}-\E{\mu_2}{g(X)}|,
\end{align*}
where
\begin{align*}
    \sF'&= \left\{\expt(AX_i+b)\colon A\in\rdd,b\in \re\right\}.
\end{align*}
% \begin{align*}
%     \sF'&=\left\{\expt(s(x-X_i)^\top u)\colon x\in\rdd,u\in \bS^{d-1}\right\}\subset \left\{\expt(AX_i+b)\colon A\in\rdd,b\in \re\right\}\coloneqq \sF^*.
% \end{align*}
The class of functions $\sF'$ is a constructed from a monotone function applied to a finite-dimensional vector space of measurable functions. 
%See Lemma 7.14 on page 67 of emp process notes and Lemma 7.17 on page 68
It follows from Lemma 7.15 and Lemma 7.19 of \citet{sen2018gentle} that $\VC(\sF')=d+2$. 
Therefore, the smoothed integrated dual depth is $(3,\sF')$-regular with $\VC(\sF')=d+2$. 
%%%%%%%% Simplicial depth %%%%%%%%
\hfill\newline
\hfill\newline
\textbf{Simplicial Depth:} \citet{DUMBGEN1992119} gives that 
$$\sup_x|\SMD(x,\mu_1)-\SMD(x,\mu_2)|\leq (d+1)\sup_{A\in\scA}|\mu_1(A)-\mu_2(A)|,$$
where $\mathscr{A}$ is the set of intersections of $d$ open half-spaces in $\rdd$. 
Define the set of Boolean functions $\scA'=\{g(x)=\ind{x\in A}\colon A\in\scA\}$. 
Thus, simplicial depth is $(d+1,\scA')$-regular with $\VC(\scA')=O(d^2\log d)$.
%%%%%%%% Spatial depth %%%%%%%%
\hfill\newline
\hfill\newline
\textbf{Spatial Depth:} 
Applying the reverse triangle inequality yields
\begin{align*}
    \sup_x\left|\ \norm{\E{\mu_1}{\frac{x-X}{\norm{x-X}}}}-\norm{\E{\mu_2}{\frac{x-X}{\norm{x-X}}}}\ \right|&\leq   \sup_x\norm{\E{\mu_1}{\frac{x-X}{\norm{x-X}}}-\E{\mu_2}{\frac{x-X}{\norm{x-X}}}}\\
    % &\leq \sup_{x}\norm{\E{\mu_1}{\frac{x-X}{\norm{x-X}}}-\E{\mu_2}{\frac{x-X}{\norm{x-X}}}}\\
     &\leq  \sup_{x}\sum_{j=1}^d\left|\E{\mu_1}{\frac{x_j-X_j}{\norm{x-X}}}-\E{\mu_2}{\frac{x_j-X_j}{\norm{x-X}}}\right|\\
     &\leq d\sup_{x,1\leq j\leq d}\left|\E{\mu_1}{\frac{x_j-X_j}{\norm{x-X}}}-\E{\mu_2}{\frac{x_j-X_j}{\norm{x-X}}}\right|.
\end{align*}
It remains to compute the VC-dimension of 
$$\sG=\left\{g(X)=\frac{x_j-X_j}{\norm{x-X}}\colon x\in\rdd,\ j\in[d]\right\}.$$
Now, let $E$ be the set of standard basis vectors for $\rdd$. 
We can then write 
$$\sG=\left\{g(X)=\frac{e^\top(x-X)}{\norm{x-X}}\colon x\in\rdd,\ e\in E\right\}.$$
It is now easy to see that 
$$\sG\subset\left\{g(X)=y^{\top}(x-X)\colon x\in\rdd,\ y\in\re^d\ j\in[d]\right\}.$$
It follows that $\VC(\sG)\leq O(d)$ and that spatial depth is $(d,\sG)$ regular. 
% We may instead write
% \begin{align*}
%     \sup_x|\SD(X,\mu_1)-\SD(X,\mu_2)|&= \sup_x\left|\ \norm{\Ee{\mu_1}{\frac{x-X}{\norm{x-X}}}}-\norm{\Ee{\mu_2}{\frac{x-X}{\norm{x-X}}}}\ \right|\\
%     &=\sup_x\left|\ \sup_u\left[\Ee{\mu_1}{\frac{x-X}{\norm{x-X}}}^\top u\right]^{1/2}-\sup_u\left[\Ee{\mu_2}{\frac{x-X}{\norm{x-X}}}^\top u\right]^{1/2}\ \right|\\
%     &\leq  \sup_u\sup_x\left|\ \sqrt{\Ee{\mu_1}{\frac{x-X}{\norm{x-X}}}^\top u}-\sqrt{\Ee{\mu_2}{\frac{x-X}{\norm{x-X}}}^\top u}\ \right|\\
%     &\leq  \sup_u\sup_x\sqrt{\left|\ \Ee{\mu_1}{\frac{x-X}{\norm{x-X}}^\top u}-\Ee{\mu_2}{\frac{x-X}{\norm{x-X}}^\top u}\ \right|}
% \end{align*}
\hfill\newline
\hfill\newline
\textbf{Modified Spatial Depth:} 
We may write
{\small
\begin{align*}
    \sup_x|\SD(X,\mu_1)-\SD(X,\mu_2)|&= \sup_x\left|\ \norm{\Ee{\mu_1}{\frac{x-X}{\norm{x-X}}}}^2-\norm{\Ee{\mu_2}{\frac{x-X}{\norm{x-X}}}}^2\ \right|\\
    &=\sup_x\left|\ \sup_u\left[\Ee{\mu_1}{\frac{x-X}{\norm{x-X}}}^\top u\right]-\sup_u\left[\Ee{\mu_2}{\frac{x-X}{\norm{x-X}}}^\top u\right]\ \right|\\
    &\leq  \sup_u\sup_x\left|\ \Ee{\mu_1}{\frac{x-X}{\norm{x-X}}}^\top u-\Ee{\mu_2}{\frac{x-X}{\norm{x-X}}}^\top u\ \right|\\
    &\leq  \sup_u\sup_x\left|\ \Ee{\mu_1}{\frac{x-X}{\norm{x-X}}^\top u}-\Ee{\mu_2}{\frac{x-X}{\norm{x-X}}^\top u}\ \right|.
\end{align*}}
It remains to compute the VC-dimension of 
$$\sG=\left\{g(X)=\frac{u^\top(x-X)}{\norm{x-X}}\colon x\in\rdd,\ u\in \bS^{d-1}\right\}.$$
It is now easy to see that 
$$\sG\subset\left\{g(X)=y^{\top}X+b\colon b\in\re,\ y\in\re^d\ j\in[d]\right\}.$$
It follows that $\VC(\sG)\leq d+2$ and that the modified spatial depth is $(1,\sG)$ regular. 
\end{proof}
%%%%%%%%%%%%%%%%%%%%%%%%%%%%%%%%%%%%%%%%%%%%
\begin{proof}[Proof of Theorem~\ref*{thm::depth_cond}]
The result follows from Lemmas \ref*{lem::c1_of_depth}-\ref*{lem::depth_kf}. 
\end{proof}
\begin{proof}[Proof of Corollary~\ref*{thm::gshs}]
Together, Lemma \ref*{lem::GS} and the definition of the Laplace mechanism imply that $\widetilde{\D}(x,\hat{\mu}_n)$ is $\epsilon$-differentially private. 
Furthermore, the properties of the Laplace measure give that 
\begin{align*}
    \Pr\left(\left|\widetilde{\D}(x,\hat{\mu}_n)-\D(x,\mu)\right|>t\right)&\leq \Pr\left(\left|\D(x,\hat{\mu}_n)-\D(x,\mu)\right|>t/2\right)+\Prr{|W_1\frac{K}{n\epsilon}|>t/2}\\
    &\leq \Pr\left(\left|\D(x,\hat{\mu}_n)-\D(x,\mu)\right|>t/2\right)+e^{-n\epsilon t/2K}.
\end{align*}
Applying \eqrefplain{eqn::ob_con} yields
\begin{align*}
    \Pr\left(\left|\widetilde{\D}(x,\hat{\mu}_n)-\D(x,\mu)\right|>t\right)&\leq e^{\VC(\sF)\log\left(c_1\frac{n}{\VC(\sF)}\right)-c_2n t^2}+e^{-n\epsilon t/2K}.
\end{align*}
\end{proof}

\end{document}